\documentclass{article}
\usepackage[utf8]{inputenc}
\usepackage{amsmath,amssymb}
\usepackage{algorithmic}
\usepackage[ruled,commentsnumbered,linesnumbered,vlined]{algorithm2e}
\usepackage{xcolor}
\usepackage{indentfirst}
\usepackage{comment}
\usepackage{longtable}
\usepackage{rotating}

\usepackage{longtable}
\usepackage{amsthm}
\usepackage{subfigure}
\usepackage{picture}
\usepackage{graphicx}
\usepackage{hyperref}
\usepackage{float}

\usepackage{apacite}

\usepackage{pdflscape}
\usepackage{afterpage}
\usepackage{capt-of}
\usepackage{fullpage}

\newtheorem{theo}{Theorem}
\newtheorem{prop}[theo]{Proposition}

\newtheorem{lemma}[theo]{Lemma}

\DeclareMathAlphabet\mathbfcal{OMS}{cmsy}{b}{n}

\newcommand{\rafaelC}[1]{\textcolor{black}{#1}}

\title{Valid inequalities, preprocessing, and an effective heuristic for the uncapacitated three-level lot-sizing and replenishment problem with a distribution structure}

\author{ 
Jesus O. Cunha {\thanks{Universidade Federal do Ceará, Departamento de Estatística e Matemática Aplicada, Fortaleza, Brazil.  ({\tt jesus.ossian@dema.ufc.br})}}
    \and
    Rafael A. Melo {\thanks{Universidade Federal da Bahia, Departamento de Ci\^{e}ncia da Computa\c{c}\~{a}o, Computational Intelligence and Optimization Research Lab (CInO), Salvador, Brazil.  ({\tt melo@dcc.ufba.br}) }}
}

\begin{document}

\maketitle

\begin{abstract}

We consider the uncapacitated three-level lot-sizing and replenishment problem with a distribution structure. In this problem, a single production plant sends the produced items to replenish warehouses from where they are dispatched to the retailers in order to satisfy their demands over a finite planning horizon. Transfers between warehouses or retailers are not permitted, each retailer has a single predefined warehouse from which it receives its items, and there is no restriction on the amount that can be produced or transported in a given period. The goal of the problem is to determine an integrated production and distribution plan minimizing the total costs, which comprehends fixed production and transportation setup as well as variable inventory holding costs.
We describe new valid inequalities both in the space of a standard mixed integer programming (MIP) formulation and in that of a new alternative extended MIP formulation. We show that using such extended formulation, valid inequalities having similar structures to those in the standard one allow achieving tighter linear relaxation bounds. Furthermore, we propose a preprocessing approach to reduce the size of an extended multi-commodity MIP formulation available in the literature. Such preprocessing relies on the removal of variables based on the problem's cost structure while preserving optimality guarantees. We also propose a multi-start randomized bottom-up dynamic programming-based heuristic. The heuristic employs greedy randomization via changes in certain costs and solves subproblems related to each level using dynamic programming.  
Computational experiments indicate that the use of the valid inequalities in a branch-and-cut approach significantly increase the ability of a MIP solver to solve instances to optimality. Additionally, the valid inequalities for the new alternative extended formulation outperform those for the standard one in terms of number of solved instances, running time and number of enumerated nodes. 
Moreover, the proposed heuristic is able to generate solutions with considerably low optimality gaps within very short computational times even for large instances. Combining the preprocessing approach with the heuristic, one can achieve an increase in the number of solutions solved to optimality within the time limit together with significant reductions on the average times for solving them.
\\



\end{abstract}

\section{Introduction}
\label{sec:introduction}

Supply chain optimization has become a crucial and challenging activity in nowadays competitive industrial and business environments. Very often, integrated supply chain decisions regarding production, storage, and transportation are made in situations in which the production plants, warehouses, and clients requiring the produced items are located in different geographical areas. In this direction, we study the uncapacitated three-level lot-sizing and replenishment problem with a distribution structure (3LSPD-U), which was introduced in \citeA{GruBazCorJan19}. In this three-echelon problem, a single production plant (level 0) must produce items to replenish multiple warehouses (level 1), from where these items are dispatched to multiple retailers (level 2) in order to attend their deterministic dynamic demands over a finite planning horizon. The goal consists of determining an integrated production and distribution plan minimizing the total costs, which comprise fixed production and transportation setups as well as variable inventory holding costs.

Multi-level production planning problems, especially those related to two and three-echelon supply chains, have been studied in several works.
\citeA{PocWol91} studied a multi-level lot-sizing problem and described valid inequalities for the problem, including extensions of the well-known $(l,S)$-inequalities for the uncapacitated lot-sizing~\cite{BarWol84}.
\citeA{MelWol10} considered the polynomially solvable uncapacitated two-level lot-sizing and applied the extended formulation resulting from a dynamic programming algorithm to an NP(nondeterministic polynomial-time)-hard production and distribution extension of the problem.
\citeA{AkaMil12} performed a computational analysis of the lower bounds achieved by several approaches for capacitated multi-level lot-sizing problems.
\citeA{ZhaKucYam12} studied a multi-echelon uncapacitated lot-sizing problem with intermediate demands. The authors proposed, for the two-echelon case, a polynomial-time dynamic programming algorithm and a family of valid inequalities together with a polynomial-time separation algorithm. The computational experiments showed that extended formulations can be used to successfully solve an uncapacitated multi-item two-echelon problem and that a {branch-and-cut} algorithm using the proposed inequalities is very effective for dealing with a more general capacitated multi-item multi-echelon problem.

\citeA{SolSur12} considered the one-warehouse multi-retailer problem (OWMR) in which a single warehouse replenishes multiple retailers with deterministic dynamic demands over a discrete planning horizon. The authors proposed several mixed integer programming (MIP) formulations for the problem and compared them both computationally and in terms of the provided linear relaxation bounds. 
\citeA{CunMel16} extended the study of \citeA{SolSur12} for the OWMR. They considered additional formulations and valid inequalities for the problem, which were also compared theoretically and computationally. 
\citeA{Par05} studied the integrated production and distribution planning in a two-level multi-plant, multi-retailer, and multi-item logistic environment. The author proposed a heuristic approach and showed that integrated planning could achieve substantial advantages over decoupled planning.
\citeA{MelWol12} studied a two-level supply chain with multiple items, production sites, and client areas and a discrete-time horizon. The authors developed a hybrid heuristic that uses a strong formulation to provide a good dual bound and suggest certain variables fixing, and a fix-and-optimize approach achieved by fixing variables in a standard formulation to provide the heuristic solution. They showed for different classes of medium-sized instances that the hybrid heuristic provides solutions with a guaranteed quality that are as good as or better than those provided by a commercial MIP solver running for a considerably larger time.
\citeA{HelSah10} proposed a fix-and-optimize heuristic for a capacitated multi-level lot-sizing problem with lead times.

Several authors considered the optimization of three-level supply chains.
\citeA{KadCheShaRam12} considered a multi-product three-level supply chain problem composed of vendors, production plants, and distribution centers and proposed variants of a particle swarm optimization metaheuristic. \citeA{CarTre14} tackled the problem studied in \citeA{KadCheShaRam12} using a MIP formulation and showed that the instances available in the literature could be solved to optimality within short computational times using a commercial MIP solver.
\citeA{ZhaSon18} developed a decision support system based on operations research tools and techniques for production and distribution planning at Danone Waters China Division. The authors proposed a MIP formulation and provided customizable options for company managers. Computational experiments indicated that the proposed approach could significantly increase efficiency and reduce total costs.
\citeA{GruBazCorJan19} studied uncapacitated and capacitated variants of the three-level lot-sizing and replenishment problem with a distribution structure (3LSPD-U and 3LSPD-C). The 3LSPD-U, which is the problem we consider in our work, is already NP-hard as it generalizes the one-warehouse multi-retailer problem. The authors compared several different MIP formulations to solve the problem, including extensions of those for the OWMR~\cite{SolSur12,CunMel16}, and performed extensive computational experiments. These experiments showed that a multi-commodity formulation outperformed all others for uncapacitated instances and echelon stock reformulations achieved the best results for the capacitated ones. 

Besides, extensions of multi-echelon supply chain optimization problems which also take routing decisions into consideration appeared in various works, which corroborates the importance of effective approaches for solving multi-echelon production planning problems.
\citeA{LiChuChen11} studied an infinite horizon inventory routing problem in a three-level distribution system and proposed a decomposition solution approach based on a fixed partition policy in which retailers are divided into disjoint sets and served by separate routes. Efficient algorithms were given for the sub-problems by exploring properties of their optimal solutions and a genetic algorithm was proposed to find near-optimal fixed partitions. \citeA{GuiCoeSchSca19} applied the well-known vendor-managed inventory paradigm extended to a two-echelon supply chain. They proposed a mathematical formulation, a {branch-and-cut} algorithm, and a matheuristic to tackle the problem. 
\citeA{SarBahSupSya19} considered a location inventory routing problem in a three-echelon supply chain system and proposed a two-stage heuristic. The heuristic was applied to a real case study and it was compared to a mixed integer nonlinear programming (MINLP) formulation. 
\citeA{AbdShaNdi19} studied a demand planning problem in a four-level petrochemical supply chain that integrates lot-sizing, scheduling with sequence-dependent transient costs, transportation, and warehousing decisions between suppliers, plants, warehouses, and customers. They proposed a MIP formulation and an iterative three-stage heuristic to provide quality solutions within acceptable computational times.
A literature review regarding production routing problems can be encountered in \citeA{AduCorJan15}.

\subsection{Main contributions and organization}

The main contributions of our work can be summarized as follows. Firstly, we describe new valid inequalities for the 3LSPD-U both in the space of a standard formulation and in the space of a new alternative extended formulation whose size is asymptotically equivalent to the standard one. Additionally, we provide a theoretical comparison regarding the strength of the linear relaxations of these two approaches using valid inequalities. {More specifically, we show that one can obtain stronger bounds when using valid inequalities for the new alternative extended formulation which have structures that are somehow related to those for the standard one.} 
Secondly, we present a new preprocessing technique to reduce the size of a multi-commodity formulation, which was shown in \citeA{GruBazCorJan19} to be the most effective approach available for the 3LSPD-U from a computational viewpoint. The preprocessing approach allows the removal of variables based on the problem's cost structure while preserving optimality guarantees. 
Last but not least, we propose a multi-start randomized bottom-up {dynamic programming-based heuristic} for the problem. The heuristic follows a nonintegrated planning approach {in which subproblems related to each level are solved using dynamic programming} with randomness achieved by performing {controlled} changes in certain costs to allow diversification in the obtained solutions. {Namely, the approach uses a greedy randomization idea, which is widely used for combinatorial optimization problems, but with the difference that such greedy randomization is achieved by randomizing the costs rather than the choices during the execution of the algorithm.}
{
Computational experiments evidence that,
when implemented in a branch-and-cut procedure, the described valid inequalities allow an increase in the ability of a standard MIP solver to tackle the available benchmark instances when compared with the plain formulations. Besides, the performance of the branch-and-cut approach using the new alternative extended formulation outperforms that of the one using the standard formulation.
Additionally, the preprocessing approach allows significant reductions considering the candidate variables to be removed.
Furthermore, the multi-start randomized bottom-up dynamic programming-based heuristic can achieve low optimality gaps within very few seconds on average.
}

The remainder of this paper is organized as follows. 
Section~\ref{sec:problemdefinition} formally defines the uncapacitated three-level lot-sizing and replenishment problem\rafaelC{, presents a known standard mixed integer program, and describes a multi-commodity formulation available in the literature for the problem.}
Section~\ref{sec:validinequalities} characterizes new valid inequalities for the problem, both in the space of the standard formulation as in the space of {a new} alternative extended formulation. Section~\ref{sec:preprocessing} presents a new preprocessing approach to reduce the size of the multi-commodity formulation based on the problem's cost structure. Section~\ref{sec:randomizedheuristic} proposes a multi-start randomized bottom-up {dynamic programming-based heuristic} for the problem. Section~\ref{sec:experiments} summarizes the performed computational experiments. 
Concluding remarks are discussed in Section~\ref{sec:concludingremarks}.

\section{\rafaelC{Problem definition and mixed integer programming formulations}}
\label{sec:problemdefinition}

The uncapacitated three-level lot-sizing and replenishment problem with a distribution structure (3LSPD-U) can be formally defined as follows. There is a set $F = P \cup W \cup R$ of facilities with a singleton $P=\{p\}$ containing the plant (level 0), a set $W$ of warehouses (level 1) and a set $R$ of retailers (level 2). The time horizon is defined as $T = \{1,\ldots,NT\}$. Each warehouse $w\in W$ attends a predefined set of retailers $\delta(w)$ and each retailer $r\in R$ has a predefined unique associated warehouse $\delta_w(r) \in W$. Transportation can only happen between the production plant and the warehouses, and between a warehouse and the retailers it attends. For each facility $i\in F$, a fixed setup cost $sc^i_t$ is incurred whenever production/transportation occurs in period $t\in T$. Furthermore, a per-unit holding cost $hc^i_t$ is charged whenever items are held in inventory in facility $i$ at the end of period $t\in T$.
Each retailer $r\in R$ has its time-varying demands $d^r_t$ for each period $t\in T$. The problem consists of determining a production/transportation plan minimizing the total fixed setup and variable inventory costs.

Additionally, define the demands for facility $i\in P\cup W$ as
\begin{equation*}
    d^i_t = 
    \begin{cases}
     \displaystyle\sum_{r\in R}d^r_t,  & \textrm{if } i=p;\\
     \displaystyle\sum_{r\in \delta(i)}d^r_t, & \textrm{if } i\in W.
    \end{cases}
\end{equation*}
Moreover, for each facility $i\in F$, define $d^i_{kt} = \displaystyle\sum_{l=k}^t d^i_l $ as the cumulative demand from period $k$ up to $t$ for $1\leq k \leq t \leq {|T|}$.

\subsection{Standard mixed integer programming formulation}
\label{sec:stdform}

In order to formulate the 3LSPD-U as a mixed integer program, define variable $x^i_t$ to be the amount produced at the production plant in period $t\in T$ for $i=p$, and $x^i_t$ to be the amount transported to facility $i \in W\cup R$ from its predecessor in period $t\in T$. 
Let variable $s^i_t$ be the amount of stock in facility $i\in F$ at the end of period $t\in T$.
Besides, define the setup binary variable $y^i_t$ to be equal to one whenever $x^i_t > 0$ for $i\in F$ and $t\in T$, and to be equal to zero otherwise.
The problem can be formulated as the following standard mixed integer linear program~\cite{GruBazCorJan19}:
\begin{align}
z_{STD} = & \  \min \ \ \  \sum_{t \in T} \left( \sum_{i\in F} sc^i_t y^i_t + \sum_{i\in F} hc^i_t s^i_t \right)  \label{std-obj} & \\
(STD) \qquad & s^i_{t-1} + x^i_t = \sum_{j \in \delta(i)} x^j_t + s^i_t, \qquad  \textrm{for} \ i \in P\cup W, \ t\in T, \label{std-1} \\
&  s^r_{t-1} + x^r_t = d^r_t + s^r_t, \qquad  \textrm{for} \ r \in R, \ t\in T, \label{std-2} \\
&  x^i_t \leq d^i_{t{|T|}} y^i_t, \qquad  \textrm{for} \ i \in F, \ t\in T, \label{std-3} \\
&  x^i_{t}, \ s^i_t \geq 0, \qquad  \textrm{for} \ i \in F, \ t\in T, \label{std-4}\\
&  y^i_{t} \in  \{0,1\}, \qquad  \textrm{for} \ i \in F, \ t\in T. \label{std-5}
\end{align}
The objective function \eqref{std-obj} minimizes the total setup and inventory costs.
Constraints \eqref{std-1} define inventory balance constraints for the production plant and warehouses. Constraints \eqref{std-2} are inventory balance constraints for the retailers. Constraints \eqref{std-3} are setup enforcing constraints. Constraints \eqref{std-4} and \eqref{std-5} guarantee the nonnegativity and integrality requirements on the variables.
This formulation has $O({|R|} \times {|T|})$ variables and constraints.

Note that 3LSPD-U can be seen as an uncapacitated fixed-charge network flow (UFCNF) problem with a demand node for each pair $\{r\in R, t\in T\}$, with demand $d^r_t$. Such type of observation is well-known and widely used in the production planning literature~\cite{PocWol06}.

\subsection{\rafaelC{Best performing approach in the literature}}
\label{sec:bestliterature}

{In this section, we describe the multi-commodity \rafaelC{formulation for 3LSPD-U} proposed by \citeA{GruBazCorJan19}}.
Multi-commodity formulations~\cite{RarCho79} have been successfully applied for different production planning problems~\cite{AkaMil12,CunMel16Rem,CunMel16}.
Such formulation has shown to be very effective computationally and, furthermore, the best performing approach for solving the uncapacitated instances available in the literature~\cite{GruBazCorJan19}. We remark though that the multi-commodity formulation is not the one available which provides the tightest linear programming relaxation~\cite{GruBazCorJan19}.

Define variable $w^{0r}_{kt}$ to be the amount produced at the production plant in period $k\in T$ to satisfy $d^r_t$ for $r\in R$ and $t\in T$, $k\leq t$, variable $w^{1r}_{kt}$ to be the amount transported from the production plant to the warehouse of retailer $r\in R$ in period $k \in T$ to satisfy $d^r_t$ for $r\in R$ and $t\in T$, $k\leq t$, and variable $w^{2r}_{kt}$ to be the amount transported to retailer $r\in R$ from its corresponding warehouse in period $k\in T$ to satisfy $d^r_t$ for $r\in R$ and $t\in T$, $k\leq t$. Additionally, let variable $\sigma^{0r}_{kt}$ be the amount stocked at the production plant in the end of period $k\in T$ to satisfy $d^r_t$ for $r\in R$ and $t\in T$, $k < t$, variable $\sigma^{1r}_{kt}$ be the amount stocked at the warehouse of retailer $r\in R$ (i.e., $\delta_w(r)$) in the end of period $k\in T$ to satisfy $d^r_t$ for $r\in R$ and $t\in T$, $k < t$, and variable $\sigma^{2r}_{kt}$ be the amount stocked at the retailer $r\in R$ in the end of period $k \in T$ to satisfy $d^r_t$ for $r\in R$ and $t\in T$, $k < t$. Additionally, let $\lambda_{kt}$ be a constant equal to one if $k=t$ and zero otherwise.
The multi-commodity formulation can be defined as
\begin{align}
z_{MC} = & \  \min \ \ \  \sum_{t \in T}\left(  \sum_{i\in F} sc^i_t y^i_t + \sum_{r\in R}\sum_{k \leq t} hc^p_k \sigma^{0r}_{kt} + \sum_{r\in R}\sum_{k \leq t} hc^{\delta_w(r)}_k \sigma^{1r}_{kt} + \sum_{r\in R}\sum_{k \leq t} hc^r_k \sigma^{2r}_{kt} \right)  \label{mc-obj} & \\
(MC) \qquad & \sigma^{0r}_{k-1,t} + w^{0r}_{kt} = w^{1r}_{kt} + \sigma^{0r}_{kt}, \qquad  \textrm{for} \ r \in R, \ k\in T, \ t\in \{k,\ldots,{|T|}\}, \label{mc-1} \\
& \sigma^{1r}_{k-1,t} + w^{1r}_{kt} = w^{2r}_{kt} + \sigma^{1r}_{kt}, \qquad  \textrm{for} \ r \in R, \ k\in T, \ t\in \{k,\ldots,{|T|}\}, \label{mc-1b} \\
&  \sigma^{2r}_{k-1,t} + w^{2r}_{kt} = \lambda_{kt} d^r_t + (1-\lambda_{kt})\sigma^{2r}_{kt}, \qquad  \textrm{for} \ r \in R, \ k\in T, \ t\in \{k,\ldots,{|T|}\}, \label{mc-2} \\
&  w^{0r}_{kt} \leq d^r_{t} y^{p}_k, \qquad  \textrm{for} \ r \in R, \ k\in T, \ t\in \{k,\ldots,{|T|}\}, \label{mc-3} \\
&  w^{1r}_{kt} \leq d^r_{t} y^{\delta_w(r)}_k, \textrm{for} \ r \in R, \ k\in T, \ t\in \{k,\ldots,{|T|}\}, \label{mc-4} \\
&  w^{2r}_{kt} \leq d^r_{t} y^{r}_k, \textrm{for} \ r \in R, \ k\in T, \ t\in \{k,\ldots,{|T|}\}, \label{mc-5} \\
&  w^{0r}_{kt},\ w^{1r}_{kt},\ w^{2r}_{kt}, \ \sigma^{0r}_{kt}, \ \sigma^{1r}_{kt}, \ \sigma^{2r}_{kt} \geq 0, \qquad  \textrm{for} \ r \in R, \ k\in T, \ t\in \{k,\ldots,{|T|}\}, \label{mc-6}\\
&  y^i_{t} \in  \{0,1\}, \qquad  \textrm{for} \ i \in F, \ t\in T. \label{mc-7}
\end{align}
The objective function \eqref{mc-obj} minimizes the total setup and inventory costs.
Constraints \eqref{mc-1},\eqref{mc-1b}, and \eqref{mc-2} are the inventory balance constraints for each commodity at the production plant, warehouses, and retailers, correspondingly. Constraints \eqref{mc-3},\eqref{mc-4}, and \eqref{mc-5} are setup enforcing constraints for the production plant, warehouses, and retailers, respectively. Constraints \eqref{mc-6} and \eqref{mc-7} define the nonnegativity and integrality requirements on the variables.
This formulation has $O({|R|}\times {|T|}^2)$ variables and inequalities.

\section{Valid inequalities}
\label{sec:validinequalities}

In this section, we describe new valid inequalities for the uncapacitated three-level lot-sizing and replenishment problem with a distribution structure (3LSPD-U). 
Subsection~\ref{sec:dicutcollection} describes the well-known dicut collection inequalities, which will be used to show that the inequalities proposed in this section are valid.
Subsection \ref{sec:originalvalidinequalities} describes valid inequalities in the space of the standard formulation.
Subsection \ref{sec:extendedvalidinequalities} presents an alternative extended formulation whose size is asymptotically equivalent to that of the standard formulation and characterizes valid inequalities in this extended space.

\subsection{Dicut collection inequalities}
\label{sec:dicutcollection}

The uncapacitated fixed-charge network flow problem (UFCNF) {is an NP-hard problem~\cite{GarJoh79} that} can be defined as follows. Consider a directed graph $G=(N,A)$ with a set of nodes $N$ and a set of arcs $A=(F,\bar{F})$, where $F$ denotes the set of arcs with fixed-charge and $\bar{F}$ the set of arcs with continuous costs. The goal consists of determining a minimum cost combination of arcs to provide flows from defined supply nodes to a collection of demand nodes. As it was noted earlier in Subsection
~\ref{sec:stdform}, 3LSPD-U can be seen as an uncapacitated fixed-charge network flow problem.

Let $s$ be a source vertex and define $\mathcal{T}$ to be a set of considered sinks.
A $t$-dicut for $t\in \mathcal{T}$ is a set of arcs whose removal from $A$ blocks all flow from the source $s$ to the sink $t$. We remark that although we already use the notation $t$ for periods, its use whenever we talk about $t$-dicut is related to a demand node $t$ in the fixed-charge network. A dicut collection $\Gamma = \{\Gamma^t\}_{t \in \mathcal{T}}$ is a set of $t$-dicuts. A simple dicut collection is one with at most a single dicut for every $t\in \mathcal{T}$, i.e., $\Gamma^t\leq 1$ for every $t\in \mathcal{T}$.

Consider $\gamma^t_{ij}$ to be the number of $t$-dicuts in $\Gamma^t$ containing arc $(i,j)$, $\gamma_{ij} = \max\{\gamma^t_{ij}: t\in \mathcal{T}\}$
and $\gamma^t = |\Gamma^t|$.
\citeA{RarWol93} have shown that every inequality 
\begin{equation}\label{ineq:dicut}
    \sum_{(i,j) \in \bar{F}} \gamma_{ij} x_{ij} + \sum_{(i,j)\in F}\sum_{t \in \mathcal{T}} d_t \gamma^t_{ij} y_{ij} \geq \sum_{t\in \mathcal{T}} \gamma^t d_t
\end{equation}
derived from a dicut collection is valid for the UFCNF. A simple dicut collection inequality is one obtained from a simple dicut collection. {We remark that, to the best of our knowledge, there is no efficient separation procedure for separating dicut collection inequalities in general.}

In the remainder of this section, we will describe valid inequalities for 3LSPD-U and show that they can be obtained as simple dicut collection inequalities.

\subsection{Valid inequalities in the space of the standard formulation}\label{sec:originalvalidinequalities}

Define $X^{STD}$ as the set of integer feasible solutions for the standard formulation, i.e., those satisfying \eqref{std-1}-\eqref{std-5}.
Additionally, define $\delta^b(i)$ as the successors of $i$ at level $b$. In what follows, we characterize single-level, two-level, and three-level valid inequalities for $X^{STD}$.
{
The basic intuition behind the inequalities presented in this section is the following. Given a facility $i$ and a period $l$ with a cumulative demand $d^i_{1l}$ to be satisfied, the inequality provides a proper selection of variables for each period $1\leq k \leq l$ in the corresponding levels which guarantee that such cumulative demand is satisfied.
}

\begin{prop}\label{prop:stdsinglelevel}
Consider a facility $i\in F$. Let $l \in T$, $L^i = \{1,\ldots,l\}$ and $S^i \subseteq L^i$. The single-level inequalities
\begin{equation}\label{ineq:stdsinglelevel}
     \sum_{k \in L^i\setminus S^i}x^i_k + \sum_{k \in S^i} d^i_{kl} y^i_k \geq d^i_{1l}
\end{equation}
are valid for $X^{STD}$.
\end{prop}

\begin{proof}
    Note that \eqref{ineq:stdsinglelevel} are $(l,S)$-inequalities~\cite{BarWol84} for uncapacitated lot-sizing relaxations related to a single facility.
Given $i\in F$, $L^i$ and $S^i$, we show that inequalities \eqref{ineq:stdsinglelevel} can be obtained as simple dicut inequalities. Define a dicut collection with the $t$-dicut $\{x^i_k \ | \ k \in L^i\setminus S^i,\ k\leq t \} \cup \{y^i_k \ | \ k \in S^i,\ k\leq t \}$ for the nodes associated to each pair  $\{i,t\in L_i\}$ with demand $d^i_t$. The result follows using \eqref{ineq:dicut}. 
\end{proof}

\begin{prop}\label{prop:stdtwolevel}
Consider a facility $i\in P\cup W$ and denote its level by $b$. Let $l \in \{2,\ldots,{|T|}\}$,  $L^i = \{1,\ldots,l_i\} \subseteq L$, with $l_i<l$, and $S^i \subseteq L^i$. Besides, consider level $b'>b$ and for each $j \in \delta^{b'}(i)$, let $L^j = \{l_i+1,\ldots,l\}$ and $S^j \subseteq L^j$. The two-level inequalities
\begin{equation}\label{ineq:stdtwolevel}
    \sum_{k \in L^i\setminus S^i}x^i_k + \sum_{k \in S^i} d^i_{kl} y^i_k + \sum_{j \in \delta^{b'}(i)} \left( \sum_{k \in L^{j}\setminus S^{j}} x^{j}_k +  \sum_{k \in S^{j}} d^j_{kl} y^{j}_k \right) \geq d^i_{1l}
\end{equation}
are valid for $X^{STD}$.
\end{prop}
\begin{proof}
Note that inequalities ~\eqref{ineq:stdtwolevel} can be seen as generalized $(l,S)$-like inequalities for two-level relaxations of the problem.
 Given a facility $i\in P\cup W$, $L$, and $L^i = \{1,\ldots,l_i\}$, define a dicut collection with the $t$-dicut $\{x^i_k \ | \ k \in L^i\setminus S^i,\ k\leq t \} \cup \{y^i_k \ | \ k \in S^i,\ k\leq t \}$ for the nodes associated with each pair $\{i,t\in L_i\}$ with demand $d^i_t$, and the $t$-dicut $\{x^i_k \ | \ k \in L^i\setminus S^i \} \cup \{y^i_k \ | \ k \in S^i \} \cup \{x^j_k \ | \ k \in L^j\setminus S^j,\ k \leq t \} \cup \{y^j_k \ | \ k \in S^j,\ k \leq t \} $ for the nodes related to each pair $\{j \in \delta^{b'}(i),t\in L_j\}$ with demand $d^j_t$. Using \eqref{ineq:dicut}, the inequalities are valid. 
\end{proof}

\begin{prop}\label{prop:stdthreelevel}
Let $l \in \{3,\ldots,{|T|}\}$, $L^p = \{1,\ldots,l_p\}$, with $l_p\leq l-2$, and $S^p \subseteq L^p$. Besides, consider $L^w=\{l_p+1,\ldots,l_w\}$, with $l_w\leq l-1$, and $S^w \subseteq L^w$. Furthermore, let $L^r=\{l_w+1,\ldots,l\}$ and $S^r \subseteq L^r$. The three-level inequalities
\begin{equation}\label{ineq:stdthreelevel}
    \sum_{k \in L^p\setminus S^p}x^p_k + \sum_{k \in S^p} d^p_{kl} y^p_k + \sum_{w \in W} \left( \sum_{k \in L^{w}\setminus S^{w}} x^{w}_k + \sum_{k \in S^{w}} d^w_{kl} y^{w}_k \right) + \sum_{r \in R} \left( \sum_{k \in L^{r}\setminus S^{r}} x^{r}_k + \sum_{k \in S^{r}} d^r_{kl} y^{r}_k \right) \geq d^p_{1l}
\end{equation}
are valid for $X^{STD}$.
\end{prop}
\begin{proof}
 Given $L^p$ and $S^p$, $L^w$ and $S^w$ for each $w\in W$, $L^r$ and $S^r$ for each $r\in R$, define a dicut collection with the $t$-dicut $\{x^p_k \ | \ k \in L^p\setminus S^p,\ k\leq t \} \cup \{y^p_k \ | \ k \in S^p,\ k\leq t \}$ for all nodes associated with each pair $\{p,t\in L_p\}$ with demand $d^p_t$, the $t$-dicut $\{x^p_k \ | \ k \in L^p\setminus S^p \} \cup \{y^p_k \ | \ k \in S^p \} \cup \{x^w_k \ | \ k \in L^w\setminus S^w,\ k \leq t \} \cup \{y^w_k \ | \ k \in S^w,\ k \leq t \} $ for the nodes related to each pair $\{w \in W, t\in L_w\}$ with demand $d^w_t$, and the $t$-dicut $\{x^p_k \ | \ k \in L^p\setminus S^p \} \cup \{y^p_k \ | \ k \in S^p \} \cup \{x^{\delta_w(r)}_k \ | \ k \in L^{\delta_w(r)}\setminus S^{\delta_w(r)} \} \cup \{y^{\delta_w(r)}_k \ | \ k \in S^{\delta_w(r)} \} \cup \{x^r_k \ | \ k \in L^r\setminus S^r,\ k \leq t \} \cup \{y^r_k \ | \ k \in S^r,\ k \leq t \} $ for the nodes related to each pair $\{r \in R, t\in L_r\}$ with demand $d^r_t$. Using \eqref{ineq:dicut}, the inequalities are valid. 
\end{proof}

\subsection{Valid inequalities for an extended three-level lot-sizing based formulation} \label{sec:extendedvalidinequalities}

We now propose an extended three-level lot-sizing based formulation which decomposes the {production, transportation, and stocks} in the upper levels (levels 0 and 1) into the specific retailers for which they are related.
{We remark that, differently from the multi-commodity formulation, this decomposition does not determine the specific period of the corresponding retailer to be satisfied, but simply the retailer.}
The goal of such formulation is to permit further exploring the structure of the network flow without implying a very large number of variables. We remark that a similar idea was applied to the one-warehouse multi-retailer problem in~\citeA{CunMel16}.

Define variable $x^{0r}_t$ to be the amount produced at the production plant in period $t\in T$ to satisfy some demand of retailer $r\in R$, variable $x^{1r}_t$ to be the amount transported from the production plant to the warehouse of retailer $r\in R$ in period $t\in T$ to satisfy some demand of $r$, and variable $x^{2r}_t$ to be the amount transported to retailer $r\in R$ from its warehouse in period $t\in T$ to satisfy some of its demands. Additionally, let variable $s^{0r}_t$ be the amount stocked in the production plant at the end of period $t\in T$ to satisfy demand of retailer $r\in R$, $s^{1r}_t$ be the amount stocked in the warehouse of retailer $r \in R$ at the end of period $t\in T$ to satisfy its demand, and $s^{2r}_t$ be the amount stocked in retailer $r\in R$ at the end of period $t\in T$. An extended three-level lot-sizing based formulation can be defined as
\begin{align}
z_{3LF} = & \  \min \ \ \  \sum_{t \in T}\left( \sum_{i\in F} sc^i_t y^i_t + \sum_{r\in R} hc^p_t s^{0r}_t + \sum_{r\in R} hc^{\delta_w(r)}_t s^{1r}_t + \sum_{r\in R} hc^r_t s^{2r}_t \right)  \label{3level-obj} & \\
(3LF) \qquad & s^{0r}_{t-1} + x^{0r}_t = x^{1r}_t + s^{0r}_t, \qquad  \textrm{for} \ r \in R, \ t\in T, \label{3level-1} \\
& s^{1r}_{t-1} + x^{1r}_t = x^{2r}_t + s^{1r}_t, \qquad  \textrm{for} \ r \in R, \ t\in T, \label{3level-1b} \\
&  s^{2r}_{t-1} + x^{2r}_t = d^r_t + s^{2r}_t, \qquad  \textrm{for} \ r \in R, \ t\in T, \label{3level-2} \\
&  x^{0r}_t \leq d^r_{t{|T|}} y^{p}_t, \qquad  \textrm{for} \ r \in R, \ t\in T, \label{3level-3} \\
&  x^{1r}_t \leq d^r_{t{|T|}} y^{\delta_w(r)}_t, \qquad  \textrm{for} \ r \in R, \ t\in T, \label{3level-4} \\
&  x^{2r}_t \leq d^r_{t{|T|}} y^{r}_t, \qquad  \textrm{for} \ r \in R, \ t\in T, \label{3level-5} \\
&  x^{0r}_{t},\ x^{1r}_{t},\ x^{2r}_{t}, \ s^{0r}_t, \ s^{1r}_t, \ s^{2r}_t \geq 0, \qquad  \textrm{for} \ r \in R, \ t\in T, \label{3level-6}\\
&  y^i_{t} \in  \{0,1\}, \qquad  \textrm{for} \ i \in F, \ t\in T. \label{3level-7}
\end{align}
The objective function \eqref{3level-obj} minimizes the total setup and inventory costs.
Constraints \eqref{3level-1}, \eqref{3level-1b} and \eqref{3level-2} are inventory balance constraints for, respectively, the production plant, the warehouses and the retailers. Constraints \eqref{3level-3}, \eqref{3level-4} and \eqref{3level-5} are setup enforcing constraints.
Constraints \eqref{3level-6} ensure the nonnegativity of the variables.
This formulation, similarly to the standard formulation, has $O({|R|}\times {|T|})$ variables and inequalities.

Define $X^{3LF}$ as the set of feasible solutions for the extended three-level lot-sizing based formulation, i.e., those satisfying \eqref{3level-1}-\eqref{3level-7}.
Consider $L^b\subseteq \{1,\ldots,l\}$, $1\leq l \leq {|T|}$ , as the set of periods considered for level $b$, and let $S^b\subseteq L^b$.
Additionally, denote as $\delta^b_w(r)$ the predecessor of $r$ at level $b$, i.e., $\delta^0_w(r) = p$, $\delta^1_w(r) = \delta_w(r)$, and $\delta^2_w(r) = r$. In what follows, we describe single-level, two-level and three-level valid inequalities for $X^{3LF}$.

\begin{prop}\label{prop:3levelsinglelevel}
Consider a retailer $r\in R$, a level $b\in \{0,1,2\}$, and $l\in T$.
The single-level inequalities
\begin{equation}\label{ineq:3levelsinglelevel}
    \sum_{k \in L^b\setminus S^b}x^{br}_k + \sum_{k \in S^b} d^r_{kl} y^{\delta^b_w(r)}_k \geq d^r_{1l}
\end{equation}
are valid for $X^{3LF}$.
\end{prop}
\begin{proof}
Note that \eqref{ineq:3levelsinglelevel} are $(l,S)$-inequalities~\cite{BarWol84} for uncapacitated lot-sizing relaxations related to a single facility and the demands of a specific retailer. This result can be proven using a reasoning similar to the one in Proposition~\ref{prop:stdsinglelevel}, but considering the demand of each retailer individually.
Given $r\in R$, $b$, $L^b$ and $S^b$, we show that inequalities \eqref{ineq:3levelsinglelevel} can be obtained as simple dicut inequalities. Define a dicut collection with the $t$-dicut $\{x^{br}_k \ | \ k \in L^b\setminus S^b,\ k\leq t \} \cup \{y^{\delta^b_w(r)}_k \ | \ k \in S^b,\ k\leq t \}$ for the nodes associated to each pair $\{r,t\in L^b\}$ with demand $d^r_t$. The result follows using \eqref{ineq:dicut}.

\end{proof}

\begin{prop}\label{prop:3leveltwolevel}
Consider a retailer $r\in R$, levels $b,b' \in \{0,1,2\}$, with $b<b'$, and $l\in \{2,\ldots,{|T|}\}$. Let $L^b = \{1,\ldots,l_b\}$, with $l_b < l$, and $S^b \subseteq L^b$. Also, let $L^{b'} = \{l_b+1,\ldots,l\}$ and $S^{b'} \subseteq L^{b'}$. 
The two-level inequalities
\begin{equation}\label{ineq:3leveltwolevel}
    \sum_{k \in L^b\setminus S^b}x^{br}_k + \sum_{k \in S^b} d^r_{kl} y^{\delta^b_w(r)}_k + \sum_{k \in L^{b'}\setminus S^{b'}} x^{b'r}_k + \sum_{k \in S^{b'}} d^r_{kl} y^{\delta^{b'}_w(r)}_k \geq d^r_{1l}
\end{equation}
are valid for $X^{3LF}$.
\end{prop}
\begin{proof}
Note that \eqref{ineq:3leveltwolevel} can be seen as special cases of the two-echelon inequalities~\cite{ZhaKucYam12} for the multi-echelon lot-sizing  with intermediate demands. 
Besides, the result can be proven using the same idea as the one in Proposition~\ref{prop:stdtwolevel}, but taking into consideration the demand of each retailer individually.
Given $r\in R$, $b$, $b'$, $L^b$, and $S^b$, define a dicut collection with the $t$-dicut $\{x^{br}_k \ | \ k \in L^b\setminus S^b,\ k\leq t \} \cup \{y^{\delta^b_w(r)}_k \ | \ k \in S^b,\ k\leq t \}$ for the nodes associated to each pair $\{r,t\in L^b\}$ with demand $d^r_t$, and the $t$-dicut $\{x^{br}_k \ | \ k \in L^{b}\setminus S^{b} \} \cup \{y^{\delta^{b}_w(r)}_k \ | \ k \in S^{b} \} \cup \{x^{b'r}_k \ | \ k \in L^{b'}\setminus S^{b'},\ k \leq t \} \cup \{y^{\delta^{b'}_w(r)}_k \ | \ k \in S^{b'},\ k \leq t \} $ for the nodes associated to each pair $\{r,t\in L^{b'}\}$ with demand $d^r_t$. Using \eqref{ineq:dicut}, the inequalities are valid.
\end{proof}

\begin{prop}\label{prop:3levelthreelevel}
Consider a retailer $r\in R$, and $l\in \{3,\ldots,{|T|}\}$. Let $L^0 = \{1,\ldots,l_0\}$, with $l_0 < l-1$, and $S^0 \subseteq L^0$. Also, let $L^{1} = \{l_0+1,\ldots,l_1\}$, with $l_1<l$, and $S^{1} \subseteq L^{1}$. Furthermore, let $L^{2} = \{l_1+1,\ldots,l\}$ and $S^{2} \subseteq L^{2}$. 
The three-level inequalities
\begin{equation}\label{ineq:3levelthreelevel}
    \sum_{k \in L^0\setminus S^0}x^{0r}_k + \sum_{k \in S^0} d^r_{kl} y^p_k + \sum_{k \in L^{1}\setminus S^{1}} x^{1r}_k + \sum_{k \in S^{1}} d^r_{kl} y^{\delta_w(r)}_k  + \sum_{k \in L^{2}\setminus S^{2}} x^{2r}_k + \sum_{k \in S^{2}} d^r_{kl} y^{r}_k   \geq d^r_{1l}
\end{equation}
are valid for $X^{3LF}$.
\end{prop}
\begin{proof}
The proof can be achieved using a reasoning similar to the one in Proposition~\ref{prop:stdthreelevel}, but taking into consideration the demand of each retailer individually.
Given $L^0$, $S^0$, $L^1$, $S^1$, $L^2$ and $S^2$, define a dicut collection with the $t$-dicut $\{x^{0r}_k \ | \ k \in L^0\setminus S^0,\ k\leq t \} \cup \{y^p_k \ | \ k \in S^0,\ k\leq t \}$ for the nodes related to each pair $\{r,t\in L^0\}$ with demand $d^r_t$, the $t$-dicut $\{x^{0r}_k \ | \ k \in L^0\setminus S^0 \} \cup \{y^{p}_k \ | \ k \in S^0 \} \cup \{x^{1r}_k \ | \ k \in L^1\setminus S^1,\ k \leq t \} \cup \{y^{\delta_w(r)}_k \ | \ k \in S^1,\ k \leq t \} $ for the nodes associated with each pair $\{r,t\in L_1\}$ with demand $d^r_t$, and the $t$-dicut $\{x^{0r}_k \ | \ k \in L^0\setminus S^0 \} \cup \{y^p_k \ | \ k \in S^0 \} \cup \{x^{1r}_k \ | \ k \in L^1\setminus S^1 \} \cup \{y^{\delta_w(r)}_k \ | \ k \in S^1 \} \cup \{x^{2r}_k \ | \ k \in L^2\setminus S^2,\ k \leq t \} \cup \{y^r_k \ | \ k \in S^2,\ k \leq t \} $ for the nodes corresponding to each pair $\{r,t\in L_2\}$ with demand $d^r_t$. 
The inequalities are thus valid, using \eqref{ineq:dicut}.

\end{proof}


\subsection{Comparing the bounds achieved with the formulations strengthened with valid inequalities}

We now compare the bounds achieved when using the standard ($STD$) and extended three-level lot-sizing based ($3LF$) formulations when enhanced with the proposed valid inequalities. 
Denote by $\underline{z}_{STD+}$ and $\underline{z}_{3LF+}$, respectively, the linear relaxation values of $STD$ with the addition of valid inequalities \eqref{ineq:stdsinglelevel}, \eqref{ineq:stdtwolevel} and \eqref{ineq:stdthreelevel}, denoted as STD+, and of $3LF$ with the addition of inequalities \eqref{ineq:3levelsinglelevel}, \eqref{ineq:3leveltwolevel} and \eqref{ineq:3levelthreelevel}, denoted as $3LF+$.

Define $R^i$ to be the set of all the retailers which are descendants of facility $i \in F$, and let $b(i)$ be the level of facility $i$. 
 The proof consists of showing that for any feasible solution to the linear relaxation of $3LF+$, there is a corresponding solution to the linear relaxation of $STD+$ with the same objective value.


\begin{lemma}\label{lemma:3LFSTD}
Given a feasible solution $(\hat{x}^0,\hat{x}^1,\hat{x}^2,\hat{s}^0,\hat{s}^1,\hat{s}^2,\hat{y})$ for the linear relaxation of $3LF$, there is a corresponding solution $(\bar{x},\bar{s},\hat{y})$ for the linear relaxation of $STD$ with the same objective value.
\end{lemma}
\begin{proof}

Note that there is a direct mapping between variables $\bar{x},\bar{s}$ of $STD$ and $\hat{x}^b,\hat{s}^b$ for $b\in \{0,1,2\}$ of $3LF$, which is expressed for each $i\in F$ and $k\in T$ as
\begin{equation}\label{eq:mappingx}
\bar{x}^i_k = \sum_{r \in R^i} \hat{x}^{b(i),r}_k
\end{equation}
and 
\begin{equation}\label{eq:mappings}
\bar{s}^i_k = \sum_{r \in R^i} \hat{s}^{b(i),r}_k.
\end{equation}

For each facility $i\in P$ and period $t\in T$, and for each facility $i\in W$ and period $t\in T$, if we sum, respectively, constraints \eqref{3level-1} and \eqref{3level-1b} for every $r \in R^i$, we have that 
$$\sum_{r \in R^i} \hat{s}^{b(i),r}_{t-1} + \sum_{r \in R^i} \hat{x}^{b(i),r}_t = \sum_{r \in R^i} d^r_t + \sum_{r \in R^i} \hat{s}^{b(i),r}_t.$$
Using \eqref{eq:mappingx} and \eqref{eq:mappings}, this guarantees that
$$\bar{s}^{i}_{t-1} + \bar{x}^{i}_t =  d^i_t + \bar{s}^{i}_t, $$
implying that constraints \eqref{std-1} are satisfied by $(\bar{x},\bar{s},\hat{y})$.

For each $r\in R$ and $t\in T$, constraints \eqref{3level-2} ensure that
$$\hat{s}^{2r}_{t-1} + \hat{x}^{2r}_t = d^r_t + \hat{s}^{2r}_t,$$ which, using \eqref{eq:mappingx} and \eqref{eq:mappings}, guarantee that
$$\bar{s}^{r}_{t-1} + \bar{x}^{r}_t = d^r_t + \bar{s}^{r}_t,$$
and thus, constraints \eqref{std-2} are satisfied by $(\bar{x},\bar{s},\hat{y})$.

For each facility $i\in F$ and period $t\in T$, if we sum the corresponding inequalities \eqref{3level-3}, \eqref{3level-4} or \eqref{3level-5} for every $r \in R^i$, we have that 
$$\sum_{r \in R^i} \hat{x}^{b(i),r} \leq  \sum_{r \in R^i} d^r_{t{|T|}} \hat{y}^i_t, $$
which, using \eqref{eq:mappingx} and \eqref{eq:mappings}, guarantee that
$$ \bar{x}^{i}_t \leq d^i_{t{|T|}} \hat{y}^i_t,$$
implying that constraints \eqref{std-3} are satisfied by $(\bar{x},\bar{s},\hat{y})$.

Constraints \eqref{3level-6} and \eqref{3level-7} ensure $(\bar{x},\bar{s},\hat{y})$ satisfy, respectively, constraints \eqref{std-4} and \eqref{std-5}. 
Therefore, the result holds as the objective functions \eqref{std-obj} and \eqref{3level-obj} are equivalent given the mapping represented by equations \eqref{eq:mappingx} and \eqref{eq:mappings}.

\end{proof}


\begin{lemma}\label{lemma:singlelevel}
Given a feasible solution $(\hat{x}^0,\hat{x}^1,\hat{x}^2,\hat{s}^0,\hat{s}^1,\hat{s}^2,\hat{y})$ for the linear relaxation of $3LF$ satisfying inequalities \eqref{ineq:3levelsinglelevel}, then solution $(\bar{x},\bar{s},\hat{y})$ obtained with the mapping  \eqref{eq:mappingx} and  \eqref{eq:mappings} for the linear relaxation of $STD$ satisfies inequalities \eqref{ineq:stdsinglelevel}.
\end{lemma}
\begin{proof}
Consider an inequality \eqref{ineq:stdsinglelevel} for predefined $i\in F$, $l\in T$ and $S^i$
\begin{equation*}
     \sum_{k \in L^i\setminus S^i}x^i_k + \sum_{k \in S^i} d^i_{kl} y^i_k \geq d^i_{1l}.
\end{equation*}
 Let $S^b = S^i$ and sum inequalities \eqref{ineq:3levelsinglelevel} for every $r\in R^i$ to obtain
\begin{equation*}
    \sum_{r\in R^i} \sum_{k \in L^i\setminus S^i} \hat{x}^{b(i),r}_k + \sum_{r \in R^i}\sum_{k \in S^i} d^r_{kl} \hat{y}^{\delta^b_w(r)}_k \geq \sum_{r \in R^i} d^r_{1l}.
\end{equation*}
Using \eqref{eq:mappingx} and  \eqref{eq:mappings}, this is equivalent to
\begin{equation*}
    \sum_{k \in L^i\setminus S^i}\bar{x}^{i}_k + \sum_{k \in S^i} d^i_{kl} \hat{y}^{i}_k \geq d^i_{1l},
\end{equation*}
and thus, $(\bar{x},\bar{s},\hat{y})$ satisfies \eqref{ineq:stdsinglelevel}.
\end{proof}


\begin{lemma}\label{lemma:twolevel}
Given a feasible solution $(\hat{x}^0,\hat{x}^1,\hat{x}^2,\hat{s}^0,\hat{s}^1,\hat{s}^2,\hat{y})$ for the linear relaxation of $3LF$ satisfying inequalities \eqref{ineq:3leveltwolevel}, then solution $(\bar{x},\bar{s},\hat{y})$ obtained with the mapping  \eqref{eq:mappingx} and  \eqref{eq:mappings} for the linear relaxation of $STD$ satisfies inequalities \eqref{ineq:stdtwolevel}.
\end{lemma}
\begin{proof}

Consider an inequality \eqref{ineq:stdtwolevel} for specific $i\in P\cup W$, $l$,  $L^i$, $S^i$, $b'$, as well as, $L^j$ and $S^j$ for each $j \in \delta^{b'}(i)$
\begin{equation*}
    \sum_{k \in L^i\setminus S^i}x^i_k + \sum_{k \in S^i} d^i_{kl} y^i_k + \sum_{j \in \delta^{b'}(i)} \left( \sum_{k \in L^{j}\setminus S^{j}} x^{j}_k +  \sum_{k \in S^{j}} d^j_{kl} y^{j}_k \right) \geq d^i_{1l}.
\end{equation*}
Let $S^{b(i)}=S^i$, and for each $j \in \delta^{b'}(i)$ choose $S^{b'}$ as $S^j$ and sum inequalities \eqref{ineq:3leveltwolevel} with these choices for every $r\in R^i$ to obtain
\begin{equation*}
    \sum_{r \in R^i}\sum_{k \in L^i\setminus S^i}\hat{x}^{br}_k + \sum_{r \in R^i} \sum_{k \in S^i} d^r_{kl} \hat{y}^{\delta^b_w(r)}_k + \sum_{j \in \delta^{b'}(i)} \left( \sum_{r \in R^j} \sum_{k \in L^{b'(r)}\setminus S^{b'(r)}} \hat{x}^{b'r}_k + \sum_{r \in R^j} \sum_{k \in S^{b'(r)}} d^r_{kl} \hat{y}^{\delta^{b'}_w(r)}_k \right) \geq \sum_{r \in R^i} d^r_{1l}.
\end{equation*}
Thus, using \eqref{eq:mappingx} and \eqref{eq:mappings} we have
\begin{equation*}
    \sum_{k \in L^i\setminus S^i}\bar{x}^{i}_k + \sum_{k \in S^i} d^i_{kl} \hat{y}^{i}_k + \sum_{j \in \delta^b(i)} \left( \sum_{k \in L^{j}\setminus S^{j}} \bar{x}^{j}_k +  \sum_{k \in S^{j}} d^j_{kl} \hat{y}^{j}_k \right) \geq d^i_{1l},
\end{equation*}
implying that $(\bar{x},\bar{s},\hat{y})$ satisfies \eqref{ineq:stdtwolevel}.

\end{proof}


\begin{lemma}\label{lemma:threelevel}
Given a feasible solution $(\hat{x}^0,\hat{x}^1,\hat{x}^2,\hat{s}^0,\hat{s}^1,\hat{s}^2,\hat{y})$ for the linear relaxation of $3LF$ satisfying inequalities \eqref{ineq:3levelthreelevel}, then solution $(\bar{x},\bar{s},\hat{y})$ obtained with the mapping  \eqref{eq:mappingx} and  \eqref{eq:mappings} for the linear relaxation of $STD$ satisfies inequalities \eqref{ineq:stdthreelevel}.
\end{lemma}
\begin{proof}

Consider an inequality \eqref{ineq:stdthreelevel} for specified $l$, $L^p$, $S^p$, $L^w$, $S^w$, $L^r$ and $S^r$
\begin{equation*}
    \sum_{k \in L^p\setminus S^p}x^p_k + \sum_{k \in S^p} d^p_{kl} y^p_k + \sum_{k \in L^{w}\setminus S^{w}} x^{w}_k + \sum_{k \in S^{w}} d^w_{kl} y^{w}_k + \sum_{k \in L^{r}\setminus S^{r}} x^{r}_k + \sum_{k \in S^{r}} d^r_{kl} y^{r}_k  \geq d^p_{1l}.
\end{equation*}

Let $L^0= L^p$ and $S^0= S^p$, $L^1= L^w$ and $S^1= S^w$ for every $w \in W$, and $L^2= L^r$ and $S^2= S^r$ for every $r \in R$. Summing inequalities \eqref{ineq:3levelthreelevel} with the aforementioned choices for every $r\in R$, we obtain
\begin{multline*}
    \sum_{r\in R} \sum_{k \in L^p\setminus S^p} \hat{x}^{0r}_k + \sum_{r\in R} \sum_{k \in S^p} d^r_{kl} \hat{y}^p_k + \sum_{r\in R} \sum_{k \in L^{\delta_w(r)}\setminus S^{\delta_w(r)}} \hat{x}^{1r}_k + \sum_{r\in R} \sum_{k \in S^{\delta_w(r)}} d^r_{kl} \hat{y}^{\delta_w(r)}_k  + \\ \sum_{r\in R} \sum_{k \in L^{r}\setminus S^{r}} \hat{x}^{2r}_k + \sum_{r\in R} \sum_{k \in S^{r}} d^r_{kl} \hat{y}^{r}_k   \geq \sum_{r\in R} d^r_{1l}.
\end{multline*}
This, together with \eqref{eq:mappingx} and \eqref{eq:mappings},  ensures that
\begin{equation*}
    \sum_{k \in L^p\setminus S^p}\bar{x}^p_k + \sum_{k \in S^p} d^p_{kl} \hat{y}^p_k + \sum_{k \in L^{w}\setminus S^{w}} \bar{x}^{w}_k + \sum_{k \in S^{w}} d^w_{kl} \hat{y}^{w}_k + \sum_{k \in L^{r}\setminus S^{r}} \bar{x}^{r}_k + \sum_{k \in S^{r}} d^r_{kl} \hat{y}^{r}_k  \geq d^p_{1l},
\end{equation*}
implying that $(\bar{x},\bar{s},\hat{y})$ satisfies \eqref{ineq:stdthreelevel}.

\end{proof}


\begin{prop}
$\underline{z}_{STD+} \leq \underline{z}_{3LF+}$.
\end{prop}
\begin{proof}
It follows from Lemmas \ref{lemma:3LFSTD}, \ref{lemma:singlelevel}, \ref{lemma:twolevel} and \ref{lemma:threelevel}.
\end{proof}

\section{Preprocessing the multi-commodity formulation}
\label{sec:preprocessing}

This section details a preprocessing technique to reduce the size of the multi-commodity formulation for the uncapacitated three-level lot-sizing and replenishment problem with a distribution structure (3LSPD-U).
It has been noted in the literature that extended formulations can become computationally intractable as the sizes of the instances increase~\cite{VanWol06}. The proposed preprocessing goes into the direction of attempting to reduce the size of the formulation and make it more manageable to be tackled by a commercial mixed integer programming solver.
Consider the multi-commodity formulation described in Section~\ref{sec:bestliterature}. In what follows we show that, under certain conditions, variables can be set to zero (i.e., can be removed from the formulation) and an optimal solution for the problem can still be encountered, whenever one exists.

Firstly, remember that 3LSPD-U can be seen as an uncapacitated fixed-charge network flow problem (UFCNF).
It is well known that, in an extreme-point optimal solution for UFCNF, the underlying graph associated with the variables which assume values strictly between their lower and upper bounds is acyclic (see \citeA{PocWol06}). This implies that, for 3LSPD-U, there is an optimal solution in which, whenever items are transported from the warehouse to the retailer, they are used to satisfy all the demands of a set of consecutive periods for that retailer.

\begin{prop}\label{prop:preprocess}
Consider a retailer $r\in R$ and periods $1 \leq k<t \leq {|T|}$ for which
\begin{equation}\label{eq:preprocess}
d^r_t\times \sum_{l=k}^{t-1}hc^r_l \geq  d^r_t\times \sum_{l=k}^{t-1}hc^{\delta_w(r)}_l + sc^r_t.
\end{equation}
Variables $w^{2r}_{kt'}$ can be set to zero for every $t'\in \{t,\ldots,{|T|}\}$, and the guarantee of encountering an optimal solution, whenever it exists, is preserved.
\end{prop}
\begin{proof}
Note that the left-hand side of \eqref{eq:preprocess} defines the total cost of storing the demand $d^r_t$ from periods $k$ up to $t$ at the retailer, while its right-hand side establishes the total cost of storing this demand from periods $k$ up to $t$ at the warehouse plus the setup cost implied for transporting it from the warehouse to the retailer in period $t$. Consider a solution $(\hat{w},\hat{\sigma},\hat{y})$ in which the conditions of the proposition hold such that $\hat{w}^{2r}_{kt} = d^r_t$ and the cost is $\hat{z}$. We build a new solution $(\bar{w},\bar{\sigma},\bar{y})$ in which $\bar{w}^{2r}_{kt}=0$, $\bar{w}^{2r}_{tt}= \hat{w}^{2r}_{kt}=d^r_t$, $\bar{y}^r_t=1$, and all other variables assume the same values as in $(\hat{w},\hat{\sigma},\hat{y})$. Due to \eqref{eq:preprocess}, the new cost is thus $\bar{z} \leq \hat{z} + d^r_t \times \sum_{l=k}^{t-1}hc^{\delta_w(r)}_l + sc^r_t - d^r_t\times \sum_{l=k}^{t-1}hc^r_l \leq \hat{z}$. This implies that variable ${w}^{2r}_{kt}$ can be set to zero. Consequently, due to the property of an  extreme-point optimal solution, ${w}^{2r}_{kt'}$ can also be set to zero for $t\leq t'\leq {|T|}$.
\end{proof}

We remark that{, as can be observed in inequality \eqref{eq:preprocess},} the potential reductions {can be} very sensitive to the costs {and demands} received as inputs.

\section{Multi-start randomized bottom-up {dynamic programming-based heuristic}}
\label{sec:randomizedheuristic}

In this section, we provide an easy-to-implement multi-start randomized bottom-up {dynamic programming-based heuristic.} {As commercial solvers can sometimes encounter difficulties in obtaining good quality feasible solutions earlier in the enumeration process using general mixed integer programming heuristics, the goal of the proposed heuristic is to fastly provide such solutions} in an attempt to speed up the process of solving instances to optimality. The proposed approach is somehow related to the work of \citeA{DelJeu03}. 
{Each iteration of the proposed heuristic follows a level-by-level nonintegrated planning in which levels are sequentially tackled in a bottom-up fashion using dynamic programming, considering that certain costs receive controlled randomized changes.
It applies greedy randomization~\cite{ResRib03}, which is widely used in combinatorial optimization, in order to allow diversification in the obtained solutions, but in an alternative manner. More specifically, the costs are randomized with certain control, differently from the standard way of randomizing the choices during the execution of the algorithm. In this way, one can still take advantage of the effectiveness of the dynamic programming approaches for the subproblems to be solved.}

The heuristic is described in Algorithm~\ref{alg:multis}. It takes as inputs the instance, a randomization factor $\alpha$, and the allowed number of iterations $it_{max}$. 
The best known objective value is set as $\infty$ in line \ref{MS-0}. Next, a randomized solution is constructed at each iteration of the \textbf{for} loop of lines~\ref{MS-1}-\ref{MS-1end}. 
Firstly, the setup costs are randomized in line \ref{MS-2} according to parameter $\alpha$.
Namely, for each facility $i\in W\cup R$ and each period $t\in T$, the randomized setup cost is calculated as
\begin{equation}
\bar{sc}^i_t = sc^i_t + rand(0,\alpha) sc^i_t,
\end{equation}
where $rand(0,\alpha)$ returns a random value using a uniform distribution $U[0,\alpha]$. 
Next, the amounts to be transported from the warehouses to the retailers are determined in line \ref{MS-3} by solving an uncapacitated lot-sizing for each retailer using dynamic programming~\cite{WagWhi58,WagVanKol92}. Afterwards, the amounts to be transported from the production plant to the warehouses are determined in line \ref{MS-4} by solving an uncapacitated lot-sizing for each warehouse, considering that the demands are given by the amounts to be transported to the retailers previously determined in line \ref{MS-3}. Finally, the amounts to be produced at the production plant in each period are determined in line \ref{MS-5} by solving an uncapacitated lot-sizing, considering that the demands are given by the amounts to be transported to the warehouses previously determined in line \ref{MS-4}. Whenever the objective value $\hat{z}$ of the new obtained solution improves over the best-known, the best-known solution and value are updated (lines \ref{MS-7}-\ref{MS-7.2}).

\begin{algorithm}[H]
\caption {MS-R-BU-DPH (instance, $\alpha$, $it_{max}$)}
\label{alg:multis}

     $z^* \leftarrow \infty$\;\label{MS-0}
    \For{$it = 1,...,it_{max} $}{\label{MS-1}
        Calculate the randomized setup costs for the facilities in $W\cup R$ according to parameter $\alpha$\;\label{MS-2}
        Determine $(\hat{x}^R,\hat{y}^R,\hat{s}^R)$ by optimally solving an uncapacitated lot-sizing for each retailer $r\in R$ with randomized setup costs, using dynamic programming\;\label{MS-3}
        Determine $(\hat{x}^W,\hat{y}^W,\hat{s}^W)$ by optimally solving an uncapacitated lot-sizing for each warehouse $w\in W$ with the demands determined by $\hat{x}^R$ and randomized setup costs, using dynamic programming\;\label{MS-4}
        Determine $(\hat{x}^p,\hat{y}^p,\hat{s}^p)$ by optimally solving an uncapacitated lot-sizing for the plant with the demands determined by $x^W$, using dynamic programming, and set $\hat{z}$ as the cost of the complete solution considering the original setup costs\;\label{MS-5}
        \If{$\hat{z} < z^*$}{\label{MS-7}
            $(x^*,y^*,s^*) \leftarrow (\hat{x},\hat{y},\hat{s})$\;\label{MS-7.1}
            $z^* \leftarrow \hat{z}$\;\label{MS-7.2}
        }
    }\label{MS-1end}
    \Return $(x^*,y^*,s^*)$\;\label{MS-8}
\end{algorithm}


\section{Computational experiments}
\label{sec:experiments}

This section reports the computational experiments conducted to assess the performance of the proposed approaches.
All computational experiments were carried out on a machine running under Ubuntu GNU/Linux, with an Intel(R) Core(TM) i5-3740 CPU @ 3.20GHz processor and 8Gb of RAM. The algorithms were coded in Julia v1.4.2, using JuMP v0.18.6. The formulations were solved using Gurobi 9.0.2 with the standard configurations, except the relative optimality tolerance gap which was defined as $10^{-6}$ due to the magnitude of the costs and the algorithm to solve the root node of the multi-commodity formulation which was set to the barrier method given the characteristics of the formulation in question (large and with many equality constraints). A time limit of 3600 seconds was imposed for every execution of the MIP solver.
The goals of the performed experiments were twofold. Firstly, we wanted to analyze the effectiveness of the described valid inequalities both in terms of the offered bounds and of the improvements achieved over the plain formulations to optimally solve the problem. Secondly, we desired to assess the quality of the solutions obtained by the proposed multi-start randomized bottom-up dynamic programming{-based} heuristic as well as the speed up achieved by using the heuristic together with the preprocessing when using the multi-commodity formulation. 

\subsection{Benchmark instances}\label{sec:instancias}

The computational experiments were executed using the instances proposed in \citeA{GruBazCorJan19}, where more details can be obtained.
For the instances in this benchmark set, $P$ is a singleton with a unique plant, $|R| \in \{50,100,200\}$ , $|W| \in \{5, 10, 15, 20\}$, and $|T| \in \{15,30\}$. The demands of each retailer were generated using a uniform distribution $U[5, 100]$. The fixed costs for the production plant were defined using a uniform distribution $U[30000, 45000]$. The fixed costs for the warehouses were determined using a uniform distribution $U[1500, 4500]$. The fixed costs for the retailers were specified using a uniform distribution $U[5, 100]$. The unit inventory holding costs for the production plant were set to $0.25$. The unit inventory holding costs for the warehouses were set to $0.5$ and the unit inventory holding costs for the retailers were generated using a uniform distribution $U[0.5,1]$. The demands and fixed costs are generated as integer values and the holding costs take continuous values. The instances are organized in instance groups containing five instances each with similar characteristics. Each instance group is identified by $|R|$\_$|T|$\_$|W|$\_typeD\_typeF, where typeD and typeF define respectively the characteristics of the demands and fixed costs, which can be either static (S) or dynamic (D).
{
The instances can be classified as balanced or unbalanced. In the balanced networks, each warehouse has nearly the same number of retailers. In the unbalanced networks, around 20\% of the warehouses concentrate 80\% of the retailers. Notice that there are 960 instances (half of them are balanced while the other half are unbalanced).
}

We remark that each line in the tables presented in the remainder of this section corresponds to the instances belonging to the appropriate instance group.

\subsection{Implementation details and settings}

The separations of inequalities \eqref{ineq:stdsinglelevel}, \eqref{ineq:stdtwolevel}, \eqref{ineq:stdthreelevel}, \eqref{ineq:3levelsinglelevel}, \eqref{ineq:3leveltwolevel} and \eqref{ineq:3levelthreelevel} were performed with straightforward implementations. Namely, with all the parameters fixed but the choices of the $x$ and $y$ variables composing the inequality, the approach defines for each involved combination of facility and time period, whether the $x$ or $y$ variables will compose the inequality, and thus the separation can be performed by inspection similarly to the $(l,S)$-inequalities for the uncapacitated lot-sizing~\cite{BarWol84,PocWol06}. The separations were implemented as callbacks and are only executed at the root node. All the violated inequalities are provided to the solver.

All the settings and parameters were defined based on preliminary experiments which took into consideration a small subset containing  $\approx 3\%$ of the instances, randomly selected, with varying sizes and characteristics. For the separation procedures, the tolerance for violation was set to 10.0 (the values 1.0, 10.0, and 100.0 were tested). The maximum number of cutting plane rounds was set to 20 (the values 20, 30, and 40 were tested).
Inequalities \eqref{ineq:stdsinglelevel} and \eqref{ineq:3levelsinglelevel} are separated in every round. Inequalities \eqref{ineq:stdtwolevel} and \eqref{ineq:3leveltwolevel} are separated every five rounds. Inequalities \eqref{ineq:stdthreelevel} and \eqref{ineq:3levelthreelevel} are separated every ten rounds. This difference in the frequency of separation for the different families of inequalities is motivated by the large difference in the number of inequalities that can be violated for these families in a given round as well as the computational complexity for performing such tasks.

The settings for the multi-start randomized bottom-up {dynamic programming-based heuristic} were determined as follows. The maximum number of iterations $it_{max}$ was set to 500 (the values 100, 200, 500, 700, and 1000 were tested, but 500 offered a good compromise between running time and solution quality). The randomization parameter $\alpha$ was set to 0.20 (the values 0.05, 0.10, 0.15, 0.20, and 0.25 were tested).

\subsection{Results analyzing the use of valid inequalities}
\label{sec:resultsvalidinequalities}

The computational experiments assessing the effectiveness of the proposed valid inequalities are summarized in Tables \ref{tab:balNT15}-\ref{tab:unbalNT30}.
Results are presented for the standard formulation (STD), the standard formulation strengthened with valid inequalities (STD+), the three-level lot-sizing based formulation (3LF), and the three-level lot-sizing based formulation strengthened with valid inequalities (3LF+).
In each of these tables, the first column represents the instance group. Next, for each of the formulations, the table presents the number of instances solved to optimality, the average time in seconds (time), and the average number of nodes (nodes) for the instances solved to optimality as well as the average open gap (gap) for the unsolved instances, calculated as $100\times \frac{bestsol-bestbound}{bestsol}$ for each instance in the group. The presence of a '--' in the columns time and nodes indicate that none of the instances in the group were solved to optimality, while its presence in the column gap means that they were all solved to optimality. The last two lines provide, respectively, the average and total sum considering the lines in the table.

Table~\ref{tab:balNT15} shows the results when using STD, 3LF, STD+ and 3LF+ for the balanced instances with $|T|= 15$. All instances with $|R| = 50$ were solved to optimality using all the formulations. It can be observed that 3LF usually presented lower average times, while 3LF+ achieved lower averages for the number of nodes. For the instances with $|R|=100$, STD+ presents lower average times for most of the instance groups, while 3LF+ presents again better results when it comes to the average number of nodes. It is possible to see that STD already started finding difficulties to solve instances of this size to optimality. For those instances with $|R| = 200$, only 3LF+ solved all of them to optimality, and it also obtained the best results regarding the average number of nodes.

Table \ref{tab:unbalNT15} summarizes the results for the unbalanced instances with $|T|=15$. All instances with $|R| = 50$ were solved to optimality with all the formulations. STD and 3LF had similar performance when it comes to the average times, whereas 3LF+ achieved better results when we consider the average number of nodes. For the instances with $|R|=100$, 3LF+ outperformed all other formulations when we consider the average times. For those instances with $|R| = 200$, only 3LF+ was able to solve all of them to optimality.

Table \ref{tab:balNT30} presents the results for the balanced instances with $|T| = 30$. For the instances with $|R| = 50$, only 3LF+ was able to solve all of them to optimality. We can see that STD+ usually presents lower average times and 3LF+ outperforms the others when it comes to the average number of nodes. For those with $|R| = 100$, STD+ solved more instances to optimality than 3LF+, and STD+ and 3LF+ achieved lower optimality gaps for the unsolved instances. For the instances with $|R|=200$, the table shows that none of them were solved to optimality, and STD+ achieved better results when it comes to the open gaps. 

Table \ref{tab:unbalNT30} displays the results for the unbalanced instances with $|T|=30$. For the instances with $|R|=50$, we can see that only 3LF+ solved all of them to optimality. Additionally, 3LF+ achieved the majority of the best average values when considering the time and number of nodes. For those instances with $|R|=100$, 3LF+ was the only formulation to solve some of them to optimality and, besides, achieved lower gaps for most of the unsolved instance groups. None of the instances with $|R|=200$ were solved to optimality. For these instances, the majority of the lower gaps were obtained by 3LF+.

\begin{landscape}

\begin{table}[H]
\caption{Results obtained by STD, 3LF, STD+ and 3LF+ for balanced instances with $|T|=15$.}\label{tab:balNT15}
\centering
\scriptsize
\begin{tabular}{l|rrrr|rrrr|rrrr|rrrr}
  \hline
  & \multicolumn{4}{c|}{STD} & \multicolumn{4}{c|}{3LF} & \multicolumn{4}{c|}{STD+} & \multicolumn{4}{c}{3LF+} \\ 
instancegroup & opt & time & nodes & gap & opt & time & nodes & gap & opt & time & nodes & gap & opt & time & nodes & gap \\ 
  \hline
50\_15\_5\_DD\_DF & 5 & 9.0 & 4617.6 & -- & 5 & 8.9 & 510.6 & -- & 5 & \textbf{4.0} & 1135.2 & -- & 5 & 6.3 & \textbf{20.4} & -- \\ 
  50\_15\_5\_DD\_SF & 5 & 102.4 & 70486.4 & -- & 5 & 20.9 & 1569.2 & -- & 5 & \textbf{4.1} & 1270.2 & -- & 5 & 6.9 & \textbf{27.0} & --\\ 
  50\_15\_5\_SD\_DF & 5 & 11.5 & 5068.0 & -- & 5 & 6.6 & 367.4 & -- & 5 & \textbf{3.8} & 708.6 & -- & 5 & 5.7 & \textbf{26.2} & --\\ 
  50\_15\_5\_SD\_SF & 5 & 83.4 & 61234.0 & -- & 5 & 10.3 & 860.2 & -- & 5 & \textbf{3.9} & 300.4 & -- & 5 & 5.7 & \textbf{22.8} & --\\ 
  50\_15\_10\_DD\_DF & 5 & \textbf{1.3} & 797.0 & -- & 5 & 1.9 & 2.6 & -- & 5 & 3.4 & 129.2 & -- & 5 & 3.2 & \textbf{1.0} & --\\ 
  50\_15\_10\_DD\_SF & 5 & \textbf{3.3} & 1821.8 & -- & 5 & 3.5 & 44.8 & -- & 5 & 3.5 & 187.2 & -- & 5 & 4.0 & \textbf{3.8} & --\\ 
  50\_15\_10\_SD\_DF & 5 & \textbf{2.1} & 2349.4 & -- & 5 & 3.4 & 120.4 & -- & 5 & 3.5 & 268.4 & -- & 5 & 3.7 & \textbf{2.0} & --\\ 
  50\_15\_10\_SD\_SF & 5 & 18.0 & 15720.6 & -- & 5 & 3.2 & 426.0 & -- & 5 & 3.8 & 313.4 & -- & 5 & \textbf{3.1} & \textbf{1.0} & --\\ 
  50\_15\_15\_DD\_DF & 5 & 1.2 & 171.8 & -- & 5 & \textbf{1.0} & \textbf{1.0} & -- & 5 & 3.9 & 142.0 & -- & 5 & 3.3 & \textbf{1.0} & --\\ 
  50\_15\_15\_DD\_SF & 5 & 2.0 & 2148.4 & -- & 5 & \textbf{1.1} & \textbf{1.0} & -- & 5 & 3.7 & 14.0 & -- & 5 & 3.1 & \textbf{1.0} & --\\ 
  50\_15\_15\_SD\_DF & 5 & 1.1 & 172.2 & -- & 5 & \textbf{1.0} & 2.4 & -- & 5 & 3.7 & 111.8 & -- & 5 & {3.2} & \textbf{1.0} & --\\ 
  50\_15\_15\_SD\_SF & 5 & 5.6 & 2515.2 & -- & 5 & \textbf{2.0} & 213.4 & -- & 5 & 3.8 & 101.2 & -- & 5 & {3.2} & \textbf{1.0} & --\\ 
  50\_15\_20\_DD\_DF & 5 & 1.1 & 74.4 & -- & 5 & \textbf{0.8} & \textbf{1.0} & -- & 5 & 3.8 & 51.6 & -- & 5 & {3.1} & \textbf{1.0} & --\\ 
  50\_15\_20\_DD\_SF & 5 & 1.1 & 68.2 & -- & 5 & \textbf{1.0} & \textbf{1.0} & -- & 5 & 3.5 & 3.0 & -- & 5 & {3.0} & \textbf{1.0} & --\\ 
  50\_15\_20\_SD\_DF & 5 & 1.2 & 64.4 & -- & 5 & \textbf{0.8} & \textbf{1.0} & -- & 5 & 3.8 & 29.6 & -- & 5 & {3.0} & \textbf{1.0} & --\\ 
  50\_15\_20\_SD\_SF & 5 & 1.7 & 1007.6 & -- & 5 & \textbf{1.2} & \textbf{1.0} & -- & 5 & 3.4 & 3.8 & -- & 5 & {3.2} & \textbf{1.0} & --\\ 
   \hline

100\_15\_5\_DD\_DF & 5 & 399.9 & 56341.0 & -- & 5 & 78.1 & 1339.4 & -- & 5 & \textbf{25.4} & 9653.6 & -- & 5 & 41.3 & \textbf{21.8} & --\\ 
  100\_15\_5\_DD\_SF & 4 & 1868.4 & 288582.2 & 0.02 & 5 & 575.9 & 19476.2 & -- & 5 & 78.6 & 33150.2 & -- & 5 & \textbf{39.5} & \textbf{47.8} & --\\ 
  100\_15\_5\_SD\_DF & 5 & 292.0 & 65225.2 & -- & 5 & 69.2 & 1276.2 & -- & 5 & \textbf{29.4} & 12777.4 & -- & 5 & 31.1 & \textbf{43.4} & --\\ 
  100\_15\_5\_SD\_SF & 5 & 1249.2 & 497633.0 & -- & 5 & 1025.4 & 349851.4 & -- & 5 & \textbf{21.4} & 7565.4 & -- & 5 & 23.5 & \textbf{89.4} & --\\ 
  100\_15\_10\_DD\_DF & 3 & 2325.5 & 654996.7 & 1.13 & 5 & 136.8 & 4764.6 & -- & 5 & \textbf{17.2} & 5038.6 & -- & 5 & 31.1 & \textbf{50.8} & --\\ 
  100\_15\_10\_DD\_SF & 3 & 1886.2 & 583844.0 & 0.02 & 5 & 261.4 & 17313.0 & -- & 5 & \textbf{11.9} & 2099.8 & -- & 5 & 24.0 & \textbf{12.2} & --\\ 
  100\_15\_10\_SD\_DF & 4 & 494.7 & 184435.5 & 0.54 & 5 & 119.7 & 1955.8 & -- & 5 & 22.1 & 6022.8 & -- & 5 & \textbf{18.3} & \textbf{9.2} & --\\ 
  100\_15\_10\_SD\_SF & 0 & -- & -- & 0.07 & 5 & 299.0 & 52669.0 & -- & 5 & \textbf{6.4} & 589.8 & -- & 5 & 15.0 & \textbf{73.8} & --\\ 
  100\_15\_15\_DD\_DF & 5 & 212.9 & 62496.2 & -- & 5 & 38.0 & 403.8 & -- & 5 & \textbf{6.2} & 701.6 & -- & 5 & 9.0 & \textbf{1.0} & --\\ 
  100\_15\_15\_DD\_SF & 4 & 808.5 & 453046.0 & 0.03 & 5 & 29.7 & 950.0 & -- & 5 & \textbf{6.7} & 1122.4 & -- & 5 & 10.1 & \textbf{1.0} & --\\ 
  100\_15\_15\_SD\_DF & 5 & 576.5 & 214981.0 & -- & 5 & 58.5 & 1225.4 & -- & 5 & \textbf{9.9} & 2170.4 & -- & 5 & 14.0 & \textbf{14.4} & --\\ 
  100\_15\_15\_SD\_SF & 1 & 270.0 & 96429.0 & 0.01 & 5 & 62.4 & 3674.0 & -- & 5 & \textbf{5.9} & 122.8 & -- & 5 & 11.2 & \textbf{11.8} & --\\ 
  100\_15\_20\_DD\_DF & 5 & 77.0 & 24173.0 & -- & 5 & 31.3 & 616.2 & -- & 5 & 14.5 & 5551.4 & -- & 5 & \textbf{13.3} & \textbf{13.2} & --\\ 
  100\_15\_20\_DD\_SF & 5 & 71.5 & 22093.2 & -- & 5 & 15.5 & 153.8 & -- & 5 & 11.1 & 2733.6 & -- & 5 & \textbf{8.8} & \textbf{1.0} & --\\ 
  100\_15\_20\_SD\_DF & 5 & 55.7 & 16229.2 & -- & 5 & 25.4 & 976.2 & -- & 5 & \textbf{7.1} & 1335.2 & -- & 5 & 10.1 & \textbf{4.2} & --\\ 
  100\_15\_20\_SD\_SF & 5 & 435.3 & 308889.2 & -- & 5 & 28.3 & 1281.8 & -- & 5 & \textbf{6.0} & 118.0 & -- & 5 & 8.2 & \textbf{1.0} & --\\ 
   \hline

200\_15\_5\_DD\_DF & 0 & -- & -- & 1.09 & 5 & 1242.9 & 4575.2 & -- & 2 & 2103.4 & 691950.0 & 0.05 & \textbf{5} & 458.5 & \textbf{744.6} & --\\ 
  200\_15\_5\_DD\_SF & 0 & -- & -- & 2.26 & 0 & -- & -- & 2.41 & 0 & -- & -- & 0.07 & \textbf{5} & 1176.6 & \textbf{6776.6} & --\\ 
  200\_15\_5\_SD\_DF & 0 & -- & -- & 0.97 & 4 & 2644.9 & 15413.2 & 0.03 & 2 & 1103.9 & 428699.5 & 0.06 & \textbf{5} & 334.7 & \textbf{621.8} & --\\ 
  200\_15\_5\_SD\_SF & 0 & -- & -- & 1.42 & 1 & 2353.3 & 20857.0 & 0.36 & 4 & 1129.8 & 282235.5 & 0.01 & \textbf{5} & 896.7 & \textbf{8899.6} & --\\ 
  200\_15\_10\_DD\_DF & 0 & -- & -- & 3.25 & 2 & 777.2 & 8539.5 & 1.20 & 3 & 891.9 & 294587.0 & 0.03 & \textbf{5} & 484.8 & \textbf{627.8} & --\\ 
  200\_15\_10\_DD\_SF & 0 & -- & -- & 2.44 & 0 & -- & -- & 2.01 & 4 & 993.8 & 341265.8 & 0.04 & \textbf{5} & 348.7 & \textbf{2.0} & --\\ 
  200\_15\_10\_SD\_DF & 0 & -- & -- & 2.65 & 5 & 1822.0 & 4771.2 & -- & 2 & 226.0 & 79713.0 & 0.04 & \textbf{5} & 356.6 & \textbf{486.4} & --\\ 
  200\_15\_10\_SD\_SF & 0 & -- & -- & 2.08 & 0 & -- & -- & 0.03 & 5 & \textbf{137.8} & 23658.4 & -- & 5 & 211.4 & \textbf{1.0} & --\\ 
  200\_15\_15\_DD\_DF & 0 & -- & -- & 3.77 & 4 & 1807.9 & 5400.2 & 0.03 & 4 & 183.8 & 48555.0 & 0.03 & \textbf{5} & 350.9 & \textbf{192.0} & --\\ 
  200\_15\_15\_DD\_SF & 0 & -- & -- & 4.04 & 4 & 2860.7 & 19021.0 & 0.01 & 3 & 1594.2 & 605558.7 & 0.02 & \textbf{5} & 267.6 & \textbf{1.4} & --\\ 
  200\_15\_15\_SD\_DF & 0 & -- & -- & 3.98 & 4 & 2362.7 & 20079.5 & 3.59 & 3 & 644.7 & 193569.3 & 0.05 & \textbf{5} & 390.3 & \textbf{1579.0} & --\\ 
  200\_15\_15\_SD\_SF & 0 & -- & -- & 3.05 & 1 & 2352.4 & 30020.0 & 0.01 & 5 & \textbf{56.9} & 4867.4 & -- & 5 & 783.0 & \textbf{13235.4} & --\\ 
  200\_15\_20\_DD\_DF & 0 & -- & -- & 3.65 & 3 & 2479.5 & 13489.3 & 0.22 & 5 & 628.1 & 196638.0 & -- & 5 & \textbf{271.5} & \textbf{1320.0} & --\\ 
  200\_15\_20\_DD\_SF & 0 & -- & -- & 2.65 & 2 & 2539.6 & 50112.5 & 0.27 & 5 & \textbf{71.5} & 6053.4 & -- & 5 & 139.8 & \textbf{1.0} & --\\ 
  200\_15\_20\_SD\_DF & 0 & -- & -- & 3.40 & 5 & 2275.9 & 6497.4 & -- & 4 & 1385.1 & 547147.2 & 0.00 & \textbf{5} & 285.8 & \textbf{2157.8} & --\\ 
  200\_15\_20\_SD\_SF & 0 & -- & -- & 2.58 & 2 & 2601.8 & 64801.5 & 0.05 & 5 & \textbf{140.9} & 27578.2 & -- & 5 & 225.7 & \textbf{3733.0} & --\\ 
   \hline
   
        Average &  & 363.5 & 119281.0 & 1.96 &  & 689.8 & 16125.1 & 0.79 &  & 247.5 & 82289.3 & 0.04 &  & 153.2 & 851.8 & -- \\ 
  Total & 144 &  &  &  & 202 &  &  &  & 216 &  &  &  & {\bf 240} &  &  &  \\ 
\hline

\end{tabular}
\end{table}

\begin{table}[H]
\caption{Results obtained by STD, 3LF, STD+ and 3LF+ for unbalanced instances with $|T|=15$.}\label{tab:unbalNT15}
\centering
\scriptsize
\begin{tabular}{l|rrrr|rrrr|rrrr|rrrr}
  \hline
  & \multicolumn{4}{c|}{STD} & \multicolumn{4}{c|}{3LF} & \multicolumn{4}{c|}{STD+} & \multicolumn{4}{c}{3LF+} \\ 
instancegroup & opt & time & nodes & gap & opt & time & nodes & gap & opt & time & nodes & gap & opt & time & nodes & gap \\ 
  \hline
  
50\_15\_5\_DD\_DF & 5 & \textbf{2.2} & 991.0 & -- & 5 & 7.5 & 133.4 & -- & 5 & 4.6 & 1737.2 & -- & 5 & 7.0 & \textbf{1.4} & --\\ 
  50\_15\_5\_DD\_SF & 5 & 7.3 & 2321.2 & -- & 5 & 15.0 & 888.2 & -- & 5 & \textbf{5.1} & 2084.2 & -- & 5 & 5.5 & \textbf{1.0} & --\\ 
  50\_15\_5\_SD\_DF & 5 & \textbf{2.6} & 1275.4 & -- & 5 & 6.6 & 19.0 & -- & 5 & 5.8 & 3702.2 & -- & 5 & 4.4 & \textbf{1.0} & --\\ 
  50\_15\_5\_SD\_SF & 5 & 6.3 & 3736.2 & -- & 5 & 13.6 & 972.8 & -- & 5 & \textbf{5.1} & 1273.2 & -- & 5 & 8.4 & \textbf{389.8} & --\\ 
  50\_15\_10\_DD\_DF & 5 & \textbf{2.0} & 973.0 & -- & 5 & 4.0 & 14.6 & -- & 5 & 3.7 & 414.4 & -- & 5 & 4.0 & \textbf{1.0} & --\\ 
  50\_15\_10\_DD\_SF & 5 & 15.6 & 5303.0 & -- & 5 & 6.5 & 379.6 & -- & 5 & \textbf{4.4} & 1590.8 & -- & 5 & 6.1 & \textbf{4.8} & --\\ 
  50\_15\_10\_SD\_DF & 5 & \textbf{2.3} & 1426.4 & -- & 5 & 6.5 & 129.4 & -- & 5 & 3.7 & 416.2 & -- & 5 & 4.0 & \textbf{3.0} & --\\ 
  50\_15\_10\_SD\_SF & 5 & 20.0 & 7578.4 & -- & 5 & 5.7 & 219.8 & -- & 5 & \textbf{4.0} & 442.6 & -- & 5 & 5.9 & \textbf{13.6} & --\\ 
  50\_15\_15\_DD\_DF & 5 & \textbf{1.1} & 660.0 & -- & 5 & 2.0 & 16.0 & -- & 5 & 4.0 & 98.8 & -- & 5 & 3.6 & \textbf{3.6} & --\\ 
  50\_15\_15\_DD\_SF & 5 & 6.3 & 3952.0 & -- & 5 & \textbf{2.7} & 57.0 & -- & 5 & 3.6 & 123.8 & -- & 5 & 3.2 & \textbf{1.0} & --\\ 
  50\_15\_15\_SD\_DF & 5 & \textbf{1.2} & 720.4 & -- & 5 & 2.7 & 21.6 & -- & 5 & 3.7 & 200.8 & -- & 5 & 3.2 & \textbf{1.0} & --\\ 
  50\_15\_15\_SD\_SF & 5 & 10.1 & 4748.4 & -- & 5 & \textbf{3.7} & 94.4 & -- & 5 & 3.9 & 58.4 & -- & 5 & 3.8 & \textbf{1.0} & --\\ 
  50\_15\_20\_DD\_DF & 5 & 1.1 & 85.6 & -- & 5 & \textbf{0.8} & \textbf{1.0} & -- & 5 & 3.7 & 30.0 & -- & 5 & 3.0 & \textbf{1.0} & --\\ 
  50\_15\_20\_DD\_SF & 5 & 1.4 & 281.4 & -- & 5 & \textbf{1.1} & 19.4 & -- & 5 & 3.6 & 6.4 & -- & 5 & 2.9 & \textbf{1.0} & --\\ 
  50\_15\_20\_SD\_DF & 5 & 1.2 & 652.8 & -- & 5 & \textbf{0.9} & \textbf{1.0} & -- & 5 & 3.7 & 62.2 & -- & 5 & 3.1 & \textbf{1.0} & --\\ 
  50\_15\_20\_SD\_SF & 5 & 2.3 & 696.0 & -- & 5 & \textbf{1.4} & 1.8 & -- & 5 & 3.8 & 28.4 & -- & 5 & 3.1 & \textbf{1.0} & --\\ 
   \hline

100\_15\_5\_DD\_DF & 5 & 31.4 & 10049.8 & -- & 5 & 34.9 & 209.4 & -- & 5 & 394.2 & 333252.6 & -- & 5 & \textbf{21.5} & \textbf{18.2} & --\\ 
  100\_15\_5\_DD\_SF & 5 & 112.7 & 34257.6 & -- & 5 & 34.0 & 1100.2 & -- & 5 & 162.4 & 100392.0 & -- & 5 & \textbf{27.1} & \textbf{95.0} & --\\ 
  100\_15\_5\_SD\_DF & 5 & 22.4 & 5764.6 & -- & 5 & 31.0 & 185.4 & -- & 5 & 141.3 & 105049.2 & -- & 5 & \textbf{14.7} & \textbf{1.0} & --\\ 
  100\_15\_5\_SD\_SF & 5 & 56.2 & 10524.0 & -- & 5 & 65.9 & 4789.0 & -- & 5 & 98.0 & 64757.4 & -- & 5 & \textbf{23.6} & \textbf{1344.6} & --\\ 
  100\_15\_10\_DD\_DF & 5 & 41.5 & 13306.0 & -- & 5 & 44.5 & 147.8 & -- & 5 & 70.2 & 42429.4 & -- & 5 & \textbf{28.2} & \textbf{110.6} & --\\ 
  100\_15\_10\_DD\_SF & 5 & 150.2 & 52508.4 & -- & 5 & 58.9 & 2198.2 & -- & 5 & 69.4 & 30275.8 & -- & 5 & \textbf{29.2} & \textbf{30.2} & --\\ 
  100\_15\_10\_SD\_DF & 5 & 39.1 & 10885.2 & -- & 5 & 32.3 & 330.8 & -- & 5 & 49.6 & 26875.8 & -- & 5 & \textbf{15.9} & \textbf{1.4} & --\\ 
  100\_15\_10\_SD\_SF & 5 & 65.5 & 13061.6 & -- & 5 & 50.5 & 3549.8 & -- & 5 & 25.3 & 10449.2 & -- & 5 & \textbf{17.3} & \textbf{179.4} & --\\ 
  100\_15\_15\_DD\_DF & 5 & 36.0 & 4637.4 & -- & 5 & 53.6 & 662.8 & -- & 5 & 21.3 & 6156.0 & -- & 5 & \textbf{13.5} & \textbf{2.8} & --\\ 
  100\_15\_15\_DD\_SF & 5 & 74.0 & 14396.6 & -- & 5 & 77.7 & 2152.0 & -- & 5 & 25.7 & 6578.0 & -- & 5 & \textbf{27.0} & \textbf{1.8} & --\\ 
  100\_15\_15\_SD\_DF & 5 & 77.5 & 42603.0 & -- & 5 & 39.7 & 426.4 & -- & 5 & 36.5 & 18592.0 & -- & 5 & \textbf{20.0} & \textbf{18.2} & --\\ 
  100\_15\_15\_SD\_SF & 5 & 102.7 & 41805.0 & -- & 5 & 81.0 & 4226.2 & -- & 5 & 26.6 & 10533.6 & -- & 5 & \textbf{22.2} & \textbf{508.6} & --\\ 
  100\_15\_20\_DD\_DF & 5 & 51.7 & 10813.6 & -- & 5 & 33.8 & 721.8 & -- & 5 & 26.7 & 9280.6 & -- & 5 & \textbf{26.1} & \textbf{40.2} & --\\ 
  100\_15\_20\_DD\_SF & 4 & 134.1 & 34781.2 & 0.02 & 5 & 63.6 & 5399.2 & -- & 5 & \textbf{29.8} & 10807.2 & -- & 5 & 41.7 & \textbf{914.0} & --\\ 
  100\_15\_20\_SD\_DF & 5 & 37.4 & 6834.8 & -- & 5 & 31.1 & 435.4 & -- & 5 & 21.6 & 6518.6 & -- & 5 & \textbf{17.2} & \textbf{86.8} & --\\ 
  100\_15\_20\_SD\_SF & 5 & 713.2 & 632343.8 & -- & 5 & 144.4 & 16630.2 & -- & 5 & \textbf{17.9} & 9950.4 & -- & 5 & 21.6 & \textbf{161.6} & --\\ 
   \hline

200\_15\_5\_DD\_DF & 2 & 557.2 & 146939.0 & 0.02 & 5 & 238.5 & 1200.8 & -- & 3 & 576.9 & 215702.7 & 0.04 & 5 & \textbf{155.4} & \textbf{323.8} & --\\ 
  200\_15\_5\_DD\_SF & 0 & -- & -- & 0.03 & 5 & \textbf{298.4} & \textbf{1829.8} & -- & 0 & -- & -- & 0.04 & 5 & 374.6 & 4836.4 & --\\ 
  200\_15\_5\_SD\_DF & 2 & 1550.8 & 503084.5 & 0.03 & 5 & 275.0 & 611.2 & -- & 1 & 126.5 & 36605.0 & 0.04 & 5 & \textbf{102.3} & \textbf{1.0} & --\\ 
  200\_15\_5\_SD\_SF & 1 & 2593.6 & 712717.0 & 0.01 & 5 & 934.9 & \textbf{12154.6} & -- & 3 & 798.1 & 245763.0 & 0.00 & 5 & \textbf{673.4} & 39953.2 & --\\ 
  200\_15\_10\_DD\_DF & 2 & 1219.3 & 264328.5 & 0.03 & 5 & 360.7 & 1327.8 & -- & 1 & 958.1 & 291938.0 & 0.04 & 5 & \textbf{219.9} & \textbf{101.2} & --\\ 
  200\_15\_10\_DD\_SF & 0 & -- & -- & 1.37 & 4 & 2133.8 & 18597.5 & 0.22 & 1 & 2559.0 & 700171.0 & 0.04 & \textbf{5} & 240.0 & \textbf{1349.4} & --\\ 
  200\_15\_10\_SD\_DF & 1 & 2136.4 & 631927.0 & 0.14 & 5 & 830.7 & 4560.4 & -- & 0 & -- & -- & 0.05 & 5 & \textbf{275.1} & \textbf{2232.8} & --\\ 
  200\_15\_10\_SD\_SF & 0 & -- & -- & 0.55 & 5 & 1927.0 & 68226.4 & -- & 5 & 757.2 & 141809.8 & -- & 5 & \textbf{565.5} & \textbf{10755.2} & --\\ 
  200\_15\_15\_DD\_DF & 1 & 1386.2 & 114121.0 & 0.02 & 5 & 947.5 & 3190.6 & -- & 1 & 1195.7 & 415905.0 & 0.03 & 5 & \textbf{221.4} & \textbf{514.8} & --\\ 
  200\_15\_15\_DD\_SF & 1 & 2910.5 & 194040.0 & 1.10 & 4 & 1720.8 & 13740.5 & 0.01 & 1 & 1176.9 & 140399.0 & 0.04 & \textbf{5} & 422.2 & \textbf{3812.4} & --\\ 
  200\_15\_15\_SD\_DF & 1 & 3025.9 & 781870.0 & 0.03 & 5 & 1367.0 & 3352.8 & -- & 1 & 267.0 & 61281.0 & 0.02 & 5 & \textbf{207.2} & \textbf{1168.0} & --\\ 
  200\_15\_15\_SD\_SF & 0 & -- & -- & 0.40 & 4 & 2713.9 & 92533.5 & 0.06 & 4 & 681.2 & 126569.2 & 0.02 & \textbf{5} & 634.2 & \textbf{16261.6} & --\\ 
  200\_15\_20\_DD\_DF & 3 & 1862.2 & 217406.7 & 1.09 & 5 & 1536.8 & 12738.2 & -- & 3 & 839.8 & 259242.7 & 0.04 & 5 & \textbf{347.0} & \textbf{499.0} & --\\ 
  200\_15\_20\_DD\_SF & 0 & -- & -- & 0.35 & 5 & 2394.3 & 30438.8 & -- & 2 & 1453.7 & 518522.5 & 0.02 & 5 & \textbf{211.3} & \textbf{1686.2} & --\\ 
  200\_15\_20\_SD\_DF & 0 & -- & -- & 1.00 & 5 & 1859.3 & 5427.8 & -- & 3 & 507.2 & 112870.0 & 0.02 & 5 & \textbf{243.5} & \textbf{1470.8} & --\\ 
  200\_15\_20\_SD\_SF & 0 & -- & -- & 0.55 & 5 & 1952.3 & 52061.8 & -- & 5 & 1133.6 & 228909.8 & -- & 5 & \textbf{522.1} & \textbf{6484.2} & --\\ 
   \hline
  
   Average &  & 465.1 & 110741.6 & 0.40 &  & 467.7 & 7669.3 & 0.09 &  & 311.2 & 93475.1 & 0.03 &  & 121.9 & 1987.3 & -- \\ 
  
  Total & 173 &  &  &  & 237 &  &  &  & 194 &  &  &  & {\bf 240} &  &  &  \\ \hline
\end{tabular}
\end{table}

\begin{table}[H]
\caption{Results obtained by STD, 3LF, STD+ and 3LF+ for balanced instances with $|T|=30$.}\label{tab:balNT30}
\centering
\scriptsize
\begin{tabular}{l|rrrr|rrrr|rrrr|rrrr}
  \hline
  & \multicolumn{4}{c|}{STD} & \multicolumn{4}{c|}{3LF} & \multicolumn{4}{c|}{STD+} & \multicolumn{4}{c}{3LF+} \\ 
instancegroup & opt & time & nodes & gap & opt & time & nodes & gap & opt & time & nodes & gap & opt & time & nodes & gap \\ 
  \hline
  
50\_30\_5\_DD\_DF & 0 & -- & -- & 11.15 & 2 & 3273.6 & 10868.5 & 6.29 & 5 & \textbf{276.1} & 90345.0 & -- & 5 & 428.8 & \textbf{1456.6} & --\\ 
  50\_30\_5\_DD\_SF & 0 & -- & -- & 16.96 & 0 & -- & -- & 9.70 & 3 & 884.3 & 241763.0 & 0.01 & \textbf{5} & 1185.7 & \textbf{7318.2} & --\\ 
  50\_30\_5\_SD\_DF & 0 & -- & -- & 9.27 & 0 & -- & -- & 7.07 & 5 & \textbf{306.6} & 106711.8 & -- & 5 & 1009.1 & \textbf{8735.4} & --\\ 
  50\_30\_5\_SD\_SF & 0 & -- & -- & 14.98 & 0 & -- & -- & 10.44 & 1 & 778.7 & 83831.0 & 0.24 & \textbf{5} & 1537.7 & \textbf{31528.2} & --\\ 
  50\_30\_10\_DD\_DF & 0 & -- & -- & 9.72 & 0 & -- & -- & 5.25 & 5 & \textbf{121.6} & 9546.2 & -- & 5 & 489.8 & \textbf{3162.6} & --\\ 
  50\_30\_10\_DD\_SF & 0 & -- & -- & 11.91 & 0 & -- & -- & 8.12 & 5 & \textbf{150.5} & 28373.8 & -- & 5 & 433.0 & \textbf{1025.2} & --\\ 
  50\_30\_10\_SD\_DF & 0 & -- & -- & 11.96 & 0 & -- & -- & 7.63 & 5 & \textbf{136.3} & 24732.8 & -- & 5 & 325.6 & \textbf{1207.6} & --\\ 
  50\_30\_10\_SD\_SF & 0 & -- & -- & 10.53 & 0 & -- & -- & 8.67 & 5 & \textbf{113.0} & 11270.4 & -- & 5 & 365.3 & \textbf{1901.6} & --\\ 
  50\_30\_15\_DD\_DF & 0 & -- & -- & 9.74 & 0 & -- & -- & 2.07 & 5 & \textbf{111.5} & 11914.8 & -- & 5 & 238.1 & \textbf{540.4} & --\\ 
  50\_30\_15\_DD\_SF & 0 & -- & -- & 8.91 & 0 & -- & -- & 6.78 & 5 & \textbf{106.7} & 21797.0 & -- & 5 & 206.7 & \textbf{224.8} & --\\ 
  50\_30\_15\_SD\_DF & 0 & -- & -- & 8.92 & 0 & -- & -- & 3.93 & 5 & \textbf{117.0} & 13103.8 & -- & 5 & 198.5 & \textbf{467.2} & --\\ 
  50\_30\_15\_SD\_SF & 0 & -- & -- & 8.23 & 0 & -- & -- & 6.24 & 5 & 210.3 & 71703.6 & -- & 5 & \textbf{166.6} & \textbf{3006.8} & --\\ 
  50\_30\_20\_DD\_DF & 0 & -- & -- & 7.95 & 0 & -- & -- & 4.30 & 5 & \textbf{142.8} & 15597.4 & -- & 5 & 206.4 & \textbf{1606.4} & --\\ 
  50\_30\_20\_DD\_SF & 0 & -- & -- & 6.47 & 0 & -- & -- & 5.47 & 5 & \textbf{103.8} & 9378.8 & -- & 5 & 155.8 & \textbf{3072.8} & --\\ 
  50\_30\_20\_SD\_DF & 0 & -- & -- & 7.68 & 0 & -- & -- & 4.59 & 5 & \textbf{122.4} & 11193.2 & -- & 5 & 143.4 & \textbf{299.0} & --\\ 
  50\_30\_20\_SD\_SF & 0 & -- & -- & 4.06 & 0 & -- & -- & 5.51 & 5 & 150.0 & 46191.6 & -- & 5 & \textbf{121.3} & \textbf{487.8} & --\\ 
   \hline
   
100\_30\_5\_DD\_DF & 0 & -- & -- & 25.94 & 0 & -- & -- & 15.30 & 0 & -- & -- & 1.86 & 0 & -- & -- & \textbf{1.58} \\ 
  100\_30\_5\_DD\_SF & 0 & -- & -- & 26.29 & 0 & -- & -- & 17.55 & 0 & -- & -- & 5.14 & 0 & -- & -- & \textbf{2.90} \\ 
  100\_30\_5\_SD\_DF & 0 & -- & -- & 19.12 & 0 & -- & -- & 16.30 & 0 & -- & -- & \textbf{1.42} & 0 & -- & -- & 1.49 \\ 
  100\_30\_5\_SD\_SF & 0 & -- & -- & 26.68 & 0 & -- & -- & 17.68 & 0 & -- & -- & 4.31 & 0 & -- & -- & \textbf{2.51} \\ 
  100\_30\_10\_DD\_DF & 0 & -- & -- & 28.45 & 0 & -- & -- & 15.95 & 0 & -- & -- & \textbf{2.01} & 0 & -- & -- & 2.83 \\ 
  100\_30\_10\_DD\_SF & 0 & -- & -- & 28.35 & 0 & -- & -- & 18.06 & 0 & -- & -- & 3.85 & 0 & -- & -- & \textbf{3.17} \\ 
  100\_30\_10\_SD\_DF & 0 & -- & -- & 25.74 & 0 & -- & -- & 17.79 & 0 & -- & -- & \textbf{2.31} & 0 & -- & -- & 3.70 \\ 
  100\_30\_10\_SD\_SF & 0 & -- & -- & 27.43 & 0 & -- & -- & 18.30 & 0 & -- & -- & \textbf{3.40} & 0 & -- & -- & 5.17 \\ 
  100\_30\_15\_DD\_DF & 0 & -- & -- & 25.32 & 0 & -- & -- & 14.51 & \textbf{2} & 2012.8 & 192681.0 & 1.60 & 0 & -- & -- & 1.96 \\ 
  100\_30\_15\_DD\_SF & 0 & -- & -- & 25.26 & 0 & -- & -- & 16.99 & 0 & -- & -- & \textbf{1.01} & 0 & -- & -- & 1.49 \\ 
  100\_30\_15\_SD\_DF & 0 & -- & -- & 22.61 & 0 & -- & -- & 15.95 & 0 & -- & -- & \textbf{2.26} & 0 & -- & -- & 3.76 \\ 
  100\_30\_15\_SD\_SF & 0 & -- & -- & 22.57 & 0 & -- & -- & 17.06 & 0 & -- & -- & 1.46 & 0 & -- & -- & \textbf{0.88} \\ 
  100\_30\_20\_DD\_DF & 0 & -- & -- & 22.77 & 0 & -- & -- & 13.68 & \textbf{2} & 1965.8 & 253676.5 & 1.84 & 0 & -- & -- & 1.18 \\ 
  100\_30\_20\_DD\_SF & 0 & -- & -- & 24.26 & 0 & -- & -- & 16.78 & 0 & -- & -- & \textbf{0.23} & 0 & -- & -- & 0.47 \\ 
  100\_30\_20\_SD\_DF & 0 & -- & -- & 21.11 & 0 & -- & -- & 13.46 & 0 & -- & -- & 1.16 & 0 & -- & -- & \textbf{0.94} \\ 
  100\_30\_20\_SD\_SF & 0 & -- & -- & 19.42 & 0 & -- & -- & 14.48 & 2 & \textbf{2077.0} & 167045.0 & 0.99 & 2 & 3135.6 & \textbf{1133.0} & 0.60 \\ 
   \hline
   
200\_30\_5\_DD\_DF & 0 & -- & -- & 30.30 & 0 & -- & -- & 25.68 & 0 & -- & -- & \textbf{7.69} & 0 & -- & -- & 16.27 \\ 
  200\_30\_5\_DD\_SF & 0 & -- & -- & 33.46 & 0 & -- & -- & 30.70 & 0 & -- & -- & \textbf{10.65} & 0 & -- & -- & 21.83 \\ 
  200\_30\_5\_SD\_DF & 0 & -- & -- & 28.16 & 0 & -- & -- & 22.27 & 0 & -- & -- & \textbf{8.96} & 0 & -- & -- & 10.19 \\ 
  200\_30\_5\_SD\_SF & 0 & -- & -- & 30.82 & 0 & -- & -- & 23.63 & 0 & -- & -- & \textbf{9.09} & 0 & -- & -- & 14.94 \\ 
  200\_30\_10\_DD\_DF & 0 & -- & -- & 34.76 & 0 & -- & -- & 22.42 & 0 & -- & -- & \textbf{10.94} & 0 & -- & -- & 12.81 \\ 
  200\_30\_10\_DD\_SF & 0 & -- & -- & 34.94 & 0 & -- & -- & 22.73 & 0 & -- & -- & 13.39 & 0 & -- & -- & \textbf{12.75} \\ 
  200\_30\_10\_SD\_DF & 0 & -- & -- & 32.61 & 0 & -- & -- & 25.91 & 0 & -- & -- & \textbf{11.96} & 0 & -- & -- & 12.81 \\ 
  200\_30\_10\_SD\_SF & 0 & -- & -- & 34.11 & 0 & -- & -- & 34.22 & 0 & -- & -- & \textbf{12.02} & 0 & -- & -- & 20.20 \\ 
  200\_30\_15\_DD\_DF & 0 & -- & -- & 34.33 & 0 & -- & -- & 23.61 & 0 & -- & -- & \textbf{10.18} & 0 & -- & -- & 12.13 \\ 
  200\_30\_15\_DD\_SF & 0 & -- & -- & 36.14 & 0 & -- & -- & 24.33 & 0 & -- & -- & \textbf{12.17} & 0 & -- & -- & 13.83 \\ 
  200\_30\_15\_SD\_DF & 0 & -- & -- & 32.98 & 0 & -- & -- & 26.72 & 0 & -- & -- & \textbf{10.62} & 0 & -- & -- & 13.57 \\ 
  200\_30\_15\_SD\_SF & 0 & -- & -- & 33.59 & 0 & -- & -- & 27.08 & 0 & -- & -- & \textbf{10.94} & 0 & -- & -- & 13.88 \\ 
  200\_30\_20\_DD\_DF & 0 & -- & -- & 33.19 & 0 & -- & -- & 23.16 & 0 & -- & -- & \textbf{8.65} & 0 & -- & -- & 10.73 \\ 
  200\_30\_20\_DD\_SF & 0 & -- & -- & 35.01 & 0 & -- & -- & 25.51 & 0 & -- & -- & \textbf{10.29} & 0 & -- & -- & 13.68 \\ 
  200\_30\_20\_SD\_DF & 0 & -- & -- & 31.23 & 0 & -- & -- & 25.88 & 0 & -- & -- & \textbf{8.28} & 0 & -- & -- & 11.21 \\ 
  200\_30\_20\_SD\_SF & 0 & -- & -- & 32.68 & 0 & -- & -- & 26.95 & 0 & -- & -- & \textbf{10.41} & 0 & -- & -- & 13.51 \\ 
   \hline
   
     Average &  & -- & -- & 22.46 &  & 3273.6 & 10868.5 & 16.10 &  & 520.4 & 74255.6 & 5.92 &  & 608.7 & 3951.4 & 8.09 \\ 
     
  Total & 0 &  &  &  & 2 &  &  &  & 80 &  &  &  & {\bf 82} &  &  &  \\ 
   \hline
   
\end{tabular}
\end{table}

\begin{table}[H]
\caption{Results obtained by STD, 3LF, STD+ and 3LF+ for unbalanced instances with $|T|=30$.}\label{tab:unbalNT30}
\centering
\scriptsize
\begin{tabular}{l|rrrr|rrrr|rrrr|rrrr}
  \hline
  
  & \multicolumn{4}{c|}{STD} & \multicolumn{4}{c|}{3LF} & \multicolumn{4}{c|}{STD+} & \multicolumn{4}{c}{3LF+} \\ 
instancegroup & opt & time & nodes & gap & opt & time & nodes & gap & opt & time & nodes & gap & opt & time & nodes & gap \\ 
  \hline
  
50\_30\_5\_DD\_DF & 0 & -- & -- & 0.79 & 5 & 1612.7 & 6601.8 & -- & 5 & 256.9 & 96907.0 & -- & 5 & \textbf{143.6} & 37.8 & --\\ 
  50\_30\_5\_DD\_SF & 0 & -- & -- & 7.35 & 2 & 2868.1 & 26976.5 & 2.27 & 5 & 791.5 & 290772.0 & -- & 5 & \textbf{289.4} & 4263.8 & --\\ 
  50\_30\_5\_SD\_DF & 4 & 2545.6 & 555166.8 & 0.53 & 5 & 1098.4 & 10613.6 & -- & 5 & 374.9 & 201810.0 & -- & 5 & \textbf{241.5} & 1001.2 & --\\ 
  50\_30\_5\_SD\_SF & 0 & -- & -- & 3.31 & 3 & 2440.1 & 37522.7 & 7.19 & 5 & 370.0 & 65685.2 & -- & 5 & \textbf{235.7} & 2036.4 & --\\ 
  50\_30\_10\_DD\_DF & 0 & -- & -- & 7.85 & 2 & 3014.9 & 12211.5 & 4.05 & 5 & 272.7 & 64824.2 & -- & 5 & \textbf{217.2} & 627.0 & --\\ 
  50\_30\_10\_DD\_SF & 0 & -- & -- & 13.12 & 0 & -- & -- & 6.75 & 5 & 883.4 & 208723.2 & -- & 5 & \textbf{314.9} & 1522.8 & --\\ 
  50\_30\_10\_SD\_DF & 0 & -- & -- & 7.57 & 2 & 2454.3 & 19436.0 & 5.02 & 5 & 389.0 & 114201.0 & -- & 5 & \textbf{220.7} & 1394.8 & --\\ 
  50\_30\_10\_SD\_SF & 0 & -- & -- & 10.35 & 0 & -- & -- & 5.62 & 4 & 771.0 & 184995.8 & 0.82 & \textbf{5} & 382.6 & 4983.6 & --\\ 
  50\_30\_15\_DD\_DF & 0 & -- & -- & 10.15 & 0 & -- & -- & 3.83 & 5 & \textbf{138.0} & 22171.4 & -- & 5 & 237.9 & 821.6 & --\\ 
  50\_30\_15\_DD\_SF & 0 & -- & -- & 9.30 & 0 & -- & -- & 7.09 & 5 & 296.3 & 53832.4 & -- & 5 & \textbf{287.2} & 3821.8 & --\\ 
  50\_30\_15\_SD\_DF & 0 & -- & -- & 10.25 & 2 & 2935.3 & 42355.0 & 4.62 & 5 & \textbf{145.3} & 31264.8 & -- & 5 & 200.9 & 795.8 & --\\ 
  50\_30\_15\_SD\_SF & 0 & -- & -- & 7.92 & 0 & -- & -- & 4.64 & 5 & 321.3 & 72995.8 & -- & 5 & \textbf{259.8} & 1517.0 & --\\ 
  50\_30\_20\_DD\_DF & 0 & -- & -- & 8.97 & 0 & -- & -- & 5.07 & 5 & \textbf{112.8} & 10728.0 & -- & 5 & 182.5 & 1943.0 & --\\ 
  50\_30\_20\_DD\_SF & 0 & -- & -- & 6.28 & 0 & -- & -- & 5.57 & 5 & \textbf{113.3} & 24250.0 & -- & 5 & 171.0 & 292.4 & --\\ 
  50\_30\_20\_SD\_DF & 0 & -- & -- & 7.65 & 0 & -- & -- & 3.94 & 5 & \textbf{119.7} & 13094.6 & -- & 5 & 144.7 & 611.6 & --\\ 
  50\_30\_20\_SD\_SF & 0 & -- & -- & 3.72 & 0 & -- & -- & 3.85 & 5 & 141.1 & 25953.2 & -- & 5 & \textbf{103.5} & 480.0 & --\\ 
   \hline
   
100\_30\_5\_DD\_DF & 0 & -- & -- & 15.34 & 0 & -- & -- & 8.11 & 0 & -- & -- & 0.04 & \textbf{5} & 1382.3 & 2476.2 & --\\ 
  100\_30\_5\_DD\_SF & 0 & -- & -- & 22.41 & 0 & -- & -- & 11.82 & 0 & -- & -- & 0.50 & 0 & -- & -- & \textbf{0.27} \\ 
  100\_30\_5\_SD\_DF & 0 & -- & -- & 13.82 & 0 & -- & -- & 10.46 & 0 & -- & -- & 0.25 & \textbf{4} & 2193.1 & 4393.8 & 0.05 \\ 
  100\_30\_5\_SD\_SF & 0 & -- & -- & 19.96 & 0 & -- & -- & 11.18 & 0 & -- & -- & 0.71 & \textbf{2} & 2827.0 & 21253.0 & 0.63 \\ 
  100\_30\_10\_DD\_DF & 0 & -- & -- & 22.80 & 0 & -- & -- & 11.52 & 0 & -- & -- & 0.44 & \textbf{5} & 2228.8 & 2495.6 & --\\ 
  100\_30\_10\_DD\_SF & 0 & -- & -- & 24.68 & 0 & -- & -- & 15.37 & 0 & -- & -- & 3.28 & 0 & -- & -- & \textbf{0.35} \\ 
  100\_30\_10\_SD\_DF & 0 & -- & -- & 23.24 & 0 & -- & -- & 12.74 & 0 & -- & -- & 0.70 & \textbf{2} & 2492.5 & 2952.5 & 0.35 \\ 
  100\_30\_10\_SD\_SF & 0 & -- & -- & 24.65 & 0 & -- & -- & 11.87 & 0 & -- & -- & \textbf{2.15} & 0 & -- & -- & 2.54 \\ 
  100\_30\_15\_DD\_DF & 0 & -- & -- & 21.22 & 0 & -- & -- & 12.92 & 0 & -- & -- & 0.88 & \textbf{1} & 2364.2 & 910.0 & 0.18 \\ 
  100\_30\_15\_DD\_SF & 0 & -- & -- & 24.14 & 0 & -- & -- & 14.87 & 0 & -- & -- & 4.60 & 0 & -- & -- & \textbf{0.54} \\ 
  100\_30\_15\_SD\_DF & 0 & -- & -- & 17.84 & 0 & -- & -- & 11.82 & 0 & -- & -- & 1.63 & \textbf{2} & 2056.6 & 4933.0 & 0.38 \\ 
  100\_30\_15\_SD\_SF & 0 & -- & -- & 22.25 & 0 & -- & -- & 12.42 & 0 & -- & -- & 4.56 & 0 & -- & -- & \textbf{0.82} \\ 
  100\_30\_20\_DD\_DF & 0 & -- & -- & 18.34 & 0 & -- & -- & 12.17 & 0 & -- & -- & 1.73 & \textbf{2} & 2744.5 & 8036.0 & 0.43 \\ 
  100\_30\_20\_DD\_SF & 0 & -- & -- & 20.47 & 0 & -- & -- & 14.64 & 0 & -- & -- & 4.79 & 0 & -- & -- & \textbf{0.34} \\ 
  100\_30\_20\_SD\_DF & 0 & -- & -- & 19.04 & 0 & -- & -- & 13.03 & 0 & -- & -- & 1.19 & \textbf{1} & 2085.7 & 1313.0 & 0.18 \\ 
  100\_30\_20\_SD\_SF & 0 & -- & -- & 20.18 & 0 & -- & -- & 13.40 & 0 & -- & -- & 4.47 & 0 & -- & -- & \textbf{0.58} \\ 
   \hline
   
200\_30\_5\_DD\_DF & 0 & -- & -- & 27.67 & 0 & -- & -- & 18.12 & 0 & -- & -- & \textbf{7.36} & 0 & -- & -- & 10.01 \\ 
  200\_30\_5\_DD\_SF & 0 & -- & -- & 27.42 & 0 & -- & -- & 26.58 & 0 & -- & -- & \textbf{4.58} & 0 & -- & -- & 15.01 \\ 
  200\_30\_5\_SD\_DF & 0 & -- & -- & 27.15 & 0 & -- & -- & 22.12 & 0 & -- & -- & \textbf{7.71} & 0 & -- & -- & 11.79 \\ 
  200\_30\_5\_SD\_SF & 0 & -- & -- & 28.38 & 0 & -- & -- & 24.10 & 0 & -- & -- & \textbf{1.74} & 0 & -- & -- & 6.67 \\ 
  200\_30\_10\_DD\_DF & 0 & -- & -- & 28.79 & 0 & -- & -- & 20.19 & 0 & -- & -- & 8.29 & 0 & -- & -- & \textbf{7.52} \\ 
  200\_30\_10\_DD\_SF & 0 & -- & -- & 32.04 & 0 & -- & -- & 23.73 & 0 & -- & -- & 9.86 & 0 & -- & -- & \textbf{9.44} \\ 
  200\_30\_10\_SD\_DF & 0 & -- & -- & 30.32 & 0 & -- & -- & 23.26 & 0 & -- & -- & \textbf{8.25} & 0 & -- & -- & 10.52 \\ 
  200\_30\_10\_SD\_SF & 0 & -- & -- & 31.41 & 0 & -- & -- & 23.63 & 0 & -- & -- & \textbf{9.34} & 0 & -- & -- & 10.97 \\ 
  200\_30\_15\_DD\_DF & 0 & -- & -- & 28.95 & 0 & -- & -- & 20.86 & 0 & -- & -- & 9.85 & 0 & -- & -- & \textbf{8.66} \\ 
  200\_30\_15\_DD\_SF & 0 & -- & -- & 31.66 & 0 & -- & -- & 23.85 & 0 & -- & -- & 10.79 & 0 & -- & -- & \textbf{10.10} \\ 
  200\_30\_15\_SD\_DF & 0 & -- & -- & 29.29 & 0 & -- & -- & 20.29 & 0 & -- & -- & 9.81 & 0 & -- & -- & \textbf{7.71} \\ 
  200\_30\_15\_SD\_SF & 0 & -- & -- & 30.72 & 0 & -- & -- & 23.22 & 0 & -- & -- & 10.00 & 0 & -- & -- & \textbf{9.46} \\ 
  200\_30\_20\_DD\_DF & 0 & -- & -- & 28.88 & 0 & -- & -- & 18.99 & 0 & -- & -- & 9.53 & 0 & -- & -- & \textbf{7.15} \\ 
  200\_30\_20\_DD\_SF & 0 & -- & -- & 31.52 & 0 & -- & -- & 22.14 & 0 & -- & -- & 11.31 & 0 & -- & -- & \textbf{10.19} \\ 
  200\_30\_20\_SD\_DF & 0 & -- & -- & 28.30 & 0 & -- & -- & 21.95 & 0 & -- & -- & \textbf{9.29} & 0 & -- & -- & 9.39 \\ 
  200\_30\_20\_SD\_SF & 0 & -- & -- & 30.67 & 0 & -- & -- & 22.83 & 0 & -- & -- & 11.10 & 0 & -- & -- & \textbf{9.22} \\ 
   \hline
   
   Average &  & 2545.6 & 555166.8 & 19.14 &  & 2346.3 & 22245.3 & 13.56 &  & 343.6 & 92638.0 & 5.20 &  & 960.3 & 2996.5 & 5.38 \\
   
  Total & 4 &  &  &  & 21 &  &  &  & 79 &  &  &  & {\bf 104} &  &  &  \\ 
   \hline
   
\end{tabular}
\end{table}

\end{landscape}

\subsection{Results analyzing the effectiveness of preprocessing and multi-start randomized bottom-up {dynamic programming-based heuristic}}
\label{sec:speedup}

Tables \ref{tab:mcbalNT15}-\ref{tab:mcunbalNT30} summarize the computational experiments performed to assess the effectiveness of the proposed preprocessing and multi-start randomized bottom-up {dynamic programming-based heuristic}, and how their combination compare against solely using the multi-commodity formulation. Results are displayed for the multi-commodity formulation (MC), the heuristic (DPH), and the combination of the heuristic with the preprocessed multi-commodity formulation (DPH-pMC).

In each of these tables, the first column represents the instance group. Columns 2-6 give, for MC, the average best-encountered solution value (best), the number of instances solved to optimality (opt), the average time in seconds (time), and the average number of nodes (nodes) for the instances solved to optimality as well as the average open gap (gap) for the unsolved instances. Columns 7-10 present, for DPH, the best and average solution values (considering the $it_{max}$ executions), the time in seconds and the deviation from the best known solution value (gap$_{b^*}$), given as $gap_{b^*} = 100\times \frac{best- b^*}{b^*}$, where $b^*$ denotes the best known solution value. Columns 11-15 give, for DPH-pMC, the best-encountered solution value (best), the number of instances solved to optimality (opt), the average time in seconds (time), and the average number of nodes (nodes) for the instances solved to optimality, the average open gap (gap) for the unsolved instances, and the percentual reduction considering the potential variables for reduction (red). This last value is calculated as  $\textrm{red} = 100\times \frac{np}{pot}$, where $np$ is the number of removed variables using preprocessing and $pot$ is the number of potential variables to be removed, i.e., $pot = |R| \sum_{k=1}^{|T|} (|T|-k) = |R| \times \frac{|T| \times (|T|-1)}{2}$.

Table \ref{tab:mcbalNT15} shows the results using MC, DPH and DPH-pMC for the balanced instances with $|T|=15$. All instances were solved to optimality using both MC and DPH-pMC, with the remark that DPH-pMC presents lower average times for all but one instance group. The average number of nodes is very similar for MC and DPH-pMC, with nearly all of them being at most one. DPH achieved solutions with an average gap of 6.4\% in 0.8 seconds on average.

Table \ref{tab:mcunbalNT15} displays the results for unbalanced instances with $|T| = 15$. All instances were solved using MC and DPH-pMC. DPH-pMC presents lower average times and similar averages for the number of nodes when compared to MC. Note that MC reaches lower average times for only 12 instance groups. DPH achieved solutions with an average gap of 4.7\% in 0.8 seconds on average.

Table \ref{tab:mcbalNT30} exhibits the results for balanced instances with $|T| = 30$. All instances were solved using MC and DPH-pMC, with DPH-pMC presenting lower average times. Differently from what happened to the instances with 15 periods, the differences in average times for these instances is very considerable. MC only achieved lower average times for two instance groups. For 14 instance groups, the average times of DPH-pMC were less than half of those of MC. MC and DPH-pMC present similar averages for the number of nodes.
DPH achieved solutions with an average gap of 8.9\% in 1.5 seconds on average.

Table \ref{tab:mcunbalNT30} summarizes the results for unbalanced instances with $|T| = 30$. DPH-pMC solved all the 240 instances to optimality, while 10 of them remained unsolved when using MC. It can be observed that all instances which were not solved by MC have $|R|=200$. Considering the instances which could be solved to optimality by both formulations, DPH-pMC presents lower average times. Note that MC could not finish solving the linear relaxation in one instance of group $200\_30\_10\_SD\_DF$, which explains the 100\% gap. Additionally, MC could not solve any of the instances in group $200\_30\_20\_DD\_SF$ to optimality. The average number of nodes was similar for MC and DPH-pMC. DPH achieved solutions with an average gap of 6.7\% in 1.5 seconds on average.

\begin{landscape}

\begin{table}[H]
\caption{Results obtained by MC and DPH-pMC for balanced instances with $|T|=15$.}\label{tab:mcbalNT15}
\centering
\scriptsize
\begin{tabular}{l|rrrrr|rrrr|rrrrrr}
  \hline
  & \multicolumn{5}{c|}{MC} & \multicolumn{4}{c|}{DPH} & \multicolumn{6}{c}{DPH-pMC} \\ 
instancegroup & best &opt & time & nodes & gap & best & avg & time & gap$_{b^*}$ & best & opt & time & nodes & gap & red \\ 
  \hline
  
50\_15\_5\_DD\_DF & 178781.0 & 5 &{\bf 0.9} & 0.0 & -- & 184493.0 & 185434.3 & 0.5 & 3.2 & 178781.0 & 5 &{\bf 0.9} & 0.0 & -- & 62.1 \\ 
  50\_15\_5\_DD\_SF & 176582.2 & 5 & {\bf 0.9} & 0.0 & -- & 185734.5 & 188182.1 & 0.5 & 5.2 & 176582.2 & 5 & { 1.0} & 0.0 & -- & 56.0 \\ 
  50\_15\_5\_SD\_DF & 177723.0 & 5 & {\bf 0.9} & 0.0 & -- & 184359.6 & 186141.2 & 0.5 & 3.7 & 177723.0 & 5 & { 1.0} & 0.0 & -- & 56.1 \\ 
  50\_15\_5\_SD\_SF & 186936.5 & 5 &{\bf 1.6} & 0.8 & -- & 195966.8 & 199006.7 & 0.5 & 4.9 & 186936.5 & 5 &{\bf 1.6} & 0.8 & -- & 45.3 \\ 
  50\_15\_10\_DD\_DF & 203922.7 & 5 & 1.2 & 0.0 & -- & 209584.1 & 211410.9 & 0.5 & 2.8 & 203922.7 & 5 & {\bf 1.0} & 0.0 & -- & 63.4 \\ 
  50\_15\_10\_DD\_SF & 211618.1 & 5 & 1.3 & 0.0 & -- & 218486.1 & 220861.0 & 0.5 & 3.3 & 211618.1 & 5 & {\bf 1.1} & 0.0 & -- & 56.6 \\ 
  50\_15\_10\_SD\_DF & 204215.4 & 5 & 1.4 & 0.0 & -- & 209873.2 & 211822.2 & 0.5 & 2.8 & 204215.4 & 5 & {\bf 1.0} & 0.0 & -- & 56.4 \\ 
  50\_15\_10\_SD\_SF & 227013.2 & 5 & 1.3 & 0.0 & -- & 233924.8 & 235590.8 & 0.5 & 3.0 & 227013.2 & 5 & {\bf 1.2} & 0.0 & -- & 45.5 \\ 
  50\_15\_15\_DD\_DF & 224337.5 & 5 & 1.2 & 0.0 & -- & 229290.7 & 230559.4 & 0.5 & 2.2 & 224337.5 & 5 & {\bf 0.8} & 0.0 & -- & 62.5 \\ 
  50\_15\_15\_DD\_SF & 229780.6 & 5 & 1.2 & 0.0 & -- & 240579.5 & 242858.4 & 0.5 & 4.7 & 229780.6 & 5 & {\bf 0.9} & 0.0 & -- & 51.7 \\ 
  50\_15\_15\_SD\_DF & 222972.6 & 5 & 1.1 & 0.0 & -- & 228627.7 & 229987.9 & 0.6 & 2.5 & 222972.6 & 5 & {\bf 0.9} & 0.0 & -- & 53.3 \\ 
  50\_15\_15\_SD\_SF & 239593.8 & 5 & 1.2 & 0.0 & -- & 248231.3 & 250586.2 & 0.6 & 3.6 & 239593.8 & 5 & {\bf 0.9} & 0.0 & -- & 46.5 \\ 
  50\_15\_20\_DD\_DF & 247537.1 & 5 & 0.9 & 0.0 & -- & 255933.1 & 257977.4 & 0.6 & 3.4 & 247537.1 & 5 & {\bf 0.8} & 0.0 & -- & 64.1 \\ 
  50\_15\_20\_DD\_SF & 251127.4 & 5 & 1.0 & 0.0 & -- & 261344.8 & 264410.6 & 0.6 & 4.1 & 251127.4 & 5 & {\bf 0.8} & 0.0 & -- & 56.3 \\ 
  50\_15\_20\_SD\_DF & 255120.9 & 5 & 0.9 & 0.0 & -- & 262770.3 & 264394.2 & 0.6 & 3.0 & 255120.9 & 5 & {\bf 0.8} & 0.0 & -- & 61.6 \\ 
  50\_15\_20\_SD\_SF & 258538.2 & 5 & 1.0 & 0.0 & -- & 268121.7 & 271474.9 & 0.6 & 3.7 & 258538.2 & 5 & {\bf 0.8} & 0.0 & -- & 47.1 \\ \hline
  
  100\_15\_5\_DD\_DF & 257450.6 & 5 & 2.1 & 0.2 & -- & 271103.0 & 274256.9 & 0.7 & 5.3 & 257450.6 & 5 & {\bf 1.8} & 0.2 & -- & 63.1 \\ 
  100\_15\_5\_DD\_SF & 258441.1 & 5 & 3.1 & 0.6 & -- & 273882.3 & 279822.2 & 0.7 & 6.0 & 258441.1 & 5 & {\bf 2.6} & 0.6 & -- & 55.0 \\ 
  100\_15\_5\_SD\_DF & 250053.2 & 5 & 2.6 & 0.4 & -- & 265333.5 & 269227.8 & 0.7 & 6.1 & 250053.2 & 5 & {\bf 2.0} & 0.4 & -- & 56.3 \\ 
  100\_15\_5\_SD\_SF & 273377.8 & 5 & 4.6 & 1.0 & -- & 289682.5 & 297755.5 & 0.7 & 6.0 & 273377.8 & 5 & {\bf 3.4} & 1.0 & -- & 47.2 \\ 
  100\_15\_10\_DD\_DF & 287589.0 & 5 & 2.4 & 0.0 & -- & 306051.9 & 309919.3 & 0.7 & 6.4 & 287589.0 & 5 & {\bf 1.9} & 0.0 & -- & 61.8 \\ 
  100\_15\_10\_DD\_SF & 295625.5 & 5 & 2.6 & 0.2 & -- & 320989.6 & 327593.5 & 0.8 & 8.6 & 295625.5 & 5 & {\bf 2.1} & 0.2 & -- & 55.1 \\ 
  100\_15\_10\_SD\_DF & 286684.3 & 5 & 2.2 & 0.0 & -- & 307282.6 & 311062.5 & 0.8 & 7.2 & 286684.3 & 5 & {\bf 1.9} & 0.0 & -- & 56.2 \\ 
  100\_15\_10\_SD\_SF & 305063.7 & 5 & 3.5 & 0.8 & -- & 329739.0 & 338807.2 & 0.8 & 8.1 & 305063.7 & 5 & {\bf 3.3} & 0.8 & -- & 48.0 \\ 
  100\_15\_15\_DD\_DF & 315099.2 & 5 & 3.0 & 0.0 & -- & 340864.6 & 345628.3 & 0.8 & 8.2 & 315099.2 & 5 & {\bf 2.2} & 0.0 & -- & 63.7 \\ 
  100\_15\_15\_DD\_SF & 332588.2 & 5 & 3.3 & 0.0 & -- & 356873.0 & 363993.0 & 0.8 & 7.3 & 332588.2 & 5 & {\bf 2.6} & 0.0 & -- & 54.6 \\ 
  100\_15\_15\_SD\_DF & 311527.6 & 5 & 3.2 & 0.0 & -- & 333838.3 & 337887.6 & 0.8 & 7.2 & 311527.6 & 5 & {\bf 2.5} & 0.0 & -- & 53.5 \\ 
  100\_15\_15\_SD\_SF & 339451.5 & 5 & 4.3 & 0.6 & -- & 366118.4 & 375251.4 & 0.8 & 7.9 & 339451.5 & 5 & {\bf 3.7} & 0.6 & -- & 46.6 \\ 
  100\_15\_20\_DD\_DF & 352544.1 & 5 & 2.8 & 0.2 & -- & 375571.8 & 379842.0 & 0.8 & 6.5 & 352544.1 & 5 & {\bf 2.2} & 0.2 & -- & 61.3 \\ 
  100\_15\_20\_DD\_SF & 364628.0 & 5 & 2.9 & 0.0 & -- & 385364.6 & 392673.2 & 0.8 & 5.7 & 364628.0 & 5 & {\bf 2.6} & 0.0 & -- & 54.9 \\ 
  100\_15\_20\_SD\_DF & 349764.8 & 5 & 2.9 & 0.0 & -- & 369510.7 & 373736.9 & 0.8 & 5.6 & 349764.8 & 5 & {\bf 2.3} & 0.0 & -- & 56.9 \\ 
  100\_15\_20\_SD\_SF & 363580.4 & 5 & 3.2 & 0.0 & -- & 384890.3 & 393956.0 & 0.8 & 5.9 & 363580.4 & 5 & {\bf 2.5} & 0.0 & -- & 48.2 \\ \hline
  
  200\_15\_5\_DD\_DF & 373598.8 & 5 & 37.5 & 1.4 & -- & 399634.4 & 403589.8 & 1.1 & 7.0 & 373598.8 & 5 & {\bf 16.4} & 1.4 & -- & 61.9 \\ 
  200\_15\_5\_DD\_SF & 398678.4 & 5 & {\bf 24.8} & 4.6 & -- & 431384.0 & 439316.8 & 1.2 & 8.2 & 398678.4 & 5 & 26.8 & 4.6 & -- & 53.6 \\ 
  200\_15\_5\_SD\_DF & 378962.9 & 5 & 4.8 & 0.0 & -- & 409214.3 & 415681.9 & 1.2 & 8.0 & 378962.9 & 5 & {\bf 4.3} & 0.0 & -- & 57.8 \\ 
  200\_15\_5\_SD\_SF & 399004.6 & 5 & 30.5 & 7.4 & -- & 432607.9 & 442556.3 & 1.2 & 8.4 & 399004.6 & 5 & {\bf 29.5} & 7.4 & -- & 46.1 \\ 
  200\_15\_10\_DD\_DF & 428573.9 & 5 & 6.4 & 0.2 & -- & 463247.5 & 468906.7 & 1.2 & 8.1 & 428573.9 & 5 & {\bf 5.4} & 0.2 & -- & 61.2 \\ 
  200\_15\_10\_DD\_SF & 437410.8 & 5 & 5.3 & 0.0 & -- & 483698.3 & 493814.7 & 1.2 & 10.6 & 437410.8 & 5 & {\bf 4.6} & 0.0 & -- & 54.4 \\ 
  200\_15\_10\_SD\_DF & 422430.2 & 5 & 4.6 & 0.0 & -- & 460606.6 & 465490.0 & 1.2 & 9.0 & 422430.2 & 5 & {\bf 4.2} & 0.0 & -- & 57.4 \\ 
  200\_15\_10\_SD\_SF & 436691.7 & 5 & 5.7 & 0.0 & -- & 485388.8 & 496283.9 & 1.2 & 11.2 & 436691.7 & 5 & {\bf 5.1} & 0.0 & -- & 46.0 \\ 
  200\_15\_15\_DD\_DF & 468277.5 & 5 & 5.4 & 0.0 & -- & 506805.0 & 513336.9 & 1.2 & 8.2 & 468277.5 & 5 & {\bf 4.8} & 0.0 & -- & 62.6 \\ 
  200\_15\_15\_DD\_SF & 470491.5 & 5 & 5.8 & 0.0 & -- & 513548.7 & 526556.1 & 1.2 & 9.2 & 470491.5 & 5 & {\bf 5.0} & 0.0 & -- & 54.1 \\ 
  200\_15\_15\_SD\_DF & 470951.5 & 5 & 17.5 & 1.0 & -- & 510964.1 & 517099.6 & 1.2 & 8.5 & 470951.5 & 5 & {\bf 13.6} & 1.0 & -- & 57.7 \\ 
  200\_15\_15\_SD\_SF & 477258.7 & 5 & 5.6 & 0.0 & -- & 530345.5 & 542446.8 & 1.2 & 11.1 & 477258.7 & 5 & {\bf 5.1} & 0.0 & -- & 44.1 \\ 
  200\_15\_20\_DD\_DF & 504939.6 & 5 & 6.9 & 0.0 & -- & 542696.9 & 548928.9 & 1.2 & 7.5 & 504939.6 & 5 & {\bf 5.9} & 0.0 & -- & 61.8 \\ 
  200\_15\_20\_DD\_SF & 536448.5 & 5 & 6.8 & 0.0 & -- & 587327.0 & 597830.2 & 1.2 & 9.5 & 536448.5 & 5 & {\bf 6.3} & 0.0 & -- & 54.8 \\ 
  200\_15\_20\_SD\_DF & 508508.5 & 5 & 12.4 & 0.6 & -- & 545859.8 & 551704.5 & 1.2 & 7.3 & 508508.5 & 5 & {\bf 11.1} & 0.6 & -- & 55.8 \\ 
  200\_15\_20\_SD\_SF & 540006.2 & 5 & 22.5 & 1.8 & -- & 591287.8 & 604311.5 & 1.2 & 9.5 & 540006.2 & 5 & {\bf 18.2} & 1.0 & -- & 47.1 \\ \hline
  
  Average &  &  & 5.6 & 0.5 & -- &  &  & 0.8 & 6.4 &  &  & {\bf 4.6} & 0.4 & -- & {55.1} \\ 
  Total &  & 240 &  &  &  &  &  &  &  &  & 240 &  &  &  &  \\ 
   \hline
\end{tabular}
\end{table}

\begin{table}[H]
\caption{Results obtained by MC and DPH-pMC for unbalanced instances with $|T|=15$.}\label{tab:mcunbalNT15}
\centering
\scriptsize
\begin{tabular}{l|rrrrr|rrrr|rrrrrr}
  \hline
  & \multicolumn{5}{c|}{MC} & \multicolumn{4}{c|}{DPH} & \multicolumn{6}{c}{DPH-pMC} \\ 
instancegroup & best &opt & time & nodes & gap & best & avg & time & gap$_{b^*}$ & best & opt & time & nodes & gap & red \\ 
  \hline
50\_15\_5\_DD\_DF & 169377.4 & 5 & {\bf 0.9} & 0.0 & -- & 173687.7 & 174894.4 & 0.5 & 2.6 & 169377.4 & 5 & {\bf 0.9} & 0.0 & -- & 62.1 \\ 
  50\_15\_5\_DD\_SF & 166826.0 & 5 & {\bf 0.9} & 0.0 & -- & 172447.6 & 174763.2 & 0.5 & 3.4 & 166826.0 & 5 & 1.0 & 0.0 & -- & 56.0 \\ 
  50\_15\_5\_SD\_DF & 168797.0 & 5 & {\bf 0.8} & 0.0 & -- & 174562.1 & 175478.0 & 0.5 & 3.4 & 168797.0 & 5 & 0.9 & 0.0 & -- & 56.1 \\ 
  50\_15\_5\_SD\_SF & 175296.4 & 5 & {\bf 2.1} & 1.0 & -- & 181295.5 & 183633.4 & 0.5 & 3.4 & 175296.4 & 5 & 2.2 & 1.0 & -- & 45.3 \\ 
  50\_15\_10\_DD\_DF & 196472.0 & 5 & {\bf 0.9} & 0.0 & -- & 203630.9 & 205147.5 & 0.5 & 3.6 & 196472.0 & 5 &{\bf 0.9} & 0.0 & -- & 63.4 \\ 
  50\_15\_10\_DD\_SF & 199116.1 & 5 & {\bf 1.0} & 0.2 & -- & 204951.1 & 207234.2 & 0.5 & 2.9 & 199116.1 & 5 & 1.1 & 0.2 & -- & 56.6 \\ 
  50\_15\_10\_SD\_DF & 195540.5 & 5 & {\bf 0.8} & 0.0 & -- & 201959.3 & 202976.6 & 0.5 & 3.3 & 195540.5 & 5 & 0.9 & 0.0 & -- & 56.4 \\ 
  50\_15\_10\_SD\_SF & 213232.5 & 5 & {\bf 1.4} & 0.8 & -- & 220455.8 & 222335.0 & 0.6 & 3.3 & 213232.5 & 5 & 1.9 & 0.8 & -- & 45.5 \\ 
  50\_15\_15\_DD\_DF & 217119.7 & 5 & 0.9 & 0.0 & -- & 225328.5 & 227096.1 & 0.6 & 3.8 & 217119.7 & 5 & {\bf 0.8} & 0.0 & -- & 62.5 \\ 
  50\_15\_15\_DD\_SF & 222164.8 & 5 & 1.0 & 0.2 & -- & 228565.8 & 231042.9 & 0.6 & 2.9 & 222164.8 & 5 & {\bf 0.8} & 0.2 & -- & 51.7 \\ 
  50\_15\_15\_SD\_DF & 217499.8 & 5 & 0.9 & 0.0 & -- & 224194.1 & 225708.4 & 0.6 & 3.1 & 217499.8 & 5 & {\bf 0.8} & 0.0 & -- & 53.3 \\ 
  50\_15\_15\_SD\_SF & 232679.2 & 5 & 1.1 & 0.2 & -- & 240566.0 & 243253.1 & 0.6 & 3.4 & 232679.2 & 5 & {\bf 0.9} & 0.2 & -- & 46.5 \\ 
  50\_15\_20\_DD\_DF & 245279.1 & 5 & 0.9 & 0.0 & -- & 255114.2 & 257156.8 & 0.6 & 4.0 & 245279.1 & 5 & {\bf 0.8} & 0.0 & -- & 64.1 \\ 
  50\_15\_20\_DD\_SF & 249786.3 & 5 & 1.0 & 0.0 & -- & 257564.8 & 260053.6 & 0.6 & 3.1 & 249786.3 & 5 & {\bf 0.8} & 0.0 & -- & 56.3 \\ 
  50\_15\_20\_SD\_DF & 253839.5 & 5 & 1.0 & 0.0 & -- & 261452.6 & 263384.7 & 0.6 & 3.0 & 253839.5 & 5 & {\bf 0.8} & 0.0 & -- & 61.6 \\ 
  50\_15\_20\_SD\_SF & 258041.0 & 5 & 1.0 & 0.0 & -- & 267538.1 & 270195.3 & 0.6 & 3.7 & 258041.0 & 5 & {\bf 0.9} & 0.0 & -- & 47.1 \\ \hline
  
  100\_15\_5\_DD\_DF & 243090.8 & 5 & 2.1 & 0.0 & -- & 252872.9 & 254735.8 & 0.7 & 4.0 & 243090.8 & 5 & {\bf 1.8} & 0.0 & -- & 63.1 \\ 
  100\_15\_5\_DD\_SF & 242367.0 & 5 & 5.3 & 1.0 & -- & 253282.0 & 257306.2 & 0.7 & 4.5 & 242367.0 & 5 & {\bf 5.0} & 1.0 & -- & 55.0 \\ 
  100\_15\_5\_SD\_DF & 234854.7 & 5 & 2.1 & 0.2 & -- & 248561.9 & 250284.6 & 0.8 & 5.8 & 234854.7 & 5 & {\bf 1.9} & 0.2 & -- & 56.3 \\ 
  100\_15\_5\_SD\_SF & 256258.7 & 5 & 5.3 & 1.0 & -- & 268781.5 & 272816.6 & 0.8 & 4.9 & 256258.7 & 5 & {\bf 5.3} & 1.0 & -- & 47.2 \\ 
  100\_15\_10\_DD\_DF & 271470.0 & 5 & 2.2 & 0.2 & -- & 283902.7 & 288124.1 & 0.7 & 4.6 & 271470.0 & 5 & {\bf 1.9} & 0.2 & -- & 61.8 \\ 
  100\_15\_10\_DD\_SF & 280634.8 & 5 & 2.6 & 0.2 & -- & 294801.4 & 299741.9 & 0.8 & 5.1 & 280634.8 & 5 & {\bf 2.2} & 0.2 & -- & 55.1 \\ 
  100\_15\_10\_SD\_DF & 270399.5 & 5 & 2.0 & 0.0 & -- & 284104.0 & 287013.8 & 0.8 & 5.1 & 270399.5 & 5 & {\bf 1.8} & 0.0 & -- & 56.2 \\ 
  100\_15\_10\_SD\_SF & 289608.0 & 5 & 5.0 & 1.0 & -- & 304635.4 & 310745.9 & 0.8 & 5.2 & 289608.0 & 5 & {\bf 3.2} & 1.0 & -- & 48.0 \\ 
  100\_15\_15\_DD\_DF & 298589.2 & 5 & 2.0 & 0.0 & -- & 311715.4 & 315344.5 & 0.8 & 4.4 & 298589.2 & 5 & {\bf 1.7} & 0.0 & -- & 63.7 \\ 
  100\_15\_15\_DD\_SF & 317730.3 & 5 & 2.2 & 0.0 & -- & 331021.8 & 336841.0 & 0.8 & 4.2 & 317730.3 & 5 & {\bf 2.2} & 0.0 & -- & 54.6 \\ 
  100\_15\_15\_SD\_DF & 293950.1 & 5 & 2.0 & 0.0 & -- & 308703.3 & 312153.8 & 0.8 & 5.0 & 293950.1 & 5 & {\bf 1.8} & 0.0 & -- & 53.5 \\ 
  100\_15\_15\_SD\_SF & 321460.9 & 5 & 4.3 & 1.0 & -- & 336355.5 & 342644.2 & 0.8 & 4.6 & 321460.9 & 5 & {\bf 3.2} & 1.0 & -- & 46.6 \\ 
  100\_15\_20\_DD\_DF & 341631.2 & 5 & 2.0 & 0.0 & -- & 357838.6 & 361735.8 & 0.8 & 4.7 & 341631.2 & 5 & {\bf 1.7} & 0.0 & -- & 61.3 \\ 
  100\_15\_20\_DD\_SF & 356235.2 & 5 & 3.2 & 0.6 & -- & 372378.9 & 377566.4 & 0.8 & 4.5 & 356235.2 & 5 & {\bf 2.8} & 0.6 & -- & 54.9 \\ 
  100\_15\_20\_SD\_DF & 336504.2 & 5 & 2.1 & 0.0 & -- & 351676.0 & 355794.1 & 0.8 & 4.5 & 336504.2 & 5 & {\bf 1.8} & 0.0 & -- & 56.9 \\ 
  100\_15\_20\_SD\_SF & 348827.7 & 5 & 4.0 & 1.0 & -- & 363300.8 & 370264.8 & 0.8 & 4.2 & 348827.7 & 5 & {\bf 3.3} & 1.0 & -- & 48.2 \\ \hline
  
  200\_15\_5\_DD\_DF & 355263.6 & 5 & 5.1 & 0.0 & -- & 375261.2 & 378990.9 & 1.2 & 5.6 & 355263.6 & 5 & {\bf 4.9} & 0.0 & -- & 61.9 \\ 
  200\_15\_5\_DD\_SF & 378131.2 & 5 & {\bf 16.2} & 2.2 & -- & 398685.2 & 404177.8 & 1.2 & 5.4 & 378131.2 & 5 & 17.9 & 2.2 & -- & 53.6 \\ 
  200\_15\_5\_SD\_DF & 360583.7 & 5 & {\bf 4.5} & 0.0 & -- & 383390.7 & 388747.3 & 1.2 & 6.3 & 360583.7 & 5 & 4.6 & 0.0 & -- & 57.8 \\ 
  200\_15\_5\_SD\_SF & 377067.9 & 5 & 18.7 & 2.0 & -- & 401018.3 & 407509.5 & 1.2 & 6.4 & 377067.9 & 5 & {\bf 18.0} & 1.6 & -- & 46.1 \\ 
  200\_15\_10\_DD\_DF & 403827.7 & 5 & 4.9 & 0.0 & -- & 426410.3 & 429850.4 & 1.2 & 5.6 & 403827.7 & 5 & {\bf 4.3} & 0.0 & -- & 61.2 \\ 
  200\_15\_10\_DD\_SF & 416943.7 & 5 & {\bf 17.3} & 4.0 & -- & 440795.2 & 446803.4 & 1.2 & 5.7 & 416943.7 & 5 & 20.9 & 4.0 & -- & 54.4 \\ 
  200\_15\_10\_SD\_DF & 398810.8 & 5 & 24.8 & 1.2 & -- & 423256.3 & 426822.5 & 1.2 & 6.1 & 398810.8 & 5 & {\bf 13.5} & 1.2 & -- & 57.4 \\ 
  200\_15\_10\_SD\_SF & 414113.2 & 5 & {\bf 30.1} & 11.0 & -- & 439086.7 & 445643.7 & 1.2 & 6.0 & 414113.2 & 5 & 30.6 & 18.0 & -- & 46.0 \\ 
  200\_15\_15\_DD\_DF & 436027.1 & 5 & 7.2 & 0.2 & -- & 460854.2 & 465466.3 & 1.2 & 5.7 & 436027.1 & 5 & {\bf 6.1} & 0.2 & -- & 62.6 \\ 
  200\_15\_15\_DD\_SF & 443242.5 & 5 & {\bf 20.6} & 3.4 & -- & 469023.2 & 475648.5 & 1.2 & 5.8 & 443242.5 & 5 & 23.4 & 3.4 & -- & 54.1 \\ 
  200\_15\_15\_SD\_DF & 436966.3 & 5 & 6.9 & 0.4 & -- & 464004.1 & 469006.6 & 1.2 & 6.2 & 436966.3 & 5 & {\bf 5.2} & 0.4 & -- & 57.7 \\ 
  200\_15\_15\_SD\_SF & 449052.5 & 5 & 25.1 & 16.6 & -- & 480337.4 & 487077.1 & 1.2 & 7.0 & 449052.5 & 5 & {\bf 21.2} & 7.4 & -- & 44.1 \\ 
  200\_15\_20\_DD\_DF & 473455.7 & 5 & {\bf 15.1} & 0.2 & -- & 503505.2 & 508702.9 & 1.2 & 6.3 & 473455.7 & 5 & 20.7 & 0.8 & -- & 61.8 \\ 
  200\_15\_20\_DD\_SF & 507914.0 & 5 & 11.3 & 1.0 & -- & 539278.4 & 548248.4 & 1.2 & 6.2 & 507914.0 & 5 & {\bf 9.7} & 0.6 & -- & 54.8 \\ 
  200\_15\_20\_SD\_DF & 479440.7 & 5 & 26.9 & 1.8 & -- & 510544.4 & 516478.1 & 1.2 & 6.5 & 479440.7 & 5 & {\bf 18.4} & 1.4 & -- & 55.8 \\ 
  200\_15\_20\_SD\_SF & 506277.9 & 5 & 24.0 & 9.0 & -- & 540235.5 & 549489.1 & 1.2 & 6.7 & 506277.9 & 5 & {\bf 18.8} & 2.6 & -- & 47.1 \\ \hline
  
  Average &  &  & 6.7 & 1.3 & -- &  &  & 0.8 & 4.7 &  &  & {\bf 6.2} & 1.1 & -- & {55.1}  \\ 
  Total &  & 240 &  &  &  &  &  &  &  &  & 240 &  &  &  &  \\ 
   \hline
\end{tabular}
\end{table}

\begin{table}[H]
\caption{Results obtained by MC and DPH-pMC for balanced instances with $|T|=30$.}\label{tab:mcbalNT30}
\centering
\scriptsize
\begin{tabular}{l|rrrrr|rrrr|rrrrrr}
  \hline
  & \multicolumn{5}{c|}{MC} & \multicolumn{4}{c|}{DPH} & \multicolumn{6}{c}{DPH-pMC} \\ 
instancegroup & best &opt & time & nodes & gap & best & avg & time & gap$_{b^*}$ & best & opt & time & nodes & gap & red \\ 
  \hline
  
50\_30\_5\_DD\_DF & 338701.3 & 5 & 7.7 & 0.2 & -- & 362718.4 & 368521.6 & 0.8 & 7.1 & 338701.3 & 5 & {\bf 4.2} & 0.2 & -- & 76.3 \\ 
  50\_30\_5\_DD\_SF & 369959.3 & 5 & 77.9 & 3.8 & -- & 390501.8 & 399648.0 & 0.8 & 5.6 & 369959.3 & 5 & {\bf 60.4} & 5.8 & -- & 72.7 \\ 
  50\_30\_5\_SD\_DF & 339759.9 & 5 & 12.8 & 0.6 & -- & 360490.3 & 366564.0 & 0.8 & 6.1 & 339759.9 & 5 & {\bf 6.8} & 0.6 & -- & 74.9 \\ 
  50\_30\_5\_SD\_SF & 375954.2 & 5 & {\bf 83.1} & 18.0 & -- & 395406.3 & 407832.0 & 0.8 & 5.2 & 375954.2 & 5 & 85.6 & 28.8 & -- & 61.3 \\ 
  50\_30\_10\_DD\_DF & 392039.9 & 5 & 13.8 & 0.0 & -- & 418206.7 & 424044.4 & 0.8 & 6.7 & 392039.9 & 5 & {\bf 6.4} & 0.0 & -- & 76.5 \\ 
  50\_30\_10\_DD\_SF & 407797.4 & 5 & 14.8 & 0.0 & -- & 442710.1 & 453062.3 & 0.8 & 8.6 & 407797.4 & 5 & {\bf 6.8} & 0.0 & -- & 72.1 \\ 
  50\_30\_10\_SD\_DF & 393979.1 & 5 & 14.0 & 0.0 & -- & 421496.3 & 427143.8 & 0.8 & 7.0 & 393979.1 & 5 & {\bf 6.0} & 0.0 & -- & 72.1 \\ 
  50\_30\_10\_SD\_SF & 429701.0 & 5 & 27.4 & 0.6 & -- & 464157.9 & 476830.0 & 0.9 & 8.0 & 429701.0 & 5 & {\bf 15.2} & 0.6 & -- & 62.1 \\ 
  50\_30\_15\_DD\_DF & 435577.6 & 5 & 13.0 & 0.2 & -- & 465661.4 & 472284.2 & 0.9 & 6.9 & 435577.6 & 5 & {\bf 7.3} & 0.2 & -- & 76.2 \\ 
  50\_30\_15\_DD\_SF & 453096.7 & 5 & 11.6 & 0.0 & -- & 489517.7 & 500375.5 & 0.9 & 8.0 & 453096.7 & 5 & {\bf 6.5} & 0.0 & -- & 73.4 \\ 
  50\_30\_15\_SD\_DF & 432574.3 & 5 & 15.0 & 0.4 & -- & 465765.4 & 473012.1 & 0.9 & 7.7 & 432574.3 & 5 & {\bf 7.9} & 0.4 & -- & 72.8 \\ 
  50\_30\_15\_SD\_SF & 455741.0 & 5 & 11.7 & 0.0 & -- & 494557.6 & 506917.0 & 0.9 & 8.5 & 455741.0 & 5 & {\bf 7.3} & 0.0 & -- & 60.0 \\ 
  50\_30\_20\_DD\_DF & 472504.6 & 5 & 9.2 & 0.0 & -- & 504575.3 & 510688.4 & 0.9 & 6.8 & 472504.6 & 5 & {\bf 6.0} & 0.0 & -- & 76.3 \\ 
  50\_30\_20\_DD\_SF & 503023.8 & 5 & 10.1 & 0.0 & -- & 539547.1 & 550921.3 & 0.9 & 7.2 & 503023.8 & 5 & {\bf 6.1} & 0.0 & -- & 74.2 \\ 
  50\_30\_20\_SD\_DF & 475292.7 & 5 & 8.8 & 0.0 & -- & 508916.8 & 515947.2 & 0.9 & 7.1 & 475292.7 & 5 & {\bf 5.8} & 0.0 & -- & 74.4 \\ 
  50\_30\_20\_SD\_SF & 504533.7 & 5 & 13.1 & 0.2 & -- & 546303.0 & 557853.1 & 0.9 & 8.3 & 504533.7 & 5 & {\bf 8.6} & 0.2 & -- & 59.8 \\ \hline
  
  100\_30\_5\_DD\_DF & 498386.9 & 5 & 81.9 & 1.8 & -- & 537965.4 & 546695.9 & 1.3 & 7.9 & 498386.9 & 5 & {\bf 58.0} & 1.8 & -- & 76.5 \\ 
  100\_30\_5\_DD\_SF & 526280.4 & 5 & 148.0 & 18.2 & -- & 565161.1 & 576626.1 & 1.3 & 7.4 & 526280.4 & 5 & {\bf 107.8} & 17.6 & -- & 71.9 \\ 
  100\_30\_5\_SD\_DF & 501113.4 & 5 & 108.2 & 2.6 & -- & 543217.0 & 550657.1 & 1.3 & 8.4 & 501113.4 & 5 & {\bf 79.4} & 2.0 & -- & 73.2 \\ 
  100\_30\_5\_SD\_SF & 526452.2 & 5 & {\bf 144.4} & 20.6 & -- & 562327.7 & 578890.4 & 1.3 & 6.8 & 526452.2 & 5 & 170.0 & 48.6 & -- & 63.1 \\ 
  100\_30\_10\_DD\_DF & 573540.1 & 5 & 21.6 & 0.0 & -- & 623550.4 & 630947.7 & 1.3 & 8.7 & 573540.1 & 5 & {\bf 12.1} & 0.0 & -- & 77.3 \\ 
  100\_30\_10\_DD\_SF & 614187.4 & 5 & 68.4 & 4.6 & -- & 676769.3 & 691218.3 & 1.3 & 10.2 & 614187.4 & 5 & {\bf 48.1} & 2.2 & -- & 71.7 \\ 
  100\_30\_10\_SD\_DF & 565162.5 & 5 & 157.3 & 2.2 & -- & 616561.8 & 625727.2 & 1.3 & 9.1 & 565162.5 & 5 & {\bf 103.9} & 2.6 & -- & 70.9 \\ 
  100\_30\_10\_SD\_SF & 605602.9 & 5 & 96.7 & 12.0 & -- & 660288.3 & 682739.4 & 1.3 & 9.0 & 605602.9 & 5 & {\bf 71.3} & 10.8 & -- & 61.9 \\ 
  100\_30\_15\_DD\_DF & 630630.6 & 5 & 50.5 & 0.0 & -- & 685246.3 & 695455.3 & 1.3 & 8.7 & 630630.6 & 5 & {\bf 13.1} & 0.0 & -- & 80.0 \\ 
  100\_30\_15\_DD\_SF & 665812.1 & 5 & 103.0 & 0.2 & -- & 736500.9 & 751381.5 & 1.4 & 10.6 & 665812.1 & 5 & {\bf 28.7} & 0.2 & -- & 72.1 \\ 
  100\_30\_15\_SD\_DF & 621037.8 & 5 & 194.3 & 1.0 & -- & 675491.6 & 683373.1 & 1.4 & 8.8 & 621037.8 & 5 & {\bf 43.0} & 1.0 & -- & 71.8 \\ 
  100\_30\_15\_SD\_SF & 651822.5 & 5 & 122.3 & 29.4 & -- & 711095.7 & 738337.6 & 1.4 & 9.1 & 651822.5 & 5 & {\bf 79.6} & 21.8 & -- & 62.5 \\ 
  100\_30\_20\_DD\_DF & 676555.6 & 5 & 36.4 & 0.0 & -- & 736012.0 & 744644.7 & 1.4 & 8.8 & 676555.6 & 5 & {\bf 18.5} & 0.0 & -- & 76.9 \\ 
  100\_30\_20\_DD\_SF & 710104.4 & 5 & 39.6 & 0.0 & -- & 784867.0 & 799465.2 & 1.4 & 10.5 & 710104.4 & 5 & {\bf 22.2} & 0.0 & -- & 71.0 \\ 
  100\_30\_20\_SD\_DF & 683630.2 & 5 & 35.7 & 0.0 & -- & 743074.0 & 753442.6 & 1.4 & 8.7 & 683630.2 & 5 & {\bf 18.9} & 0.0 & -- & 73.2 \\ 
  100\_30\_20\_SD\_SF & 711220.4 & 5 & 90.5 & 10.2 & -- & 781146.0 & 805555.2 & 1.4 & 9.8 & 711220.4 & 5 & {\bf 59.2} & 3.4 & -- & 60.8 \\ \hline
  
  200\_30\_5\_DD\_DF & 741971.6 & 5 & 83.8 & 3.0 & -- & 810557.9 & 820845.6 & 2.3 & 9.2 & 741971.6 & 5 & {\bf 81.8} & 4.6 & -- & 78.3 \\ 
  200\_30\_5\_DD\_SF & 756470.3 & 5 & 1265.6 & 77.0 & -- & 825926.9 & 841833.7 & 2.3 & 9.2 & 756470.3 & 5 & {\bf 977.6} & 61.2 & -- & 71.3 \\ 
  200\_30\_5\_SD\_DF & 741675.3 & 5 & 89.8 & 4.2 & -- & 816112.5 & 825673.3 & 2.3 & 10.0 & 741675.3 & 5 & {\bf 77.1} & 4.2 & -- & 73.3 \\ 
  200\_30\_5\_SD\_SF & 780264.2 & 5 & 1451.0 & 65.4 & -- & 840709.5 & 865829.7 & 2.3 & 7.8 & 780264.2 & 5 & {\bf 1303.9} & 68.4 & -- & 60.9 \\ 
  200\_30\_10\_DD\_DF & 847135.9 & 5 & 103.8 & 4.4 & -- & 938434.1 & 949781.2 & 2.3 & 10.8 & 847135.9 & 5 & {\bf 80.1} & 4.4 & -- & 77.0 \\ 
  200\_30\_10\_DD\_SF & 894079.0 & 5 & 524.1 & 7.6 & -- & 1002234.7 & 1021132.3 & 2.3 & 12.1 & 894079.0 & 5 & {\bf 446.5} & 7.6 & -- & 71.8 \\ 
  200\_30\_10\_SD\_DF & 851998.3 & 5 & 133.8 & 4.4 & -- & 942194.5 & 953674.9 & 2.3 & 10.6 & 851998.3 & 5 & {\bf 116.1} & 4.8 & -- & 74.2 \\ 
  200\_30\_10\_SD\_SF & 894647.4 & 5 & 514.9 & 7.2 & -- & 991177.9 & 1022939.4 & 2.3 & 10.8 & 894647.4 & 5 & {\bf 289.5} & 6.8 & -- & 63.0 \\ 
  200\_30\_15\_DD\_DF & 926834.5 & 5 & 338.1 & 1.2 & -- & 1029138.0 & 1040510.7 & 2.3 & 11.0 & 926834.5 & 5 & {\bf 50.6} & 0.8 & -- & 77.3 \\ 
  200\_30\_15\_DD\_SF & 947725.1 & 5 & 459.9 & 1.2 & -- & 1071044.8 & 1092251.2 & 2.3 & 13.0 & 947725.1 & 5 & {\bf 103.2} & 1.2 & -- & 69.6 \\ 
  200\_30\_15\_SD\_DF & 940214.5 & 5 & 351.7 & 2.2 & -- & 1045404.0 & 1056371.8 & 2.3 & 11.2 & 940214.5 & 5 & {\bf 123.4} & 2.6 & -- & 73.6 \\ 
  200\_30\_15\_SD\_SF & 977290.3 & 5 & 262.0 & 0.0 & -- & 1090713.2 & 1129582.0 & 2.3 & 11.6 & 977290.3 & 5 & {\bf 40.5} & 0.0 & -- & 60.7 \\ 
  200\_30\_20\_DD\_DF & 988005.4 & 5 & 184.5 & 0.6 & -- & 1095861.9 & 1108400.4 & 2.3 & 10.9 & 988005.4 & 5 & {\bf 76.2} & 0.6 & -- & 77.7 \\ 
  200\_30\_20\_DD\_SF & 1030026.3 & 5 & 189.6 & 0.0 & -- & 1161879.3 & 1185745.1 & 2.4 & 12.8 & 1030026.3 & 5 & {\bf 78.5} & 0.0 & -- & 71.6 \\ 
  200\_30\_20\_SD\_DF & 979052.6 & 5 & 233.8 & 0.6 & -- & 1089293.3 & 1101685.0 & 2.3 & 11.3 & 979052.7 & 5 & {\bf 114.8} & 0.6 & -- & 73.7 \\ 
  200\_30\_20\_SD\_SF & 1039374.6 & 5 & 166.2 & 0.0 & -- & 1130762.2 & 1197921.5 & 2.4 & 8.8 & 1039374.6 & 5 & {\bf 79.0} & 0.0 & -- & 63.1 \\ \hline
  
  Average &  &  & 170.7 & 6.4 & -- &  &  & 1.5 & 8.9 &  &  & {\bf 108.7} & 6.6 & -- & {71.0} \\ 
  Total &  & 240 &  &  &  &  &  &  &  &  & 240 &  &  &  &  \\ 
   \hline
\end{tabular}
\end{table}

\begin{table}[H]
\caption{Results obtained by MC and DPH-pMC for unbalanced instances with ${|T|}=30$.}\label{tab:mcunbalNT30}
\centering
\scriptsize
\begin{tabular}{l|rrrrr|rrrr|rrrrrr}
  \hline
  & \multicolumn{5}{c|}{MC} & \multicolumn{4}{c|}{DPH} & \multicolumn{6}{c}{DPH-pMC} \\ 
instancegroup & best &opt & time & nodes & gap & best & avg & time & gap$_{b^*}$ & best & opt & time & nodes & gap & red \\ 
  \hline
  
50\_30\_5\_DD\_DF & 318427.9 & 5 & 6.4 & 0.0 & -- & 336483.7 & 339823.4 & 0.8 & 5.7 & 318427.9 & 5 & {\bf 4.3} & 0.0 & -- & 76.3 \\ 
  50\_30\_5\_DD\_SF & 345570.3 & 5 & 61.7 & 1.6 & -- & 360913.8 & 367721.7 & 0.8 & 4.5 & 345570.3 & 5 & {\bf 14.5} & 0.8 & -- & 72.7 \\ 
  50\_30\_5\_SD\_DF & 316360.1 & 5 & 9.4 & 0.2 & -- & 333803.2 & 339674.2 & 0.8 & 5.5 & 316360.1 & 5 & {\bf 5.2} & 0.2 & -- & 74.9 \\ 
  50\_30\_5\_SD\_SF & 352094.6 & 5 & 104.0 & 25.2 & -- & 371961.8 & 379128.5 & 0.8 & 5.7 & 352094.6 & 5 & {\bf 22.7} & 11.4 & -- & 61.3 \\ 
  50\_30\_10\_DD\_DF & 373732.7 & 5 & 12.1 & 0.2 & -- & 394931.6 & 400218.5 & 0.9 & 5.7 & 373732.7 & 5 & {\bf 5.8} & 0.2 & -- & 76.5 \\ 
  50\_30\_10\_DD\_SF & 391861.6 & 5 & 24.5 & 1.2 & -- & 415140.1 & 422714.6 & 0.9 & 5.9 & 391861.6 & 5 & {\bf 8.6} & 0.4 & -- & 72.1 \\ 
  50\_30\_10\_SD\_DF & 375957.6 & 5 & 6.4 & 0.0 & -- & 400071.0 & 404819.4 & 0.9 & 6.4 & 375957.6 & 5 & {\bf 4.6} & 0.0 & -- & 72.1 \\ 
  50\_30\_10\_SD\_SF & 414796.8 & 5 & {\bf 22.5} & 10.8 & -- & 437595.9 & 448016.2 & 0.9 & 5.5 & 414796.8 & 5 & 51.2 & 9.4 & -- & 62.1 \\ 
  50\_30\_15\_DD\_DF & 426065.9 & 5 & 11.1 & 0.4 & -- & 451045.5 & 458720.4 & 0.9 & 5.9 & 426065.9 & 5 & {\bf 7.1} & 0.4 & -- & 76.2 \\ 
  50\_30\_15\_DD\_SF & 449340.1 & 5 & 13.0 & 0.4 & -- & 478382.1 & 487911.0 & 0.9 & 6.5 & 449340.1 & 5 & {\bf 7.9} & 0.4 & -- & 73.4 \\ 
  50\_30\_15\_SD\_DF & 422631.3 & 5 & 9.7 & 0.2 & -- & 451760.1 & 458520.5 & 0.9 & 6.9 & 422631.3 & 5 & {\bf 6.7} & 0.2 & -- & 72.8 \\ 
  50\_30\_15\_SD\_SF & 450270.2 & 5 & 15.1 & 0.4 & -- & 481762.4 & 492309.9 & 0.9 & 7.0 & 450270.2 & 5 & {\bf 9.0} & 0.4 & -- & 60.0 \\ 
  50\_30\_20\_DD\_DF & 470573.6 & 5 & 10.2 & 0.2 & -- & 499848.2 & 506893.4 & 0.9 & 6.2 & 470573.6 & 5 & {\bf 5.6} & 0.2 & -- & 76.3 \\ 
  50\_30\_20\_DD\_SF & 500931.9 & 5 & 11.5 & 0.2 & -- & 536977.6 & 547682.1 & 0.9 & 7.2 & 500931.9 & 5 & {\bf 7.5} & 0.2 & -- & 74.2 \\ 
  50\_30\_20\_SD\_DF & 474358.5 & 5 & 8.6 & 0.0 & -- & 506004.1 & 513251.8 & 0.9 & 6.7 & 474358.5 & 5 & {\bf 5.6} & 0.0 & -- & 74.4 \\ 
  50\_30\_20\_SD\_SF & 501618.9 & 5 & 9.4 & 0.0 & -- & 541032.7 & 553791.3 & 0.9 & 7.8 & 501618.9 & 5 & {\bf 6.7} & 0.0 & -- & 59.8 \\ \hline
  
  100\_30\_5\_DD\_DF & 468289.8 & 5 & 836.2 & 0.0 & -- & 498460.1 & 506998.6 & 1.3 & 6.4 & 468289.8 & 5 & {\bf 10.9} & 0.0 & -- & 76.5 \\ 
  100\_30\_5\_DD\_SF & 495677.2 & 5 & 680.2 & 10.6 & -- & 523243.4 & 531805.5 & 1.3 & 5.6 & 495677.2 & 5 & {\bf 57.9} & 11.0 & -- & 71.9 \\ 
  100\_30\_5\_SD\_DF & 474804.8 & 5 & 1324.3 & 0.4 & -- & 507785.0 & 514767.0 & 1.3 & 6.9 & 474804.8 & 5 & {\bf 36.0} & 0.4 & -- & 73.2 \\ 
  100\_30\_5\_SD\_SF & 495174.5 & 5 & 1373.6 & 100.6 & -- & 519246.8 & 529159.4 & 1.3 & 4.8 & 495174.5 & 5 & {\bf 86.0} & 56.6 & -- & 63.1 \\ 
  100\_30\_10\_DD\_DF & 537879.5 & 5 & 17.6 & 0.0 & -- & 574666.3 & 581932.3 & 1.3 & 6.8 & 537879.5 & 5 & {\bf 10.7} & 0.0 & -- & 77.3 \\ 
  100\_30\_10\_DD\_SF & 585880.3 & 5 & 85.5 & 6.4 & -- & 624431.6 & 635490.7 & 1.3 & 6.6 & 585880.3 & 5 & {\bf 57.9} & 5.0 & -- & 71.7 \\ 
  100\_30\_10\_SD\_DF & 529729.1 & 5 & 70.0 & 2.2 & -- & 565166.5 & 573255.9 & 1.3 & 6.7 & 529729.1 & 5 & {\bf 46.6} & 1.8 & -- & 70.9 \\ 
  100\_30\_10\_SD\_SF & 568114.7 & 5 & 113.8 & 234.2 & -- & 604611.5 & 616566.5 & 1.3 & 6.4 & 568114.7 & 5 & {\bf 96.6} & 173.2 & -- & 61.9 \\ 
  100\_30\_15\_DD\_DF & 591972.7 & 5 & 128.0 & 1.4 & -- & 632603.3 & 639248.7 & 1.4 & 6.9 & 591972.7 & 5 & {\bf 62.5} & 1.4 & -- & 80.0 \\ 
  100\_30\_15\_DD\_SF & 628848.7 & 5 & {\bf 72.5} & 23.0 & -- & 670902.2 & 681860.7 & 1.4 & 6.7 & 628848.7 & 5 & 198.4 & 6.2 & -- & 72.1 \\ 
  100\_30\_15\_SD\_DF & 577253.8 & 5 & {\bf 74.8} & 0.8 & -- & 619048.3 & 625905.2 & 1.4 & 7.2 & 577253.8 & 5 & 80.1 & 1.0 & -- & 71.8 \\ 
  100\_30\_15\_SD\_SF & 614820.7 & 5 & 182.3 & 280.2 & -- & 657066.2 & 670451.3 & 1.4 & 6.9 & 614820.7 & 5 & {\bf 113.9} & 316.4 & -- & 62.5 \\ 
  100\_30\_20\_DD\_DF & 645732.6 & 5 & 53.2 & 1.0 & -- & 690564.3 & 699704.9 & 1.4 & 6.9 & 645732.6 & 5 & {\bf 22.4} & 0.6 & -- & 76.9 \\ 
  100\_30\_20\_DD\_SF & 682747.8 & 5 & 77.6 & 1.8 & -- & 733250.5 & 747422.6 & 1.4 & 7.4 & 682747.8 & 5 & {\bf 31.1} & 1.8 & -- & 71.0 \\ 
  100\_30\_20\_SD\_DF & 650369.9 & 5 & 60.5 & 1.4 & -- & 698496.7 & 706468.8 & 1.4 & 7.4 & 650369.9 & 5 & {\bf 44.3} & 0.6 & -- & 73.2 \\ 
  100\_30\_20\_SD\_SF & 682534.4 & 5 & 122.5 & 38.6 & -- & 737533.2 & 751841.5 & 1.4 & 8.1 & 682534.4 & 5 & {\bf 69.4} & 37.8 & -- & 60.8 \\ \hline
  
  200\_30\_5\_DD\_DF & 705795.2 & 5 & 99.2 & 1.2 & -- & 752181.2 & 759103.6 & 2.3 & 6.6 & 705795.2 & 5 & {\bf 73.6} & 0.8 & -- & 78.3 \\ 
  200\_30\_5\_DD\_SF & 719550.3 & 5 & 305.7 & 15.0 & -- & 770774.0 & 780736.1 & 2.3 & 7.1 & 719550.3 & 5 & {\bf 245.8} & 16.2 & -- & 71.3 \\ 
  200\_30\_5\_SD\_DF & 704916.6 & 5 & 79.1 & 2.8 & -- & 758430.7 & 764760.0 & 2.2 & 7.6 & 704916.6 & 5 & {\bf 60.9} & 2.0 & -- & 73.3 \\ 
  200\_30\_5\_SD\_SF & 735354.1 & 5 & 178.4 & 4.0 & -- & 764659.5 & 803936.7 & 2.3 & 4.0 & 735354.1 & 5 & {\bf 121.4} & 3.2 & -- & 60.9 \\ 
  200\_30\_10\_DD\_DF & 792983.7 & 5 & 2173.7 & 2.2 & -- & 856142.9 & 865450.7 & 2.4 & 8.0 & 792983.7 & 5 & {\bf 1345.3} & 2.2 & -- & 77.0 \\ 
  200\_30\_10\_DD\_SF & 842862.9 & 4 & 3042.4 & 4.8 & 0.36 & 906003.1 & 918060.8 & 2.3 & 7.5 & 842501.6 &{\bf 5} & 2081.1 & 14.0 & -- & 71.8 \\ 
  200\_30\_10\_SD\_DF & 1153974.7 & 4 & 2131.0 & 1.2 & 100.00 & 862047.0 & 872889.4 & 2.3 & 8.3 & 795863.1 &{\bf 5} & 1277.8 & 1.6 & -- & 74.2 \\ 
  200\_30\_10\_SD\_SF & 846511.7 & 5 & 2593.1 & 4.6 & -- & 890846.4 & 919510.2 & 2.3 & 5.2 & 846511.7 & 5 & {\bf 2109.5} & 7.4 & -- & 63.0 \\ 
  200\_30\_15\_DD\_DF & 855959.0 & 5 & 1044.1 & 6.8 & -- & 921752.6 & 930682.1 & 2.3 & 7.7 & 855959.0 & 5 & {\bf 111.3} & 6.4 & -- & 77.3 \\ 
  200\_30\_15\_DD\_SF & 888864.8 & 5 & 2352.6 & 37.4 & -- & 957131.4 & 970481.7 & 2.3 & 7.7 & 888864.8 & 5 & {\bf 810.0} & 25.4 & -- & 69.6 \\ 
  200\_30\_15\_SD\_DF & 865785.8 & 5 & 998.4 & 5.2 & -- & 935148.0 & 944406.6 & 2.3 & 8.0 & 865785.8 & 5 & {\bf 85.3} & 4.4 & -- & 73.6 \\ 
  200\_30\_15\_SD\_SF & 914463.8 & 5 & 2568.4 & 60.0 & -- & 967167.5 & 990341.7 & 2.4 & 5.8 & 914463.8 & 5 & {\bf 1176.3} & 57.8 & -- & 60.7 \\ 
  200\_30\_20\_DD\_DF & 930124.9 & 5 & 1249.0 & 17.0 & -- & 1003038.5 & 1012228.9 & 2.3 & 7.8 & 930124.9 & 5 & {\bf 174.2} & 29.0 & -- & 77.7 \\ 
  200\_30\_20\_DD\_SF & 985825.3 & 0 & -- & -- & 0.36 & 1061655.0 & 1078288.5 & 2.4 & 7.8 & 984800.1 &{\bf 5} & 2156.7 & 71.0 & -- & 71.6 \\ 
  200\_30\_20\_SD\_DF & 923774.7 & 5 & 1284.4 & 2.2 & -- & 1001824.2 & 1011784.4 & 2.3 & 8.4 & 923774.7 & 5 & {\bf 139.8} & 2.2 & -- & 73.7 \\ 
  200\_30\_20\_SD\_SF & 1002248.3 & 2 & 3005.5 & 52.5 & 0.64 & 1081590.9 & 1104916.9 & 2.4 & 8.0 & 1001898.5 &{\bf 5} & 2448.2 & 105.4 & -- & 63.1 \\ \hline
  Average &  &  & 610.9 & 20.4 & {25.34} &  &  & 1.5 & 6.7 &  &  & {\bf 325.3} & 20.6 & -- & {71.0} \\ 
  Total &  & 230 &  &  &  &  &  &  &  &  & {\bf 240} &  &  &  &  \\ 
   \hline
\end{tabular}
\end{table}

\end{landscape}

{

\subsection{Graphical summary}

A graphical summary of the obtained results is presented in Figures \ref{fig:optimal_std}-\ref{fig:cputime_mc}.
Figures~\ref{fig:optimal_std} and \ref{fig:cputime_std} depict for STD, 3LF, STD+, and 3LF+, respectively, the fraction (in percent) of the instances solved to optimality and the average running times (in seconds).
Figures~\ref{fig:optimal_mc} and \ref{fig:cputime_mc} exhibit for MC and DPH-pMC, correspondingly, the fraction (in percent) of the instances solved to optimality and the average running times (in seconds). For each instance, the considered time to calculate the averages is either the time to prove optimality (if less than the imposed time limit of 3600 seconds) or the time limit.

\begin{figure}[H]
\begin{center}
    \begin{tabular}{cc}
    \includegraphics[width=0.4\textwidth]{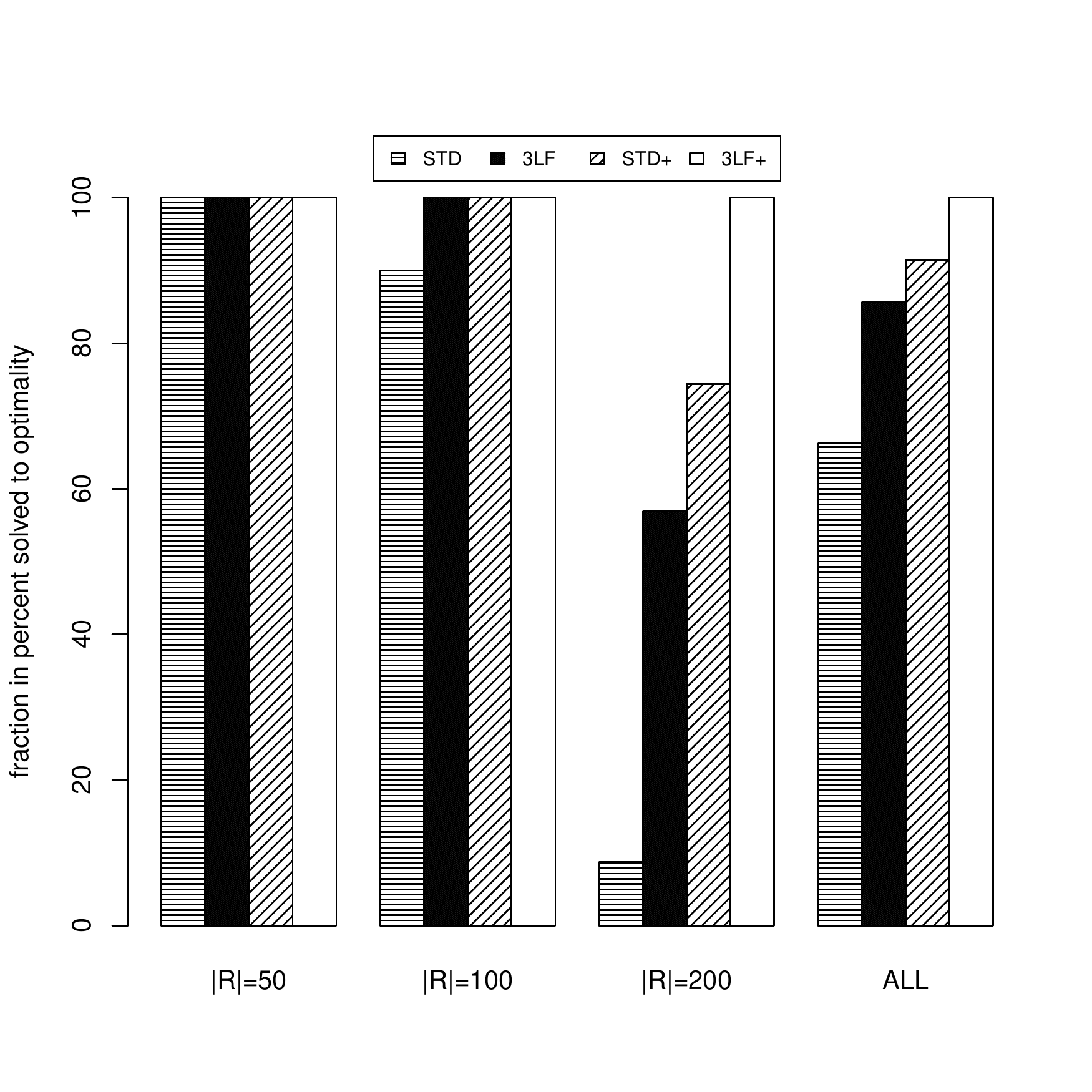} & \includegraphics[width=0.4\textwidth]{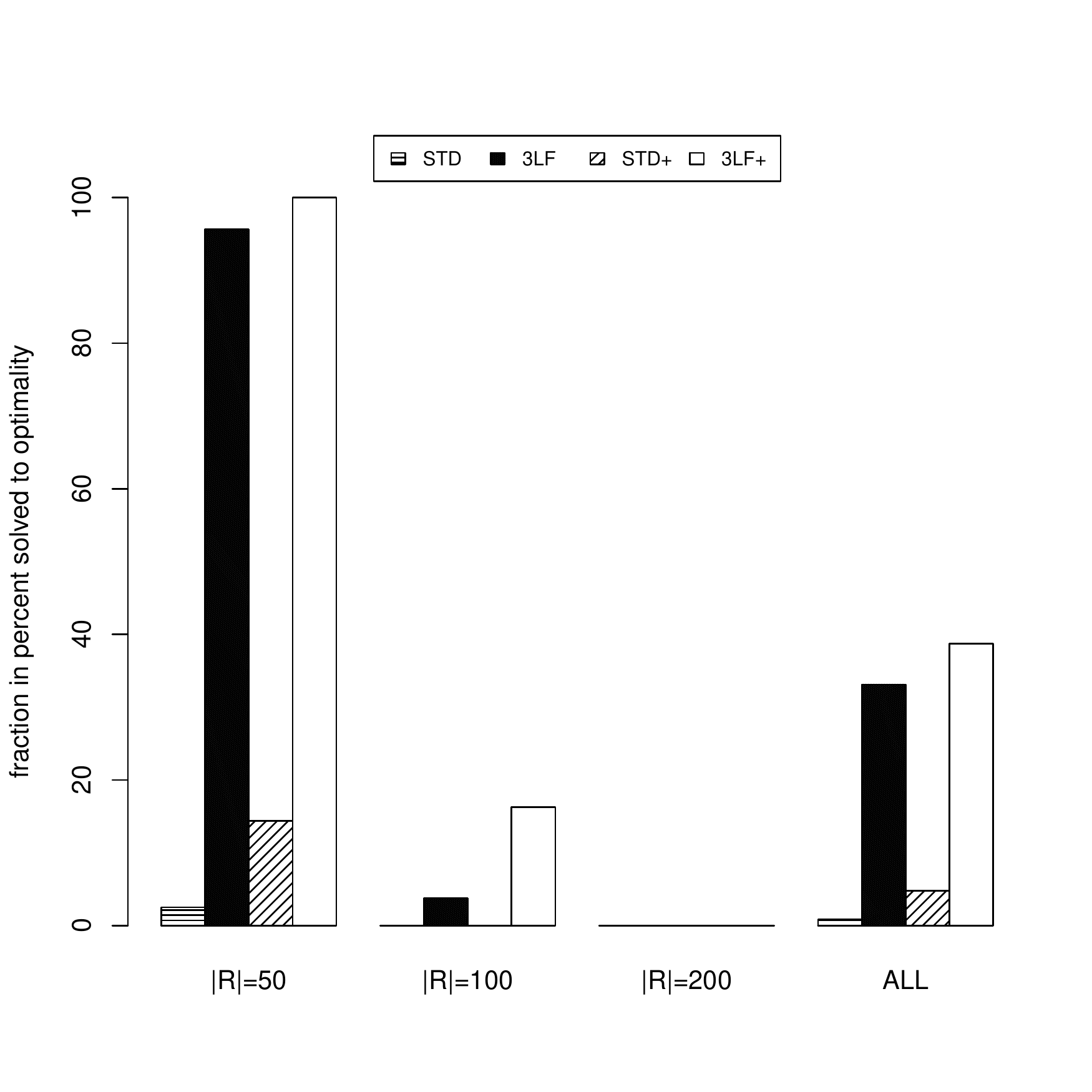}\\
(a) Instances with $|T|=15$. & (b) Instances with $|T|=30$.
\end{tabular}
\end{center}
\caption{The fraction of the instances (in percent) solved to optimality using STD, 3LF, STD+, and 3LF+ considering all the instances with each number of retailers.}
\label{fig:optimal_std}
\end{figure}

\begin{figure}[H]
\begin{center}
    \begin{tabular}{cc}
    \includegraphics[width=0.4\textwidth]{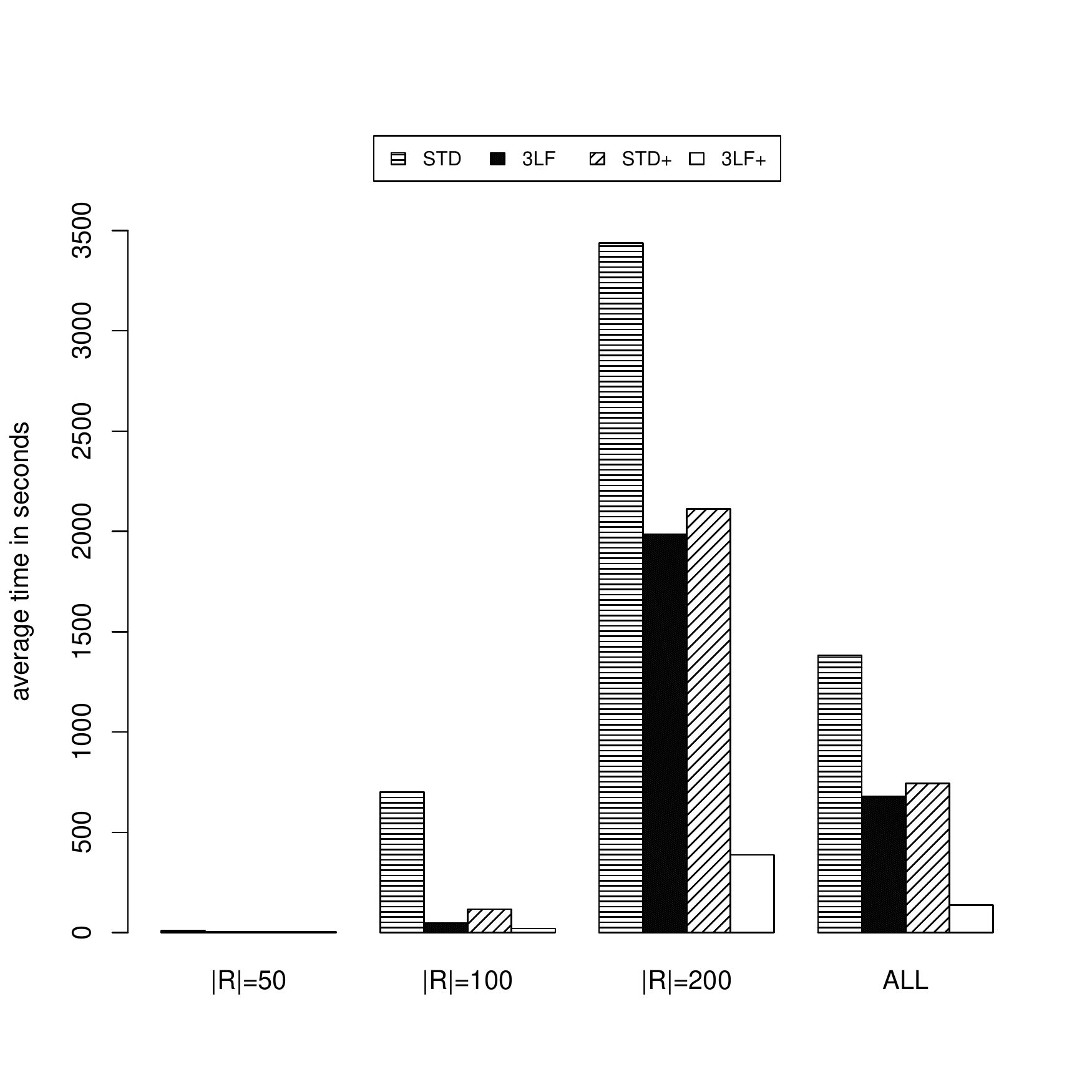} & \includegraphics[width=0.4\textwidth]{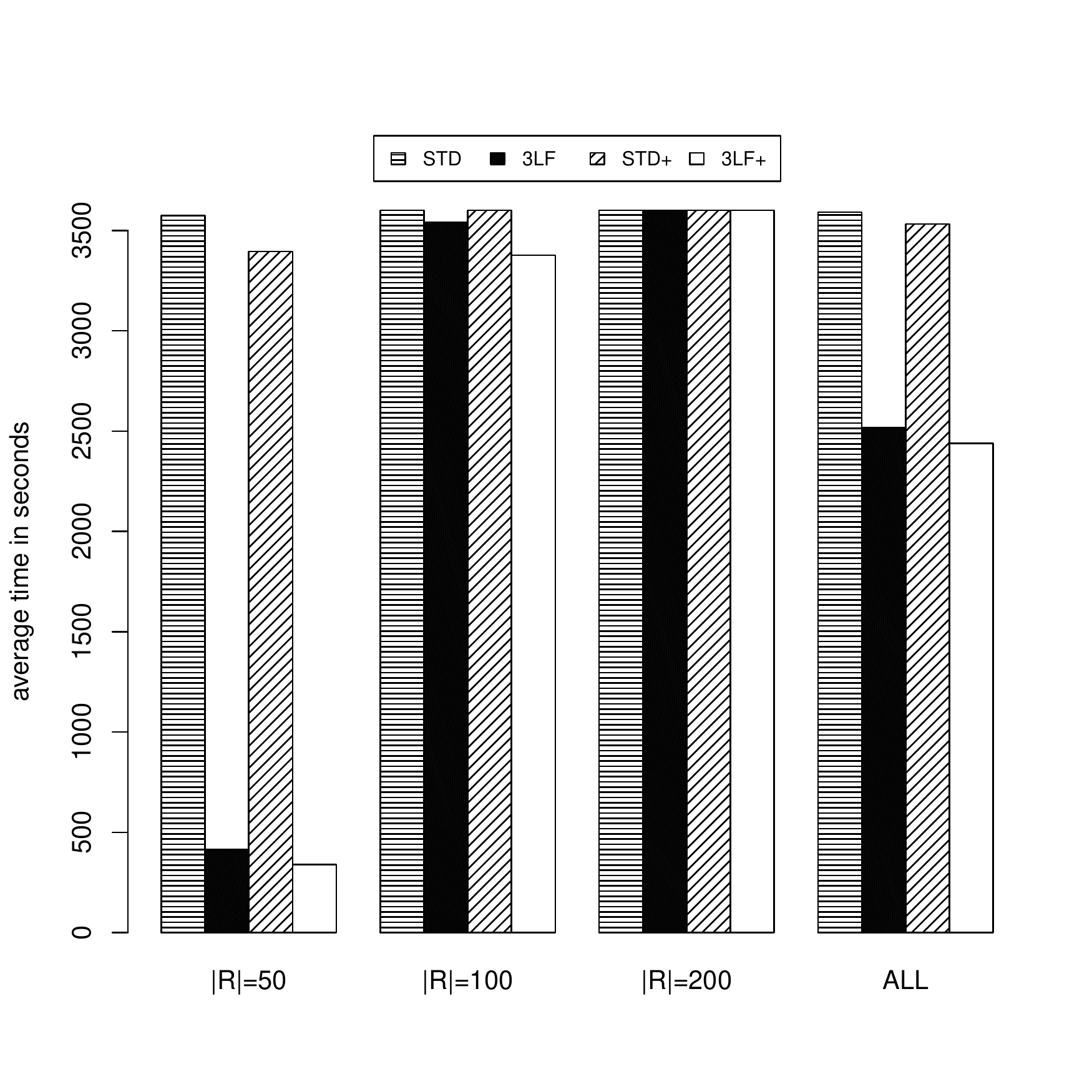}\\
(a) Instances with $|T|=15$. & (b) Instances with $|T|=30$.
\end{tabular}
\end{center}
\caption{Average computational times in seconds for STD, 3LF, STD+, and 3LF+ over all instances with each number of retailers. }
\label{fig:cputime_std}
\end{figure}

Figures~\ref{fig:optimal_std}-\ref{fig:cputime_std} show that 3LF+ outperforms STD, STD+, and 3LF+ for all the numbers of retailers and periods, but for the combination $|T|=30$ and $|R|=200$ (in this case, none of these approaches could solve a single instance to optimality). The figures also evidence the improvements achieved using the valid inequalities for the instances with $|T|=30$.

\begin{figure}[H]
\begin{center}
    \begin{tabular}{cc}
    \includegraphics[width=0.4\textwidth]{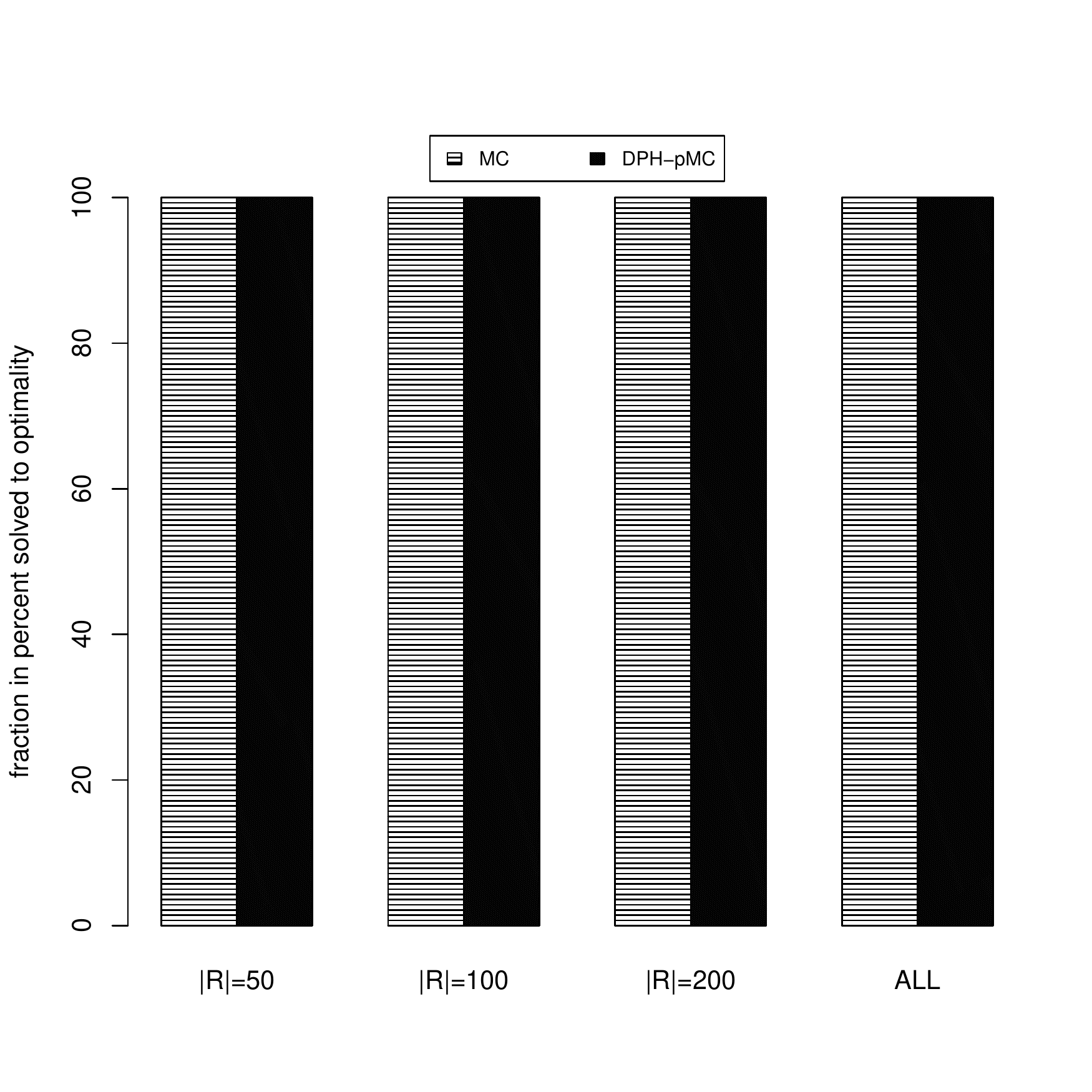} & \includegraphics[width=0.4\textwidth]{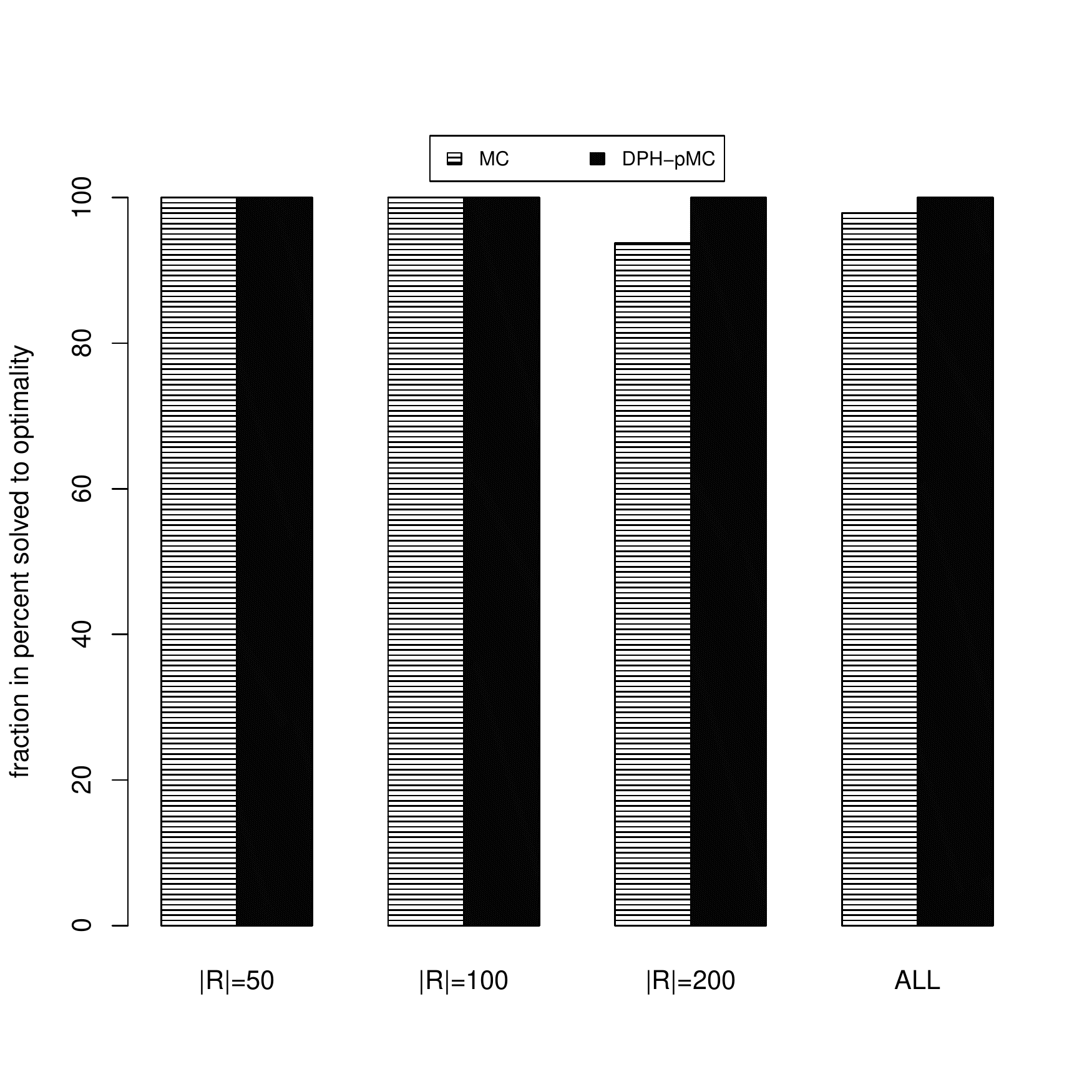}\\
(a) Instances with $|T|=15$. & (b) Instances with $|T|=30$.
\end{tabular}
\end{center}
\caption{The fraction of the instances (in percent) solved to optimality using MC and DPH-pMC considering all the instances with each number of retailers.}
\label{fig:optimal_mc}
\end{figure}

\begin{figure}[H]
\begin{center}
    \begin{tabular}{cc}
    \includegraphics[width=0.4\textwidth]{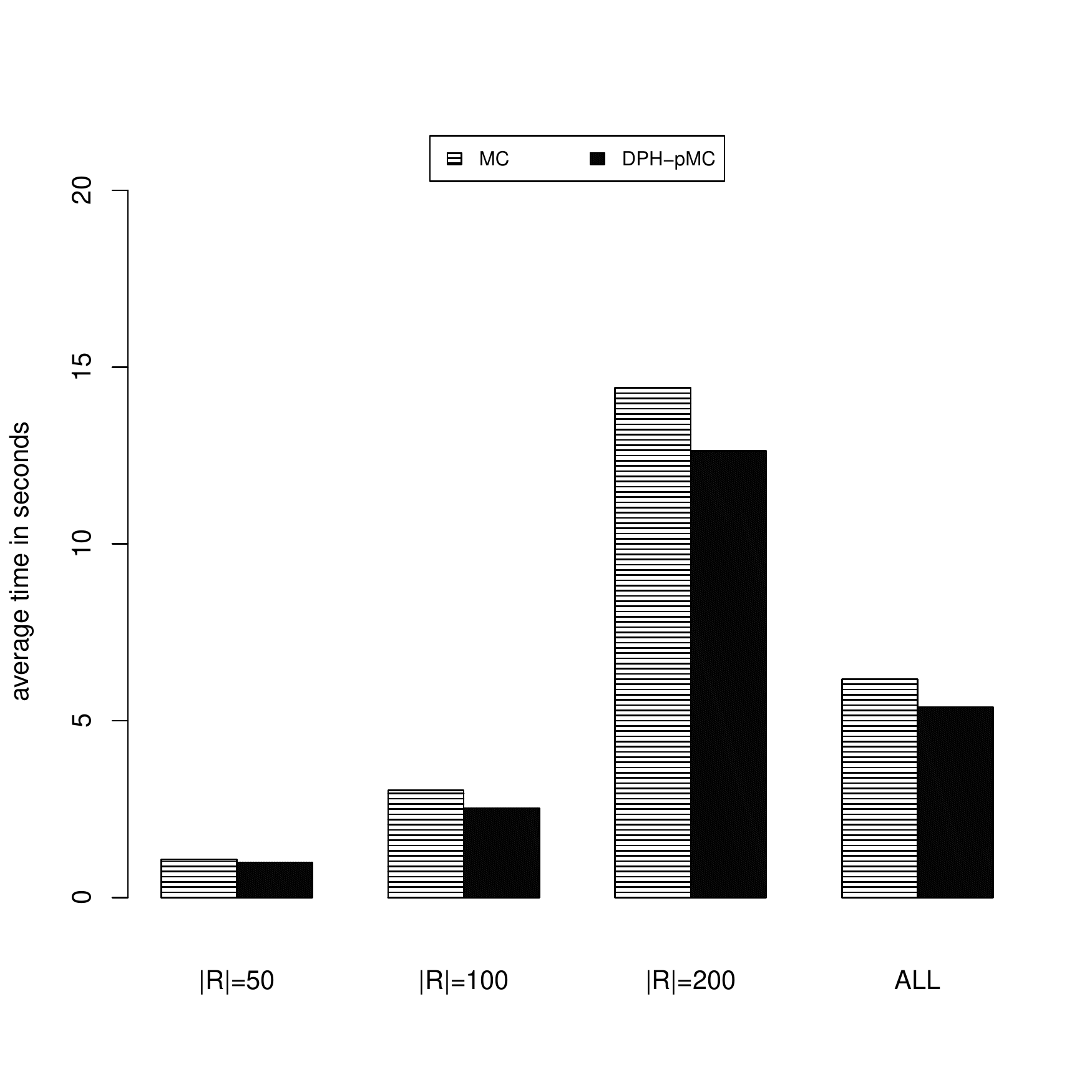} & \includegraphics[width=0.4\textwidth]{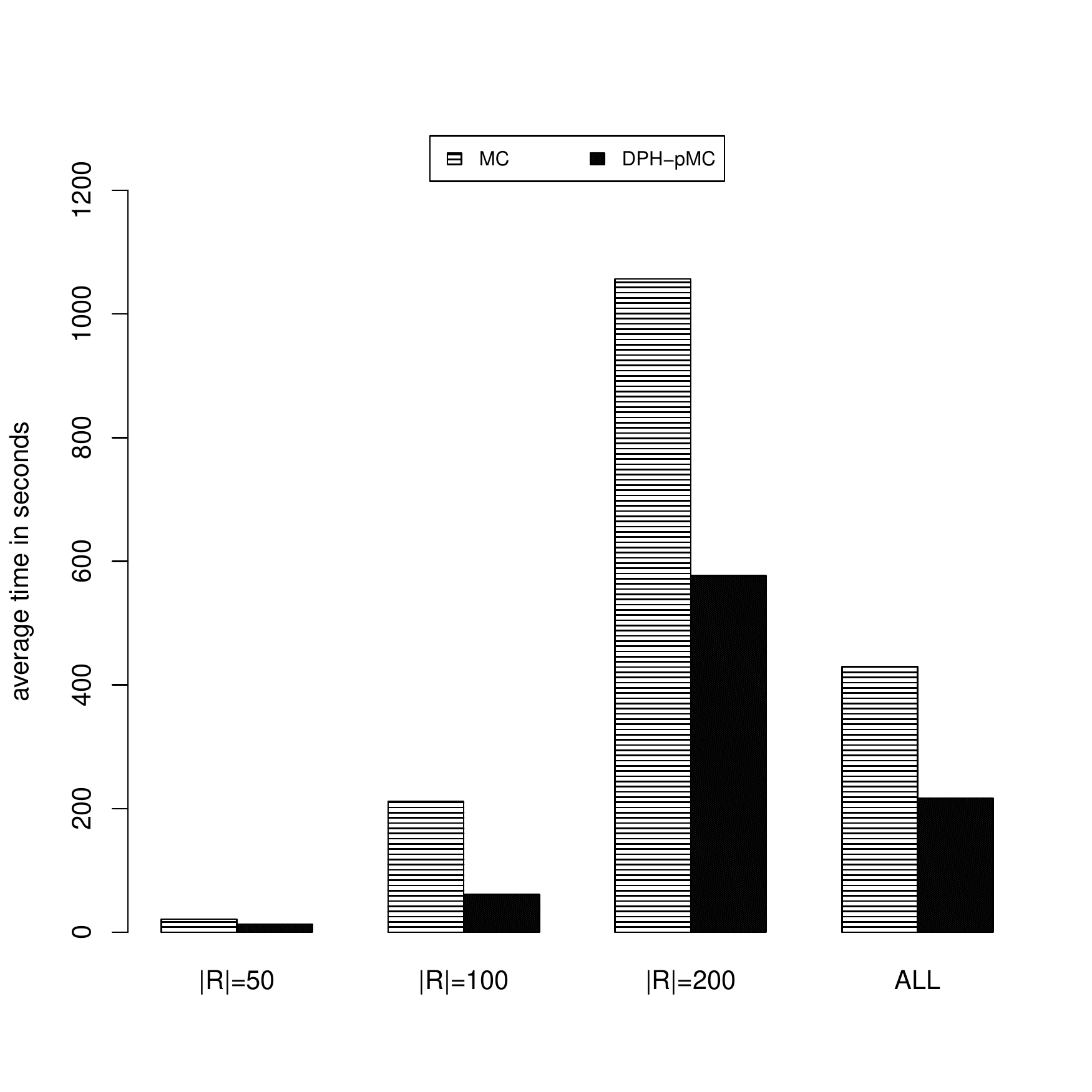}\\
(a) Instances with $|T|=15$. & (b) Instances with $|T|=30$.
\end{tabular}
\end{center}
\caption{Average computational times in seconds for MC and DPH-pMC over all instances with each number of retailers.}
\label{fig:cputime_mc}
\end{figure}

Figures~\ref{fig:optimal_mc}-\ref{fig:cputime_mc} evidence the improvements of DPH-pMC over MC. Firstly, note that nearly all instances were solved to optimality using both approaches, with an advantage for DPH-pMC considering the larger instances with $|T|=30$ and $|R|=200$. Figure~\ref{fig:cputime_mc} shows that the gains achieved by DPH-pMC in the average running times for the instances with $|T|=30$ is very remarkable.

}

\section{Concluding remarks}
\label{sec:concludingremarks}

In this work, we considered the NP-hard uncapacitated three-level lot-sizing and replenishment problem with a distribution structure (3LSPD-U). We proposed new valid inequalities, a preprocessing approach to reduce the size of an existing multi-commodity formulation for the problem, and a multi-start randomized bottom-up {dynamic programming-based heuristic}.

Computational experiments have shown that the valid inequalities can allow improvements in the bounds and, consequently, increase the solver's ability to solve instances to optimality when compared to the plain mixed integer programming formulations. 
Additionally, the use of valid inequalities using an extended formulation, which is larger but asymptotically equivalent to the standard one, can allow improvements in performance over the use of valid inequalities using the standard formulation.
The proposed multi-start randomized dynamic programming{-based} heuristic was able to encounter solutions with low optimality gaps (6.6\% on average) within very short computational times (1.2 seconds on average).
The presented preprocessing approach has shown to be effective in reducing the size of the multi-commodity formulation.  Combined, these two methods were able to increase the number of instances solved to optimality within the given time limit and to achieve significant reductions in the times spent to solve instances to optimality when compared to solely using the multi-commodity formulation.

\vspace{0.8cm}

{
\noindent \small 
\textbf{Acknowledgments:} Work of Rafael A. Melo was supported by the State of Bahia Research Foundation (FAPESB) and the Brazilian National Council for Scientific and Technological Development (CNPq). The authors are thankful to Matthieu Gruson and Raf Jans for kindly providing the benchmark instances.
The authors would like to thank three anonymous reviewers for the comments which helped to improve this paper.
}

\bibliographystyle{apacite}

\bibliography{main}
 
\end{document}